\documentclass[sn-mathphys,Numbered]{sn-jnl}

\usepackage{graphicx}%
\usepackage{multirow}%
\usepackage{amsmath,amssymb,amsfonts}%
\usepackage{amsthm}%
\usepackage{mathrsfs}%
\usepackage[title]{appendix}%
\usepackage{xcolor}%
\usepackage{textcomp}%
\usepackage{manyfoot}%
\usepackage{booktabs}%
\usepackage{algorithm}%
\usepackage{algorithmicx}%
\usepackage{algpseudocode}%
\usepackage{listings}%

\usepackage{mathtools} 
\mathtoolsset{showonlyrefs=true}

\theoremstyle{thmstyleone}%
\newtheorem{theorem}{Theorem}
\newtheorem{proposition}[theorem]{Proposition}%

\theoremstyle{thmstyletwo}%
\newtheorem{remark}{Remark}%

\theoremstyle{thmstylethree}%
\newtheorem{assumption}{Assumption}
\newtheorem{corollary}{Corollary}
\newtheorem{lemma}{Lemma}
\newcommand{\tr}{\mathrm{tr}\,}
\newcommand{\pa}{\partial}
\newcommand{\e}{\mathrm{e}}
\newcommand{\I}{\mathrm{i}}
\newcommand{\Z}{\mathbb{Z}}
\newcommand{\R}{\mathbb{R}}
\newcommand{\bs}[1]{\boldsymbol{#1}}

\begin{document}

\title[A new instability framework in 2-component reaction-diffusion systems]{A new instability framework in 2-component reaction-diffusion systems}


\author[1]{\fnm{Hirofumi} \sur{Izuhara}}\email{izuhara@cc.miyazaki-u.ac.jp}

\author*[1,2]{\fnm{Shunusuke} \sur{Kobayashi}}\email{s.kobayashi@cc.miyazaki-u.ac.jp}


\affil[1]{Faculty of Engineering, University of Miyazaki, 1-1 Gakuen Kibanadainishi Miyazaki, 889--2192, Japan}

\affil[2]{RIKEN iTHEMS, 2-1 Hirosawa, Wako-shi, Saitama 351--0198, Japan}



\abstract{This paper concerns pattern formation in 2-component reaction-diffusion systems with linear diffusion terms and a local interaction. 
We propose a new instability framework with 0-mode Hopf instability, $\textit{m}$ and $\textit{m}$ + 1 mode Turing instabilities in 2-component reaction-diffusion systems. 
The normal form for the codimension 3 bifurcation is derived via the center manifold reduction, which is one of the main results in the present paper. 
We also show numerical results on bifurcation of some reaction-diffusion systems and on a chaotic behavior of the normal form.}




\maketitle

\section{Introduction}\label{sec:intro}
Reaction-diffusion systems have been widely studied to understand self-organized pattern formation phenomena arising in nature, particularly biology, chemistry and physics (for instance \cite{M1982,M2002,KM}). 
The studies on pattern formation using reaction-diffusion systems are still now spreading further. 
The reason why reaction-diffusion systems attract the researchers is probably because they include a plain pattern forming mechanism called the diffusion-induced instability (Turing instability). 
As we briefly explain later, reaction-diffusion systems which are contained in a kind of framework provide that a suitable difference of diffusion coefficients can cause the destabilization of a spatially homogeneous steady state. 
This may look like a paradox since it seems that the diffusion effect promotes spatial homogeneity. 
Consequently, spatial structures called patterns appear after the instability of the spatially homogeneous state. 
Theoretically, the patterns can be captured as a bifurcation from the spatially homogeneous steady state. 
The Turing instability is typically explained for 2-component reaction-diffusion systems, 
but this idea can be systematically extended into 3 or more component reaction-diffusion systems (\cite{ASY}). 
Since $n$ components ($n\geq 3$) can represent a variety of reaction terms, 
we can expect that the dynamics of $n$-component reaction-diffusion systems is richer than that of 2-component systems from the pattern formation point of view (for instance \cite{YE}). 

Recently, a relation between Turing instability and other instabilities is being revealed. 
In \cite{IMN,FIMU}, it is reported that the cross-diffusion induced instability, which is one of the instability mechanisms, can be regarded as the Turing instability. 
In other words, this means that a kind of nonlinear diffusion called cross-diffusion can be approximated by reaction-diffusion systems. 
Besides, it is also known that nonlocal dispersals which are described by integral terms can be approximated by reaction-diffusion systems (\cite{NTY}). 
The reaction-diffusion system approximation can be also possible for free boundary problems (\cite{IMW}). 
For the interested readers in reaction-diffusion system approximation, see the review paper \cite{INY}. 

In this paper, we consider pattern formation in simple 2-component reaction-diffusion systems which mean to possess ``linear'' diffusion terms and ``local'' interactions only, that is, the form in one-space dimension is as follows: 
\[
\begin{aligned}
\partial_t u&=D_u \partial_{xx} u+f(u,v), &\quad t>0, \, 0<x<L, \\
\partial_t v&=D_v \partial_{xx} v+g(u,v), &\quad t>0, \, 0<x<L, 
\end{aligned}
\]
where $u$ and $v$ are certain quantities, the parameters $D_u$ and $D_v$ are diffusion coefficients, and the functions $f$ and $g$ describe a local interaction between $u$ and $v$. 
We consider this system under the zero flux boundary conditions and one-space dimension throughout this paper. 
For this system, many mathematical results have been obtained from the viewpoint of Turing instability and application of the bifurcation theory. 
A seminal and classical result is a bifurcation from a simple eigenvalue (\cite{CR}). 
Applying this result to the above reaction-diffusion systems, we can find the existence and stability of non-constant stationary solutions which bifurcate from a constant steady state.
As we shall see below, when the two diffusion coefficients vary suitably, the 2-component reaction-diffusion system can possess a doubly degenerate point where two real eigenvalues simultaneously become zero. 
In the vicinity of the doubly degenerate point, 
bifurcation structures, which are obtained applying the Lyapunov-Schmidt reduction, are thoroughly discussed in \cite{FMN}. 
The derivation of a finite dimensional dynamical system around the doubly degenerate point by using the center manifold reduction is also done in \cite{AGH,IK}. 
Another codimension $2$ bifurcation observed in the 2-component systems is a bifurcation with single zero eigenvalue and a pair of purely imaginary eigenvalues (\cite{GH,DIR}). This type of bifurcation is often called the Hopf--Turing bifurcation. 
In addition, it is known that, out of scope here, 3-component reaction-diffusion systems have more complicated dynamics such as a codimension $3$ bifurcation with $0$:$1$:$2$ mode interaction and wave bifurcation (\cite{O,OO,KS}). 

As addressed above, there are several types of bifurcations in simple 2-component reaction-diffusion systems. 
Here we arise a question: is there any other complicated bifurcation type in simple 2-component reaction-diffusion systems? 
Can the systems exhibit more complicated bifurcation than codimension 2 bifurcation? 
In this paper, we propose a new framework for a codimension 3 bifurcation generating spatio-temporal complicated dynamics in simple 2-component reaction-diffusion systems. 
To our best knowledge, this framefork is universal in simple 2-component systems with the Turing instability, and shows the most complicated bifurcation in the simple systems, which is presented in detail in the next section. 
A key idea to generate a codimension 3 bifurcation in 2-component reaction-diffusion systems is to destabilize a homogeneous steady state in the sense of Hopf instability. 
Therefore, as a by-product of this idea, 
we can capture a bifurcation from a spatially homogeneous time periodic solution. 
This study can be regarded as a time periodic solution version of Turing instability. 
Instability from spatially homogeneous time periodic solution has been also considered. 
When a stable periodic solution exists in the sense of ODEs, 
Kuramoto and Maginu separately discuss the stability of the time periodic solution under the addition of diffusion terms (\cite{K,M}).
But their argument is done in the one-dimensional whole space. 
Similar instability on a finite interval under the zero flux boundary conditions is discussed in \cite{RM,KI}.




The structure of this paper is as follows: 
In the next section, we introduce our instability framework which is discussed in this paper. One can see that simple 2-component reaction-diffusion systems with the zero flux boundary conditions can typically generate a codimension 3 bifurcation with $0$-mode Hopf instability, $m$ and $m+1$-mode Turing instabilities. 
Therefore, this bifurcation may be referred to as Hopf--Turing--Turing bifurcation in the present paper. 
In Section~\ref{Numerics}, we show some numerical bifurcation results around the codimension 3 bifurcation point with an aid of a numerical bifurcation Matlab package pde2path. In particular, we can find bifurcation structures from a spatially homogeneous time periodic solution. 
Section~\ref{CMR} is devoted to the rigorous analysis of the bifurcation. We derive a normal form of the codimension 3 bifurcation via the center manifold theory, and investigate possible bifurcation structures in the vicinity of the bifurcation point. This analysis extends the previous studies, and potentially suggests the existence of chaotic solutions. 
Finally, we give concluding remarks in Section~\ref{sec:Conclude}.


\section{A framework for Hopf--Turing--Turing bifurcation}\label{HTT}
We first briefly explain the Turing instability in Section~\ref{T-insta}, and then introduce our instability framework for simple 2-component reaction-diffusion systems in Section~\ref{H-insta}.  

\subsection{Turing instability}\label{T-insta}
We briefly explain the diffusion induced instability. 
We first consider the following system of ordinary differential equations: 
\[
\begin{aligned}
u'&=f(u,v),\\
v'&=g(u,v),
\end{aligned}
\]
where the functions $f$ and $g$ are smooth. 
In this system, we assume that there is an equilibrium $(\tilde{u}, \tilde{v})$ and it is linearly stable. Namely, the linearized matrix
\[
M:=
\begin{pmatrix}
f_u(\tilde{u}, \tilde{v}) & f_v(\tilde{u}, \tilde{v})\\
g_u(\tilde{u}, \tilde{v}) & g_v(\tilde{u}, \tilde{v})
\end{pmatrix}
\]
possesses two eigenvalues with negative real part, thus $\det M = f_u g_v-f_v g_u>0$ and $\tr M = f_u+g_v<0$. 
Here and hereafter, we put $f_u = \pa f (\tilde u, \tilde v)/\pa u$ and so forth.
Under this setting, we add the diffusion terms for each equation as follows:
\begin{equation}
\label{2rd-original}
\begin{aligned}
\partial_t u&=D_u \partial_{xx} u+f(u,v) & \text{for }t>0, \, 0<x<L,\\
\partial_t v&=D_v \partial_{xx} v+g(u,v) & \text{for }t>0, \, 0<x<L,
\end{aligned}
\end{equation}
When the zero flux boundary conditions 
\begin{equation}\label{BC}
\partial_x u(t,0)=\partial_x u(t,L)=\partial_x v(t,0)=\partial_x v(t,L)=0\quad \text{for }t>0
\end{equation}
are imposed, the linear stability analysis shows that 
the constant stationary solution $(\tilde{u}, \tilde{v})$ can be destabilized. Indeed, the linearized operator around $(\tilde{u}, \tilde{v})$ is given by 
\[
\begin{pmatrix}
D_u\partial_{xx}+f_u & f_v\\
g_u & D_v \partial_{xx}+g_v
\end{pmatrix}
\]
and for each Fourier cosine mode $n$, 
\[
\tilde{M}_n:=
\begin{pmatrix}
-D_u\left(\frac{n\pi}{L}\right)^2+f_u & f_v\\
g_u & -D_v\left(\frac{n\pi}{L}\right)^2+g_v
\end{pmatrix}
\]
is obtained. Here, obviously
\[
\text{tr}\tilde{M}_n=-(D_u+D_v)\left(\frac{n\pi}{L}\right)^2+f_u+g_v<0,
\]
on the other hand, 
\[
\text{det}\tilde{M}_n=\left( D_u\left(\frac{n\pi}{L}\right)^2-f_u \right)\left( D_v\left(\frac{n\pi}{L}\right)^2-g_v \right)-f_v g_u
\]
can be negative if $D_v>D_u$, $f_u>0>g_v$, and $f_v$ and $g_u$ have different signs. 
If $f_u>0>g_v$ and $g_u>0>f_v$, such reaction-diffusion systems are referred to as an activator-inhibitor system, on the other hand, 
if $f_u>0>g_v$ and $f_v>0>g_u$, they are called a substrate-depleted system. 
Fig.~\ref{NSC} shows neutral stability curves for a reaction-diffusion system with 
\(
M=
\begin{pmatrix}
\frac{9}{11} & \frac{121}{100}\\
-\frac{20}{11} & -\frac{121}{100}
\end{pmatrix}
\), where $\text{tr}M=-\frac{431}{1100}<0$ and $\text{det}M=\frac{121}{100}>0$. 
\begin{figure}[htbp]
\begin{center}
\includegraphics[width=80mm]{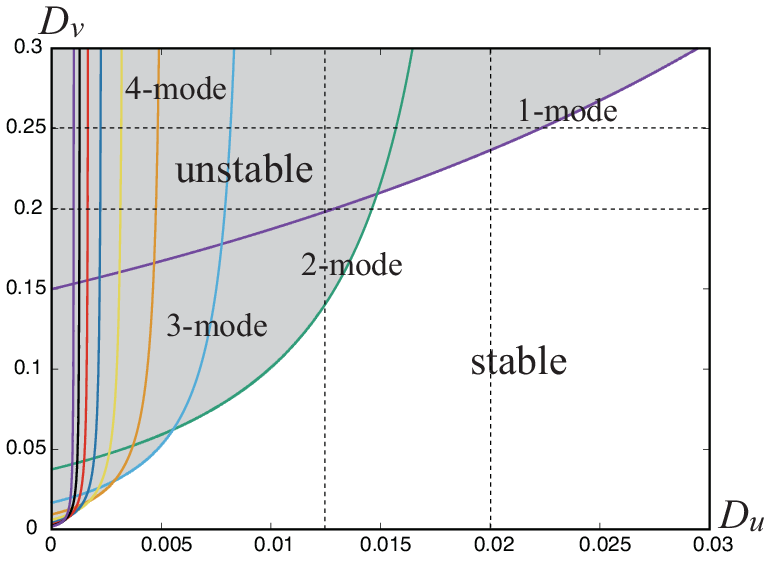}
\caption{Neutral stability curves for a reaction-diffusion system. The gray color means unstable region of $(u,v)=(0,0)$. The horizontal axis and the vertical one mean $D_u$ and $D_v$, respectively. }
\label{NSC}
\end{center}
\end{figure}
This figure implies that the diffusion-induced instability occurs by choosing suitable values of $D_u$ and $D_v$. 
Each curve corresponds to ${\rm det}\tilde{M}_n=0$ ($n=1,2,3,\cdots$). 
Moreover, doubly degenerate points where two distinct modes are simultaneously destabilized are observed. 
At these points, we can see a codimension 2 bifurcation. 
Our question in this paper is as follows: 
is there any other onset of complicated dynamics in 2-component simple reaction-diffusion systems which mean that the systems only possess ``linear'' diffusion and ``local'' reaction interaction on $f$ and $g$?  
Can we find an organizing center which generates complicated dynamics in simple reaction-diffusion systems?
In the present paper, we will answer this question and derive a finite dimensional dynamical system via the center manifold theory, which may show complicated behaviors of solutions such as a chaotic behavior.

\subsection{Hopf instability of $0$-mode}\label{H-insta}

In addition to the doubly degenerate points, the Hopf instability of 0-Fourier mode can arise. 
Therefore, 2-component reaction-diffusion systems can exhibit a simultaneously multiple bifurcation with $0$-mode Hopf, $m$-mode Turing and $m+1$-mode Turing instabilities. 

To give rise to the Hopf instability of $0$-mode, we introduce a positive parameter $\alpha$ into the reaction-diffusion system \eqref{2rd-original} as follows: 
\begin{equation}
\label{2rd}
\begin{aligned}
\partial_t u&=D_u \partial_{xx} u+f(u,v) & \text{for }t>0, \, 0<x<L,\\
\partial_t v&=\alpha\left( D_v \partial_{xx} v+g(u,v) \right) & \text{for }t>0, \, 0<x<L,
\end{aligned}
\end{equation}
where the parameter $\alpha$ can also be regarded as time constant. Under the zero flux boundary conditions \eqref{BC}, to unveil the effect of the parameter $\alpha$, we linearize \eqref{2rd} around the constant solution $(\tilde{u},\tilde{v})$, similarly to the previous section. 
Then, the linearized operator is provided as 
\[
\begin{pmatrix}
D_u\partial_{xx}+f_u & f_v\\
\alpha g_u & \alpha\left( D_v \partial_{xx}+g_v\right)
\end{pmatrix}
\]
and for each Fourier cosine mode $n$, we have 
\[
M_{n,\alpha}:=
\begin{pmatrix}
-D_u\left(\frac{n\pi}{L}\right)^2+f_u & f_v\\
\alpha g_u & \alpha\left(-D_v\left(\frac{n\pi}{L}\right)^2+g_v\right)
\end{pmatrix}.
\]
Here, we can see that for any integer $n\geq0$
\[
\begin{aligned}
\tr M_{n,\alpha} &= -(D_u+\alpha D_v)\left(\frac{n\pi}{L}\right)^2+f_u+\alpha g_v,\\
\det M_{n,\alpha} &=\alpha \text{det}\tilde{M}_{n}. 
\end{aligned}
\]
Note that for any $\alpha>0$, ${\rm det}M_{n,\alpha}$ and ${\rm det}\tilde{M}_{n}$ are the same sign. 
On the other hand, when we focus on $n=0$, the trace 
\(
\tr M_{0,\alpha} = f_u+\alpha g_v
\)
can be set as zero by adjusting the value $\alpha$ since $f_u$ and $g_v$ have different signs. 
Since $\det M_{0,\alpha}=\alpha \det \tilde{M}_{0}>0$, these imply that $0$-mode can be destabilized by virtue of the Hopf instability. 
Actually, in the previous example 
\(
M_{0,\alpha}=
\begin{pmatrix}
\frac{9}{11} & \frac{121}{100}\\
-\frac{20}{11}\alpha & -\frac{121}{100}\alpha
\end{pmatrix}
\), 
since $\tr M_{0,\alpha}=\frac{9}{11}-\frac{121}{100}\alpha$, 
we expect that 0-mode can be destabilized at $\alpha=\frac{900}{1331}\approx 0.67618332$, therefore an oscillatory behavior of uniform state may be exhibited if $\alpha$ is smaller than the value. 
Here, we note that this $0$-mode Hopf instability is independent of the Turing instability mentioned above. 
Therefore, for instance, we can select a triplet $(D_u^*,D_v^*,\alpha^*)$ satisfying 
$\tr M_{0,\alpha^*}=0$, $\det M_{0,\alpha^*}>0$, 
$\tr M_{1,\alpha^*}<0$, $\det M_{1,\alpha^*}=0$, 
$\tr M_{2,\alpha^*}<0$, $\det M_{2,\alpha^*}=0$, 
where the triplet is a triply degenerate point in the sense of $0$-mode Hopf, $1$-mode Turing and $2$-mode Turing instabilities. 
In the next section, we show some numerical bifurcation diagrams in the vicinity of the triply degenerate point for several 2-component reaction-diffusion systems.

\section{Numerical bifurcation analysis}\label{Numerics}

In this section, we numerically compute bifurcation diagrams for 2-component reaction-diffusion systems with an aid of pde2path (\cite{U-2019,U-2021,U-2022,UWR}) which is a numerical Matlab package for bifurcation analysis. 
In any cases, we set $L=1$ without loss of generality and impose the zero flux boundary conditions $u_x=v_x=0$ at $x=0$ and $1$. 

We here focus on bifurcations from a time periodic solution with spatial homogeneity. 
In each bifurcation diagram, we use the following norm for a solution $u$ in the vertical axis: 
\[
\|u\|_*=
\begin{cases}
\frac{1}{\sqrt{|L|}}\|u\|_{L^2(0,L)} & \text{for stationary solutions},\\
\frac{1}{\sqrt{T|L|}}\|u\|_{L^2((0,L)\times(0,T))} & \text{for time periodic solutions with period $T$}.
\end{cases}
\]
In addition, we use the descriptions for bifurcation diagrams listed in Table~\ref{description} throughout this section. 

\begin{table}[h]
\caption{Descriptions for bifurcation diagrams.}\label{description}
\centering
\begin{tabular}{cl}
\hline
thick curve & stable branch\\
thin curve & unstable branch\\
cyan colored curve & constant stationary solution branch\\
blue colored curve & non-constant stationary solution branch\\
magenta colored curve &  spatially homogeneous time periodic solution branch\\
red colored curve & spatially inhomogeneous time periodic solution branch\\
\hline
\end{tabular}
\end{table}

We illustrate bifurcation diagrams of three types of 2-component reaction-diffusion systems, say the Schnakenberg model, the Mimura--Murray model and an artificial system which does not possess meanings as a mathematical model.

\subsection{Schnakenberg model}\label{Schnakenberg model}

We first consider the following Schnakenberg model with a parameter $\alpha$: 
\begin{equation}\label{Schnak}
\begin{aligned}
\partial_t u&=D_u \partial_{xx} u+A-u+u^2v,\\
\partial_t v&=\alpha \left( D_v \partial_{xx} v+B-u^2v \right).
\end{aligned}
\end{equation}
The parameter values are set as $A=0.1$ and $B=1.0$ in this subsection. 
When $\alpha=1$, \eqref{Schnak} is reduced to the well known Schnakenberg model. 
This system possesses a constant stationary solution $(u,v) = \left( A+B, \frac{B}{(A+B)^2} \right)$. 
In order to investigate the stability of the constant stationary solution, we linearize \eqref{Schnak} around the constant stationary solution and use the Fourier cosine expansion as discussed above.
As a result, neutral stability curves for \eqref{Schnak} are shown in the $(D_u,D_v)$-plane as in Fig.~\ref{NSC}. 
Besides, a simple calculation provides that Hopf instability of $0$-mode occurs if the parameter $\alpha$ satisfies $\frac{B-A}{A+B}-\alpha (A+B)^2=0$ which is denoted by $\alpha^*$, that is $\alpha^* \approx 0.676183\cdots$. 
If $\alpha>\alpha^*$, then $0$-mode is stable, on the other hand, if $\alpha<\alpha^*$, it is oscillatorily unstable. 

Based on the neutral stability curves in Fig.~\ref{NSC}, we numerically compute bifurcation diagrams when either $D_u$ or $D_v$ changes as a bifurcation parameter but $\alpha$ is fixed. 
Fig.~\ref{BD-1} shows a bifurcation diagram when $\alpha=0.63$ and $D_u=0.02$, where $D_v$ acts as a bifurcation parameter.  
\begin{figure}[htbp]
\begin{center}
\includegraphics[width=80mm]{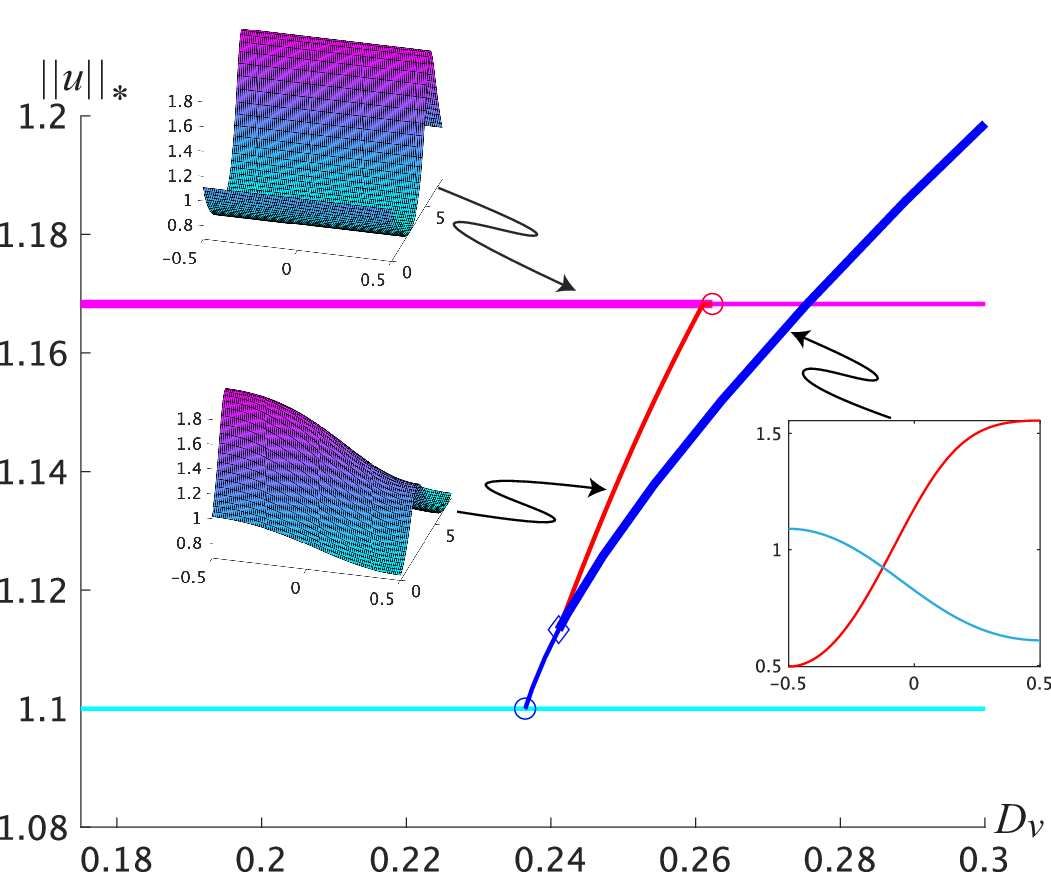}
\caption{Bifurcation diagram for \eqref{Schnak} when $\alpha=0.63$ and $D_u=0.02$. 
The horizontal axis and the vertical axis are respectively the bifurcation parameter $D_v$ and the norm $\|u\|_*$. 
The time periodic solutions in the figure indicate the profile $u$, and the stationary solution exhibits the solution profile at $D_v=0.275258$, where the red curve and the blue curve represent $u$ and $v$. }
\label{BD-1}
\end{center}
\end{figure}
\begin{figure}[htbp]
\begin{center}
\includegraphics[width=110mm]{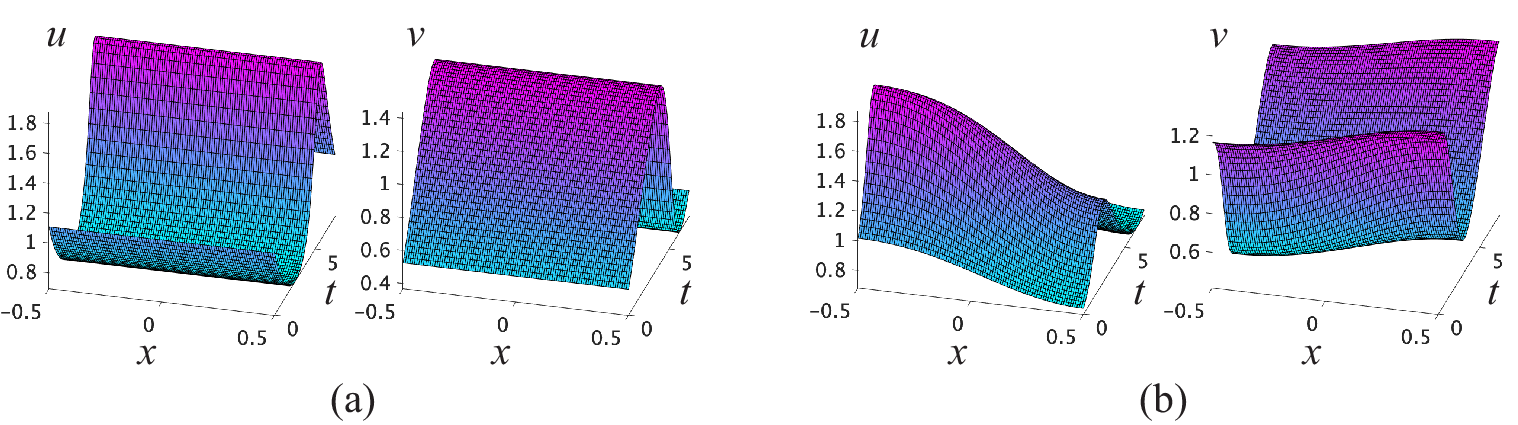}
\caption{(a) Spatially homogeneous time periodic solution with period $T=7.47676$. (b) Spatially inhomogeneous time periodic solution at $D_v=0.25044$. The period is $T=7.35746$.}
\label{profile-1}
\end{center}
\end{figure}
Since $\alpha=0.63$ in this case, the constant stationary solution $\left( A+B, \frac{B}{(A+B)^2} \right)$ (the thin cyan colored line) is unstable, thus we expect the existence of a spatially homogeneous time periodic solution. 
Actually, it exists (magenta colored line. See also Fig.~\ref{profile-1}(a).) and changes the stability according to the value of $D_v$, namely 
it is stable for $D_v<0.262314$ and unstable for $D_v>0.262314$. 
We can see that the time periodic solution is destabilized at $D_v=0.262314$ (marked with a circle) and a spatially inhomogeneous time periodic solution branch appears subcritically (the red colored curve). A typical profile of solution on this branch is exhibited in Fig.~\ref{profile-1}(b). 
This time periodic solution branch can also be regarded as a branch bifurcating from a Hopf bifurcation point at $D_v=0.241068$ (marked with a diamond) since this time periodic solution branch connects with a non-constant stationary solution branch (blue colored curve) which bifurcates from the constant stationary solution branch (cyan colored curve). 
Focusing on stable solutions in Fig.~\ref{BD-1}, we can see the spatially homogeneous time periodic solution for $D_v<0.262314$ and the non-constant stationary solutions for $D_v>0.241068$. 
In this parameter regime, there are spatially inhomogeneous time periodic solution, but they are unstable.


Fig.~\ref{BD-2} shows a bifurcation diagram of \eqref{Schnak} when $D_u=0.0125$ which is smaller than that of Fig.~\ref{BD-1} and the other parameters keep the same values.  
\begin{figure}[htbp]
\begin{center}
\includegraphics[width=80mm]{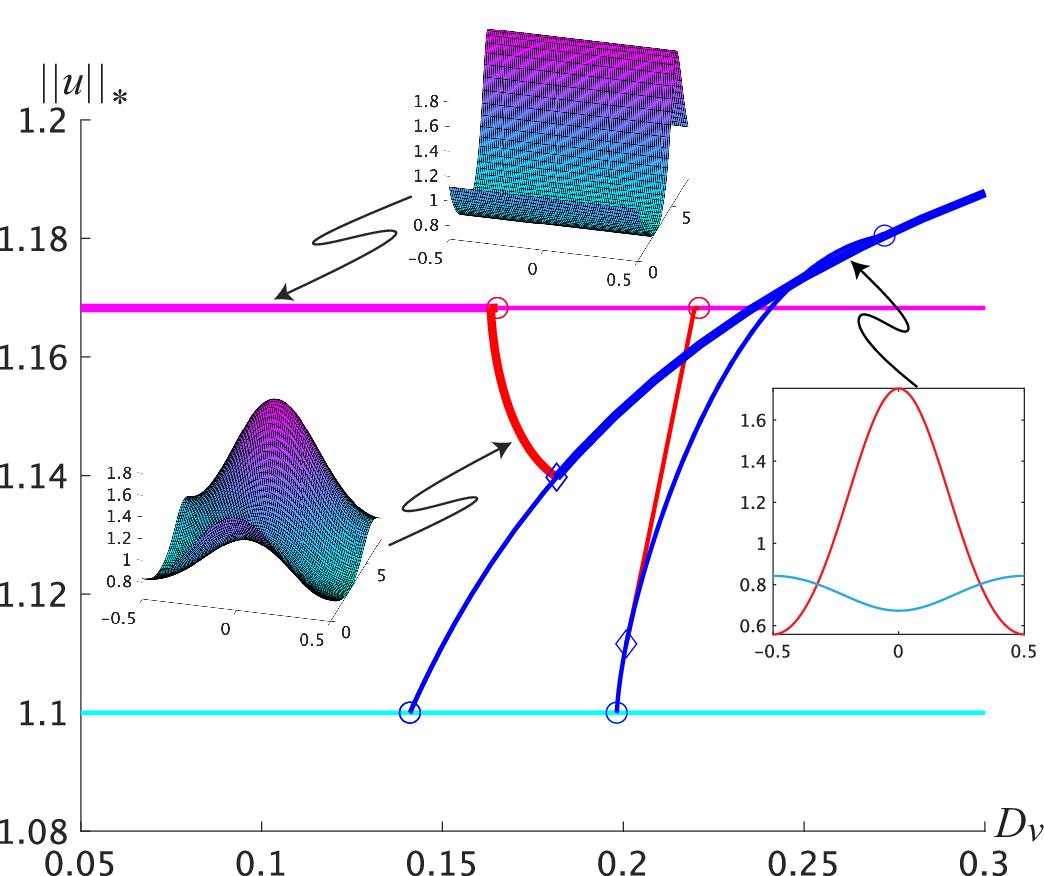}
\caption{
Bifurcation diagram for \eqref{Schnak} when $D_u=0.0125$. 
The horizontal axis and the vertical axis are respectively the bifurcation parameter $D_v$ and the norm $\|u\|_*$. 
The other parameter values are the same ones in Fig.~\ref{BD-1}. 
}
\label{BD-2}
\end{center}
\end{figure}
\begin{figure}[htbp]
\begin{center}
\includegraphics[width=52mm]{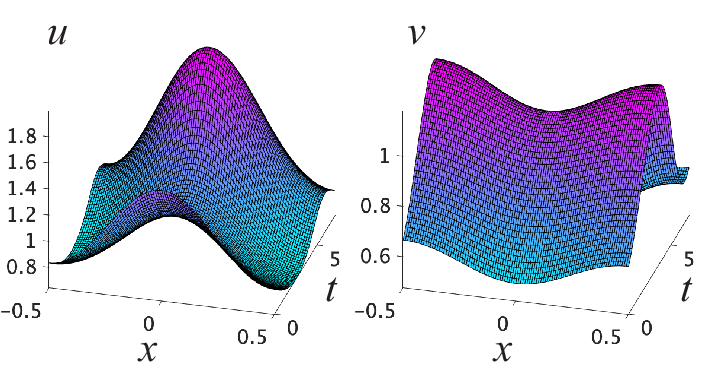}
\caption{Solution profile for $D_u=0.0125$ and $D_v=0.171085$. }
\label{hopfsol-1}
\end{center}
\end{figure}
One can see that the spatially homogeneous time periodic solution is destabilized at $D_v=0.165173$ and a spatially inhomogeneous time periodic solution branch supercritically bifurcates. Therefore, we see stable time periodic solutions with spatial heterogeneity for $0.165173<D_v<0.18156$ (see also Fig.~\ref{hopfsol-1}). 
This type of stable time periodic solution branch is never observed when $\alpha=1$. 
It seems that the oscillatory 0-mode yields this stable time periodic solution branch. 
As well as the previous case, it connects with the non-constant stationary solution branch at $D_v=0.18156$ (marked with a diamond). 
Additionally, another branch of time periodic solutions with spatial heterogeneity can be seen for $0.200817<D_v<0.220962$, but it is unstable. 

Next, we use $D_u$ as a bifurcation parameter and $D_v$ value is fixed suitably. 
We display a bifurcation diagram for \eqref{Schnak} when $D_v=0.25$ in Fig.~\ref{BD-3}. 
\begin{figure}[htbp]
\begin{center}
\includegraphics[width=80mm]{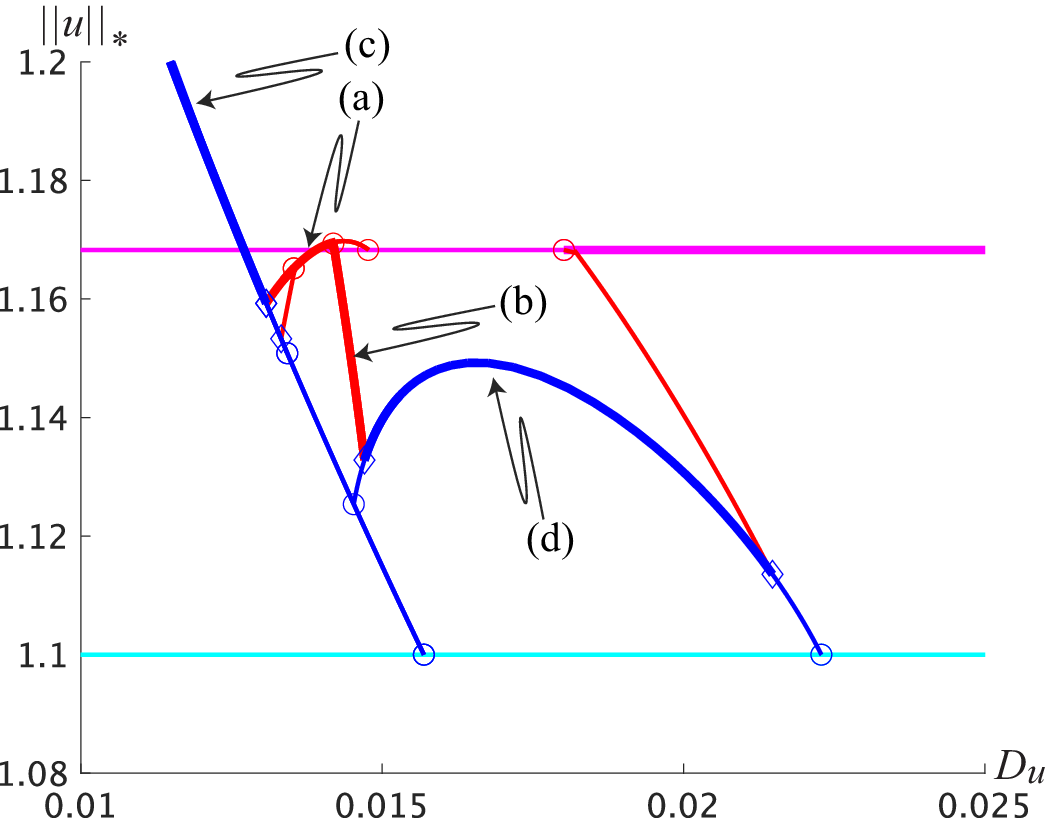}
\caption{Bifurcation diagram for \eqref{Schnak} when $D_v=0.25$. 
The horizontal axis and the vertical axis are respectively the bifurcation parameter $D_u$ and the norm $\|u\|_{*}$. 
The other parameter values are the same ones in Fig.~\ref{BD-1}.
The solution profiles at (a), (b), (c) and (d) in the figure are indicated in Fig.~\ref{hopfsol-3}. }
\label{BD-3}
\end{center}
\end{figure}
\begin{figure}[htbp]
\begin{center}
\includegraphics[width=110mm]{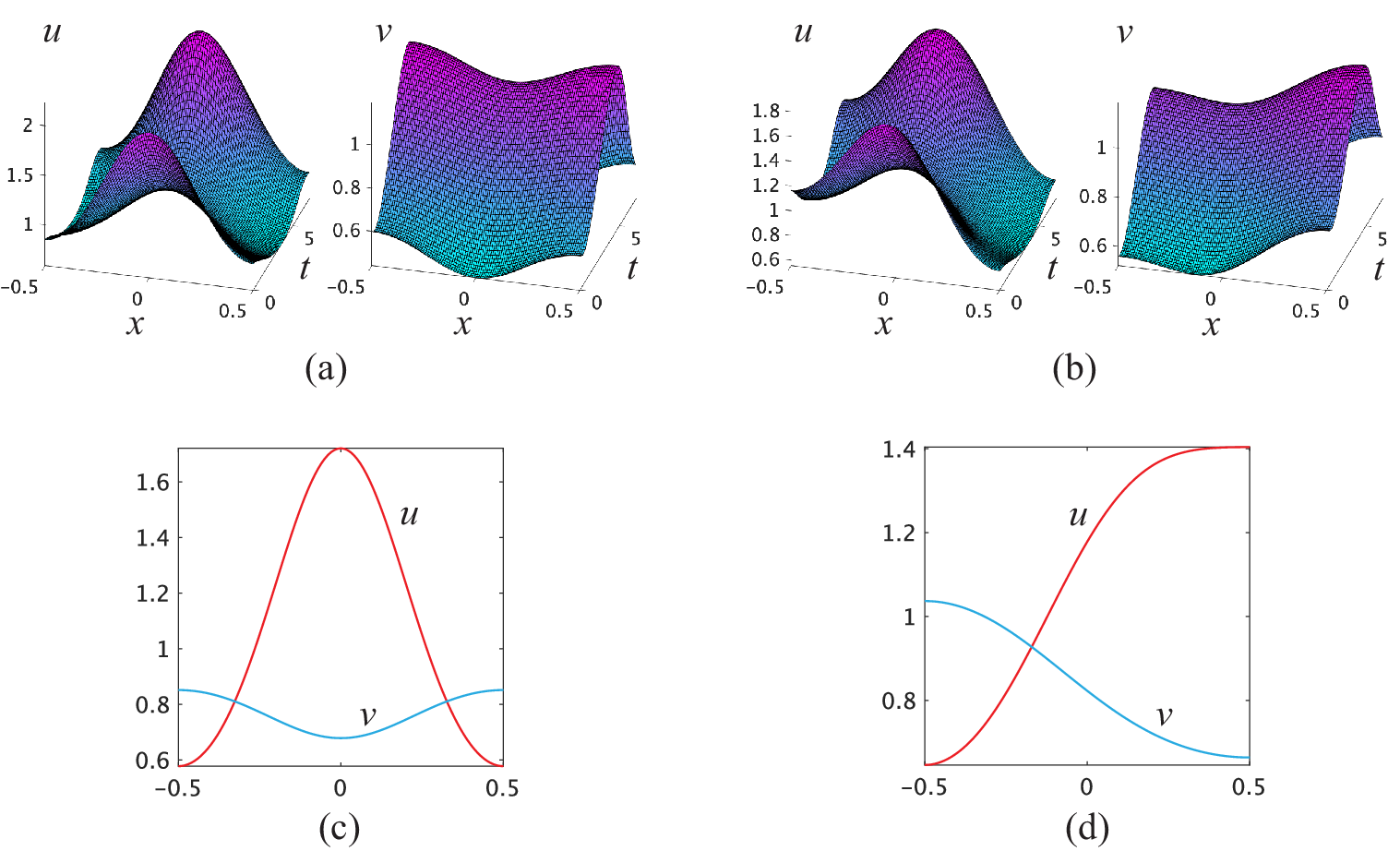}
\caption{
(a) Stable time periodic solution with spatially heterogeneity at $D_u=0.0137468$. 
(b) Stable time periodic solution with spatially heterogeneity at $D_u=0.0145548$. 
(c) Stable non-constant stationary solution at $D_u=0.012565$. 
(d) Stable non-constant stationary solution at $D_u=0.0195245$.}
\label{hopfsol-3}
\end{center}
\end{figure}
We find that the spatially homogeneous time periodic solution is stable for $D_u>0.018015$. 
Besides this time periodic solution, a time periodic solution with spatially heterogeneity in Fig.~\ref{hopfsol-3}(a), which looks like an oscillatory 2-mode solution, and a time periodic solution in Fig.~\ref{hopfsol-3}(b) which seems to an oscillatory $1:2$ mixed mode solution are also found. 
Stable periodic solutions like Fig.~\ref{hopfsol-3}(a) are observed for $0.0130752<D_u<0.0141866$ and Fig.~\ref{hopfsol-3}(b) type stable periodic solutions can be seen for $0.0141866<D_u<0.0147057$. Moreover, each time periodic solution branch connects to a different non-constant stationary solution branch whose profile is exhibited in Fig.~\ref{hopfsol-3}(c) and (d).

We show a bifurcation diagram in Fig.~\ref{BD-4} when $D_v=0.2$. 
\begin{figure}[htbp]
\begin{center}
\includegraphics[width=80mm]{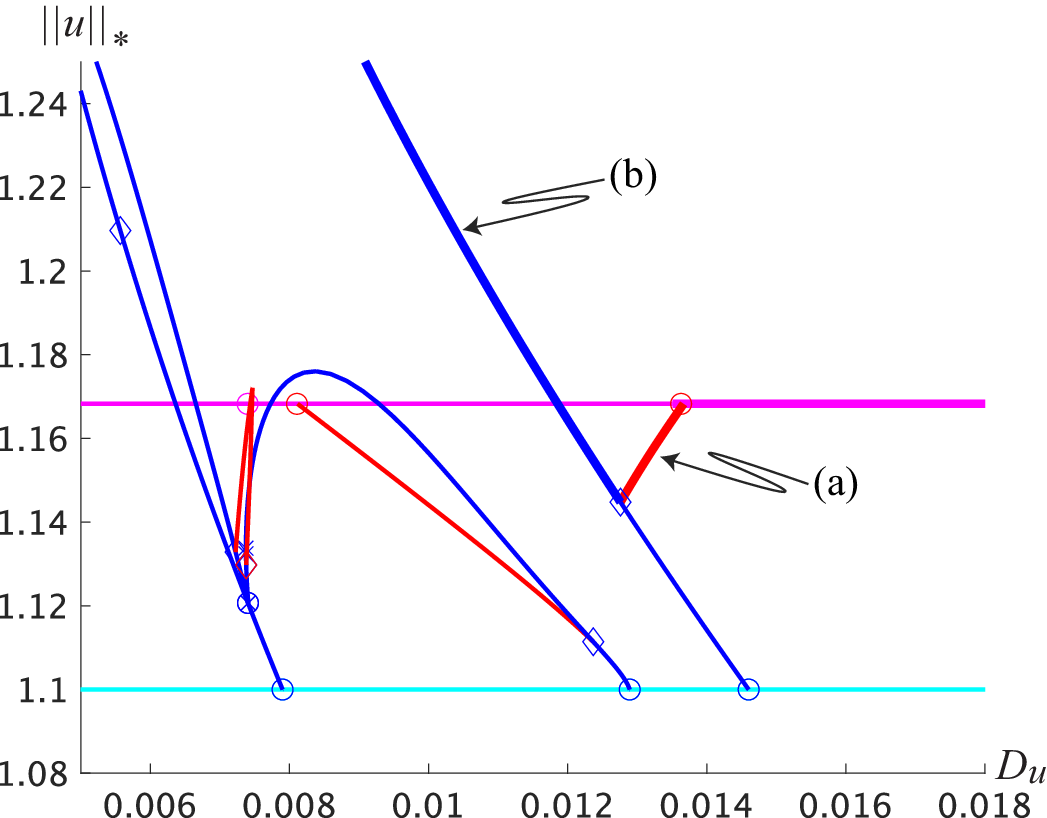}
\caption{Bifurcation diagram for \eqref{Schnak} when $D_v=0.2$. 
The horizontal axis and the vertical axis are respectively the bifurcation parameter $D_v$ and the norm $\|u\|_{*}$. 
The other parameter values are the same ones in Fig.~\ref{BD-1}. 
The solution profiles at (a) and (b) in the figure are indicated in Fig.~\ref{hopfsol-4}. }
\label{BD-4}
\end{center}
\end{figure}
\begin{figure}[htbp]
\begin{center}
\includegraphics[width=91mm]{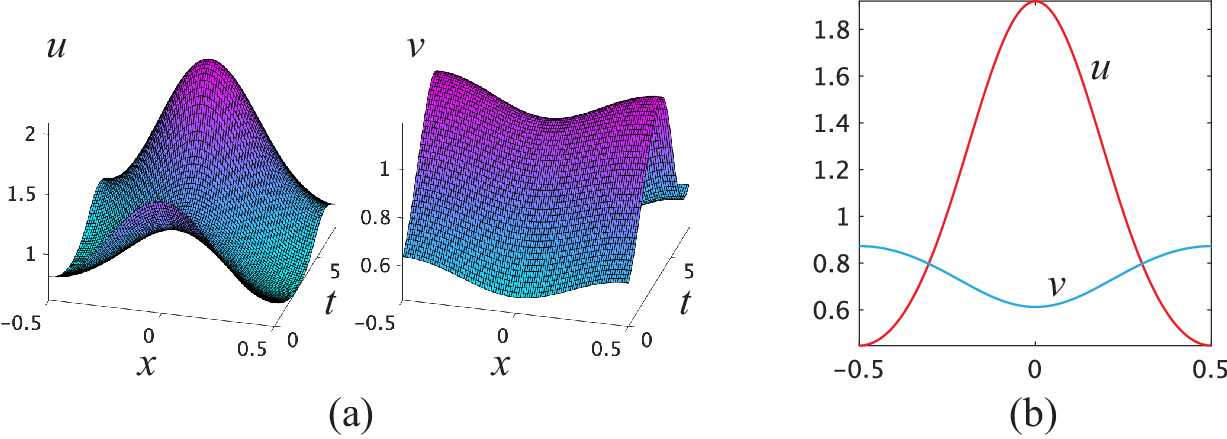}
\caption{
(a) Stable time periodic solution with spatially heterogeneity at $D_u=0.0131236$. 
(b) Stable non-constant stationary solution at $D_u=0.0101167$. }
\label{hopfsol-4}
\end{center}
\end{figure}
As shown in the figure, a stable time periodic solution branch with spatial heterogeneity primarily bifurcates from the spatially homogeneous time periodic solution at $D_u=0.0135106$ as the parameter value $D_u$ decreases. 
Therefore, we can see stable periodic solutions for $0.0127591<D_u<0.0135106$ like Fig.~\ref{hopfsol-4}(a). 
For $D_u<0.0127591$, non-constant stationary solutions exist stably, whose profiles are like Fig.~\ref{hopfsol-4}(b). 
When $D_u$ is small, some branches appear from the spatially homogeneous time periodic solution branch and a complicated bifurcation structure is observed but they all are unstable. 

Compared with Figs.~\ref{BD-1} and \ref{BD-2}, Figs.~\ref{BD-3} and \ref{BD-4} show rather complicated structures. 
As $D_u$ value decreases, infinitely many Fourier modes are destabilized one after another as in Fig.~\ref{NSC}. 
It seems that this tendency influences the complexity of the bifurcation diagram. 

\subsection{Mimura--Murray model}\label{Mimura-Murray model}

Mimura--Murray model was proposed in \cite{MM} to describe a pattern forming process of population densities of prey and predator. 
It is known that this model is classified into the activator-inhibitor systems. 
Mimura--Murray model with a parameter $\alpha$ is written as 
\begin{equation}\label{Mimura-Murray}
\begin{aligned}
\partial_t u&=D_u \partial_{xx} u+\left( \left( a+bu-u^2 \right)/c-v \right)u,\\
\partial_t v&=\alpha \left\{ D_v \partial_{xx} v+\left( u-(1+dv) \right)v, \right\}
\end{aligned}
\end{equation}
and the zero flux boundary conditions are imposed as well as in Section~\ref{Schnakenberg model}. 
Here, we use the following parameter values: 
\[
a=35,\quad b=16, \quad c=9, \quad d=0.4.
\]
Then, the constant stationary solution is given by $(u,v)=(5,10)$ which is independent of $D_u$, $D_v$ and $\alpha$. 
The neutral stability curves on the $(D_u, D_v)$ plane are shown in Fig.~\ref{NSC-2}.
\begin{figure}[htbp]
\begin{center}
\includegraphics[width=80mm]{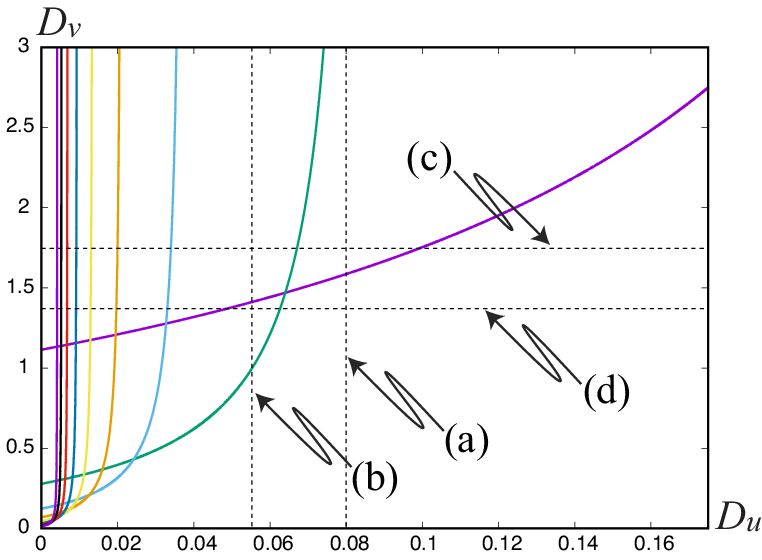}
\caption{Neutral stability curves for the Mimura--Murray model. 
The horizontal axis and the vertical one mean $D_u$ and $D_v$, respectively. }
\label{NSC-2}
\end{center}
\end{figure}
In addition, since the necessary condition for Hopf bifurcation of 0-mode is provided by $\alpha<\frac{5}{6}$, 
we set $\alpha=0.75$, namely the constant stationary solution $(u,v)=(5,10)$ is oscillatorily unstable. 

In this subsection, we show four cases of bifurcation diagrams: 
\begin{description}
\item Fig.~\ref{MM-BD}(a) : $D_u=0.08$ and $D_v$ is a bifurcation parameter.
\item Fig.~\ref{MM-BD}(b) : $D_u=0.055$ and $D_v$ is a bifurcation parameter.
\item Fig.~\ref{MM-BD}(c) : $D_v=1.75$ and $D_u$ is a bifurcation parameter. 
\item Fig.~\ref{MM-BD}(d) : $D_v=1.375$ and $D_u$ is a bifurcation parameter. 
\end{description}
\begin{figure}[htbp]
\begin{center}
\begin{tabular}{cc}
\includegraphics[width=56mm]{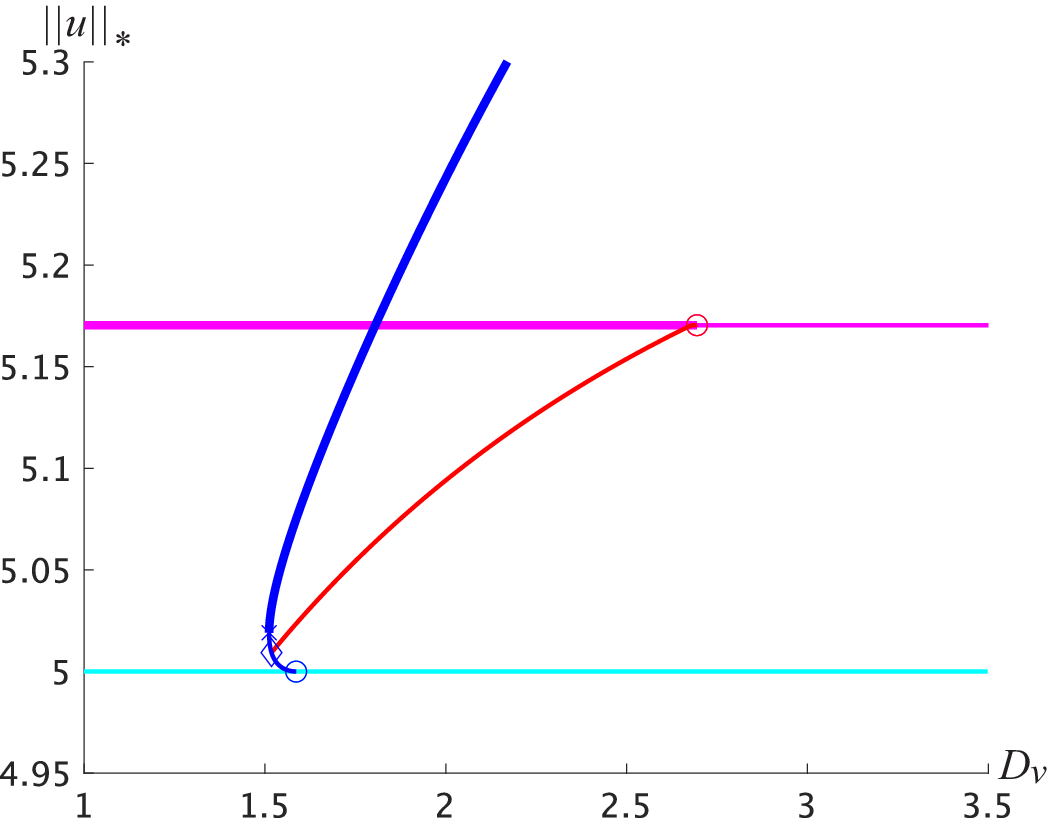}&
\includegraphics[width=56mm]{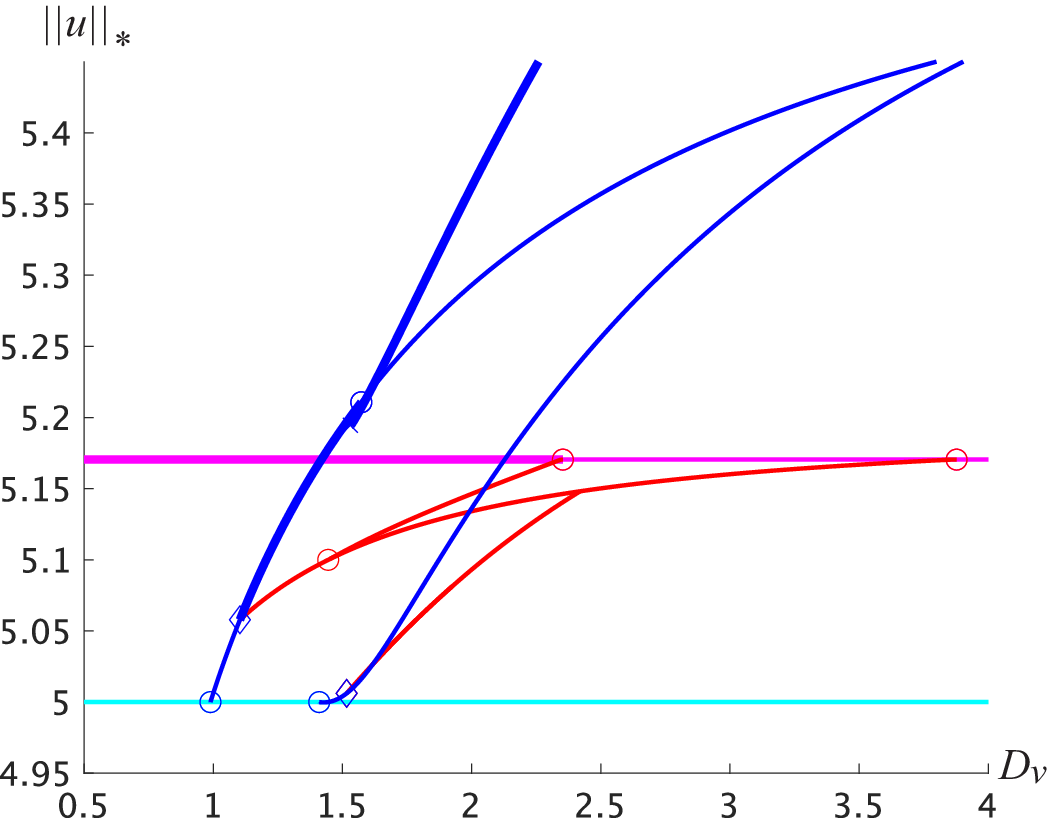}\\
(a) & (b)\\
\includegraphics[width=56mm]{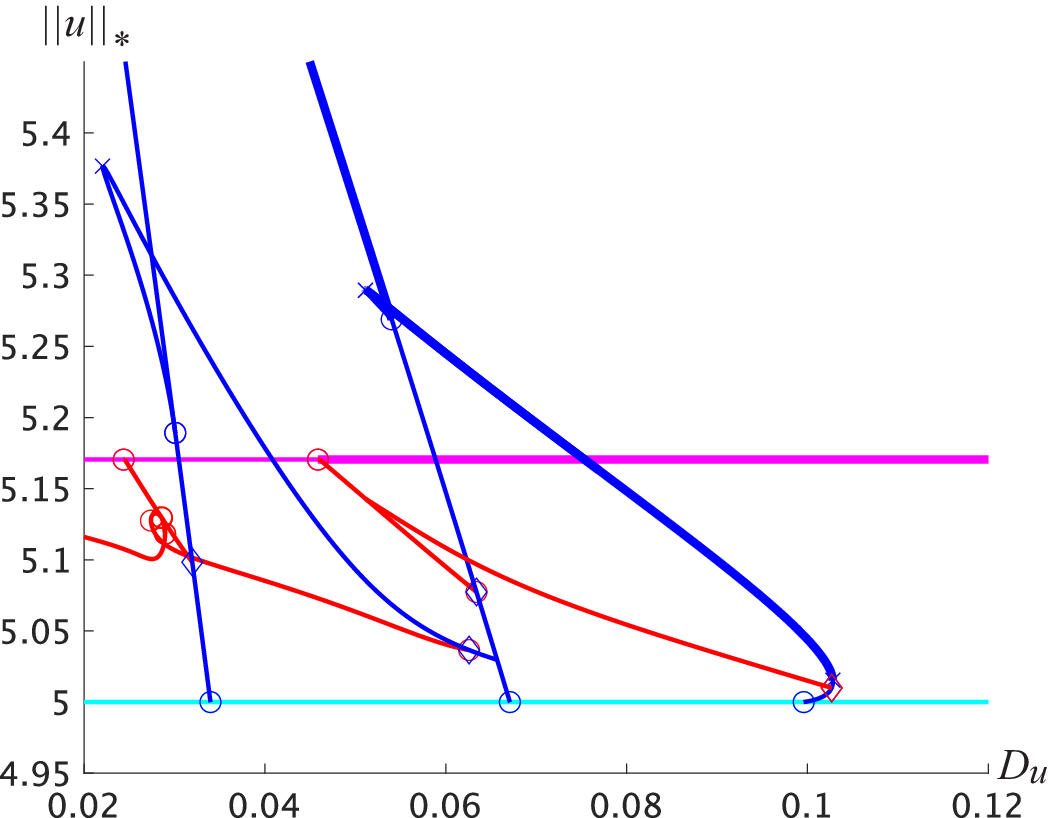}&
\includegraphics[width=56mm]{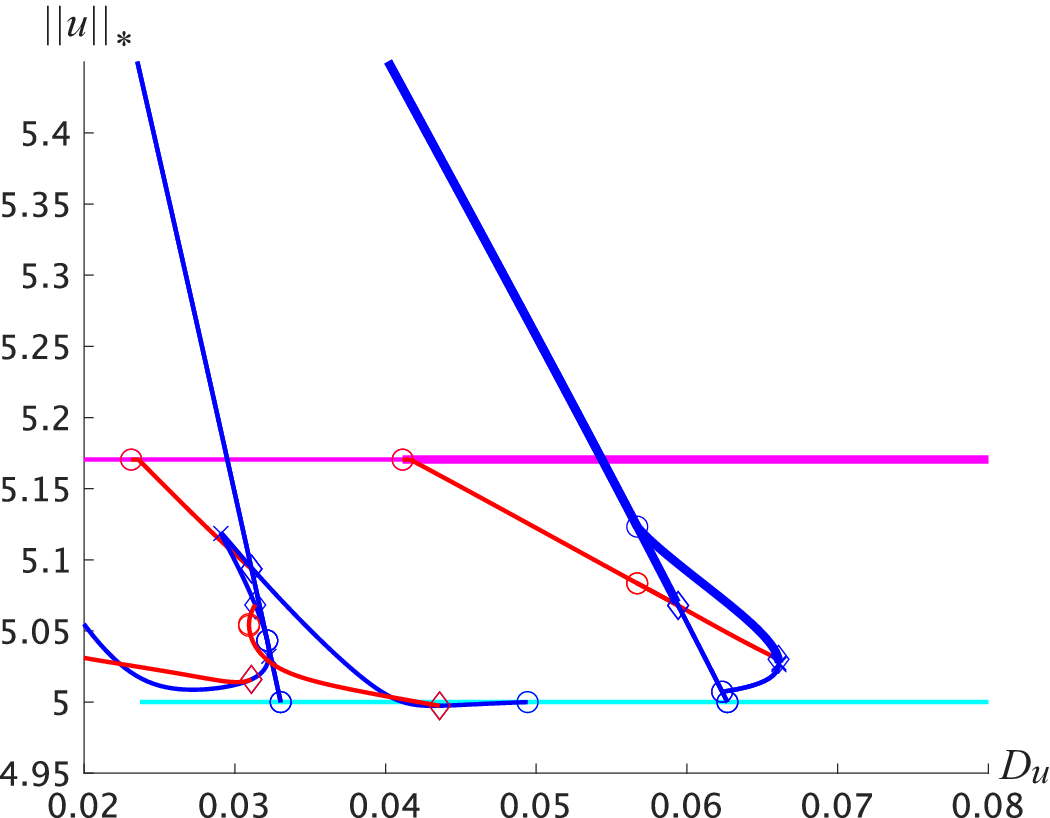}\\
(c) & (d)
\end{tabular}
\caption{Bifurcation diagrams for \eqref{Mimura-Murray} when $\alpha=0.75$. 
The horizontal axis and the vertical axis are respectively the bifurcation parameter indicated in each figure and the norm $\|u\|_{*}$. 
(a) $D_u=0.08$.
(b) $D_u=0.055$.
(c) $D_v=1.75$.
(d) $D_v=1.375$.
}
\label{MM-BD}
\end{center}
\end{figure}

\begin{figure}[htbp]
\begin{center}
\includegraphics[width=110mm]{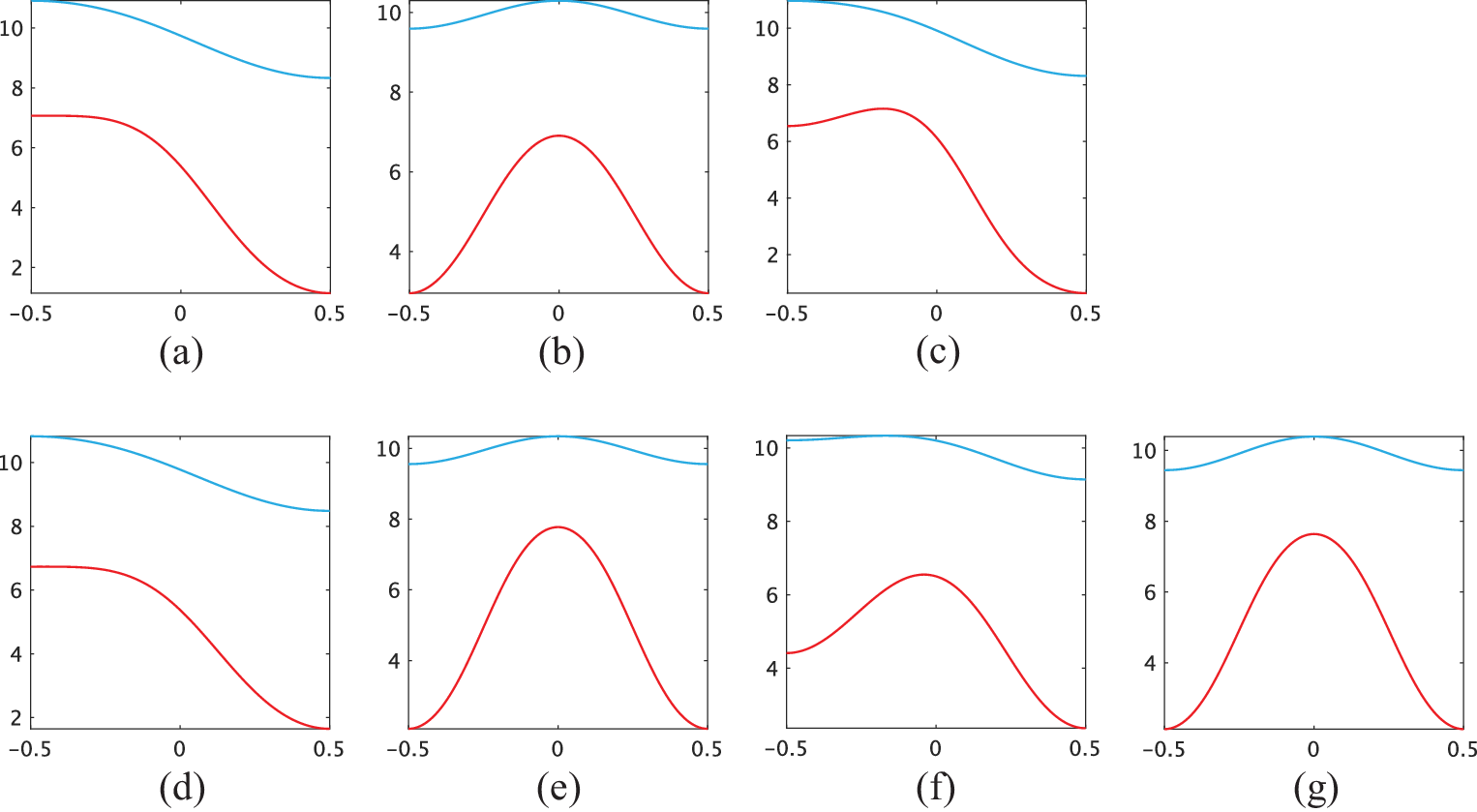}
\caption{Stable stationary solutions.  
(a) $D_u=0.08$ and $D_v=1.89834$.
(b) $D_u=0.055$ and $D_v=1.32413$.
(c) $D_u=0.055$ and $D_v=1.88668$.
(d) $D_u=0.0859729$ and $D_v=1.75$.
(e) $D_u=0.0507831$ and $D_v=1.75$.
(f) $D_u=0.058582$ and $D_v=1.375$. 
(g) $D_u=0.0484719$ and $D_v=1.375$. 
}
\label{stationarysol}
\end{center}
\end{figure}
Fig.~\ref{stationarysol} displays examples of non-constant stationary solutions on the stable branches in Fig.~\ref{MM-BD}. 
It is observed from Fig.~\ref{MM-BD} that there is no stable time periodic solution except for a spatially homogeneous time periodic solution (thick magenta colored line) in these parameter regimes. 
However, several unstable time periodic solution branches with spatial heterogeneity are obtained. 
It appears that many of these periodic solution branches do not exist for the original Mimura--Murray model with $\alpha=1$. 
The appearance of these branches may attribute to the Hopf instability of 0-mode.




\subsection{Artificial system}

So far, we considered the well known reaction-diffusion systems, Schnakenberg model and Mimura--Murray model.  
Our framework discussed in Section~\ref{HTT} can apply to any 2-component reaction-diffusion systems which cause the Turing instability. 
Therefore, we treat an artificial system which does not have any meanings from mathematical model point of view as a final example. 
The system considered is 
\begin{equation}\label{Ar}
\begin{aligned}
\partial_t u&=D_u \partial_{xx} u+u-3v-u^2,\\
\partial_t v&=\alpha \left( D_v \partial_{xx} v+2u-4v-3u^2 \right),
\end{aligned}
\end{equation}
and we impose the zero flux boundary conditions. 
This system is also included into the activator-inhibitor systems. 
Obviously, $(u,v)=(0,0)$ is a constant stationary solution, and the necessary condition for the Hopf instability of $(u,v)=(0,0)$ is given by $\alpha<1/4$. 
A bifurcation diagram when $\alpha=0.24$ and $D_v=0.24$ is shown in Fig.~\ref{BD-5}. 
\begin{figure}[htbp]
\begin{center}
\includegraphics[width=80mm]{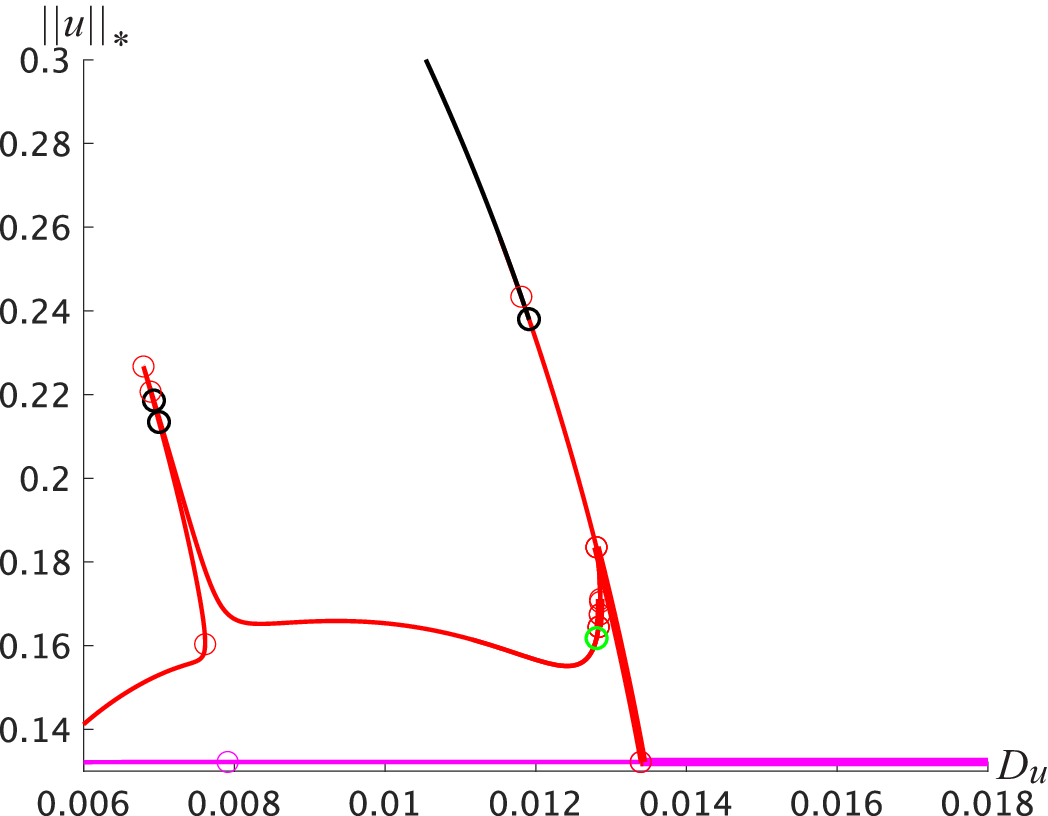}
\caption{Bifurcation diagram for \eqref{Ar} when $\alpha=0.24$ and $D_v=0.24$. }
\label{BD-5}
\end{center}
\end{figure}
\begin{figure}[htbp]
\begin{center}
\includegraphics[width=110mm]{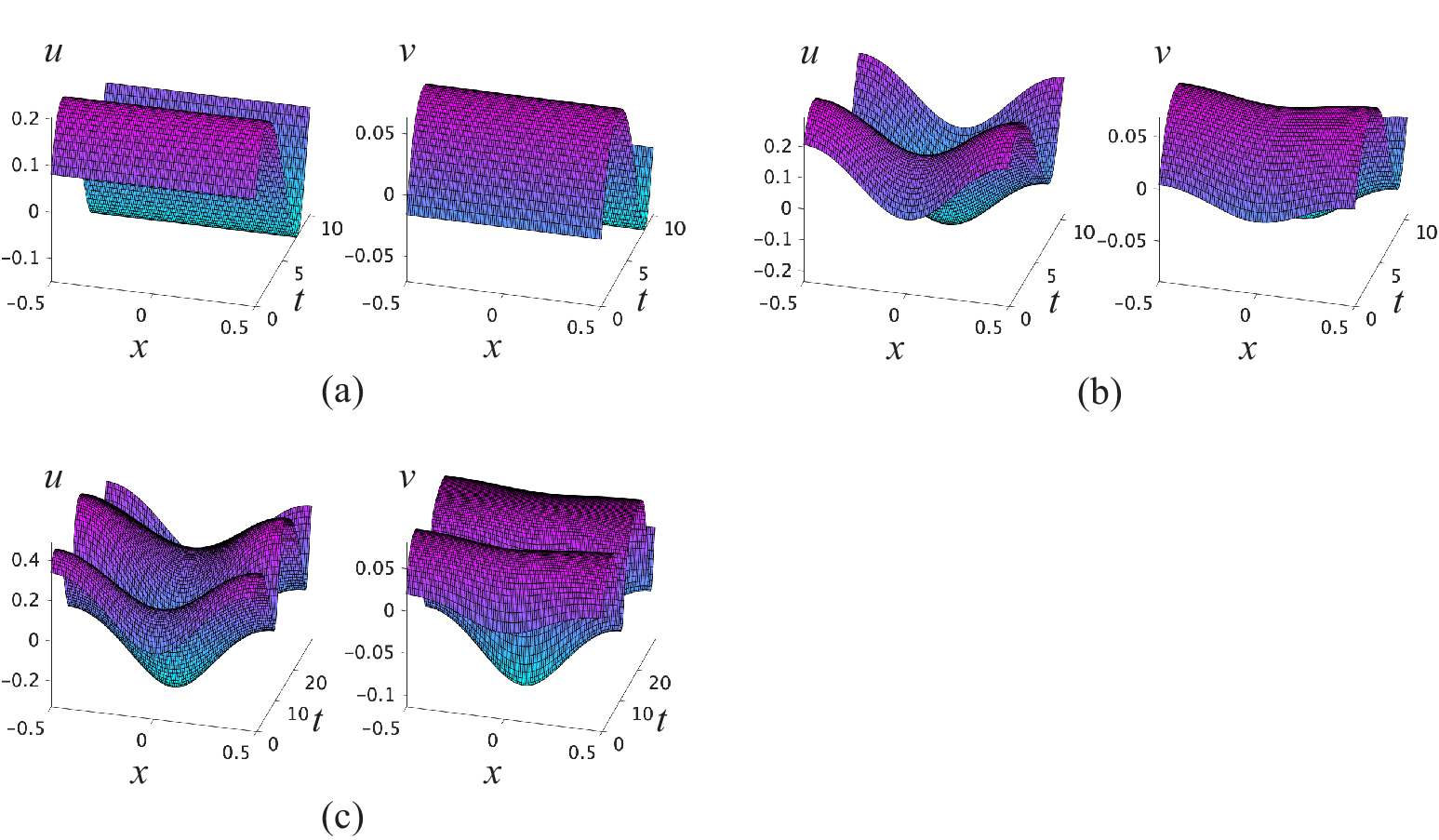}
\caption{Solution profiles. 
(a) Spatially homogeneous time periodic solution. The period is $T=10.1409$.
(b) A stable time periodic solution with spatial heterogeneity at $D_u=0.0132228$. The period is $T=10.5842$. 
(c) An unstable time periodic solution at $D_u=0.0118248$ which is located on the black branch in Fig.~\ref{BD-5} bifurcating from the period-doubling bifurcation point. The period is $T=29.7548$. 
}
\label{hopfsol-5}
\end{center}
\end{figure}
When $D_u$ value is relatively large, a spatially homogeneous time periodic solution exists stably(see Fig.~\ref{hopfsol-5}(a)), but it is destabilized at $D_u=0.0133924$ and a stable time periodic solution branch with spatial heterogeneity appears due to a bifurcation (Fig.~\ref{hopfsol-5}(b) shows a typical stable periodic solution). 
This stable time periodic solution branch loses its stability at $D_u=0.0128053$. 
However, as the value of $D_u$ decreases further, this branch undergoes a period-doubling bifurcation at $D_u=0.0119096$ which is marked with a black circle. 
The black curve in Fig.~\ref{BD-5} indicates a time periodic solution branch bifurcating from the period-doubling bifurcation point. A solution profile on this branch is exhibited in Fig.~\ref{hopfsol-5}(c) which shows that the period is much longer. 
In Fig.~\ref{BD-5}, there are two more period-doubling bifurcation points marked with black circles though we do not compute period-doubling branches. 
Moreover, interestingly, a torus bifurcation point marked with a green circle is also detected. 
Since period-doubling bifurcation points, a torus bifurcation point and other bifurcation points can be seen, it is expected that \eqref{Ar} possesses much more complicated bifurcation structure.  

\bigskip

In this section, we numerically investigated bifurcation diagrams for three reaction-diffusion systems with a aid of a Matlab package pde2path. 
Depending on systems and parameter values, bifurcation structures drastically change. 
Moreover, it is suggested that the Hopf instability of 0-mode makes bifurcation structure richer. 
For these complicated bifurcations, it seems that 
the organizing center is a triply degenerate point;  Hopf instability of 0-mode and diffusion-induced instabilities of $m$ and $m+1$-modes. 
To our best knowledge, this triply degenerate point is probably a point where the most complicated bifurcation structures can be generated in 2-component reaction-diffusion systems. 
In the next section, we focus on this point and derive a finite dimensional dynamical system via the center manifold reduction from 2-component reaction-diffusion systems to unveil a variety of pattern dynamics.



\section{Center manifold reduction}\label{CMR}

In this section, we derive the normal form for the Hopf--Turing--Turing bifurcation from \eqref{2rd} with \eqref{BC}.
In particular, we focus on the triply degenerate point with $0$-mode Hopf instability and $1$ and $2$-modes Turing instabilities. For the normal form for the Hopf--Turing--Turing bifurcation with $0:m:m+1$, see Remark~\ref{rem:normal}.

\subsection{Dynamical system on Fourier space}
In order to derive the normal form for the Hopf--Turing--Turing bifurcation, 
we impose the following assumption on \eqref{2rd}: 
\begin{assumption}\label{ass:rd}
  \begin{enumerate}
  \item[(1)]
  The functions $f(u, v)$ and $g(u, v)$ are sufficiently smooth;
  \item[(2)]
    The system \eqref{2rd} has the trivial solution $(u,v) = (\tilde u, \tilde v)$;
  \item[(3)]
    The linearized matrix $M$ satisfies $\tr{M} < 0$ and $\det{M} > 0$.
  \end{enumerate}
\end{assumption}
We define the phase space for the dynamical system \eqref{2rd} as
\begin{align*}
\mathscr{X}_\mathrm{N} := \{&(u,v) \in H^{2} (0,L) \times H^{2} (0,L);\, \partial_x u = \partial_x v = 0 \ \text{at} \ x = 0, L\},
\end{align*}
where
\begin{equation*}
\|(u,v)\|_{\mathscr{X_\mathrm{N}}} := \sqrt{\|u\|^{2}_{H^{2}(0,L)} + \|v\|^{2}_{H^{2}(0,L)}}.
\end{equation*}
Note that solutions of \eqref{2rd} with \eqref{BC} can be considered as those of periodic boundary problem with period $2L$.
If $(u(x, t), v(x, t)) \in \mathscr{X}_\mathrm{N}$ is a solution of \eqref{2rd}, then the extended solution $(u^* (x, t), v^* (x, t))$ for $x \in [0, 2L]$ such as
\[
u^* (x, t) = 
\begin{cases}
    u(x, t) & x \in [0, L],\\
    u(2L - x, t) & x \in [L, 2L],
\end{cases} \
v^* (x, t) = 
\begin{cases}
    v(x, t) & x \in [0, L],\\
    v(2L - x, t) & x \in [L, 2L]
\end{cases} 
\]
is a solution to the following system:
\begin{equation}\label{rde2}
  \begin{cases}
    \partial_t u = D_u \partial_{xx} u + f(u,v), \quad (x, t) \in (0, 2 L) \times (0, \infty)\\
    \partial_t v = \alpha \left( D_v \partial_{xx} v + g(u,v) \right)  \quad (x, t) \in (0, 2 L) \times (0, \infty),\\
    u(x, t) = u(x + 2L, t), \ \partial_x u(x, t) = \partial_x u(x + 2 L, t), \ t \in (0, \infty), \\
    v(x, t) = v(x + 2L, t), \ \partial_x v(x, t) = \partial_x v(x + 2 L, t), \ t \in (0, \infty).
  \end{cases}
\end{equation}
Hence, we consider the dynamical system \eqref{rde2} in a phase space
\begin{equation}\label{eq:pbc}
\begin{aligned}
\mathscr{X}_{\mathrm{per}} = \{(u,v) &\in H^{2}_{\mathrm{per}} (0, 2L) \times H^{2}_{\mathrm{per}} (0, 2L);\\
&(u(x), v(x)) = (u(2L - x), v(2L - x))\}.
\end{aligned}
\end{equation}
From the above formulation, we can represent a solution to \eqref{rde2} as the complex Fourier series:
\begin{equation*}
  u(x, t) - \tilde u = \sum_{m\in\Z}u_m(t) \e^{\I m k x}, \quad  v(x, t) - \tilde v = \sum_{m\in\Z}v_m(t) \e^{\I m k x}, \quad k = \dfrac{\pi}{L}.
\end{equation*}
Then, by the projection
\begin{equation}
\mathcal{P}(u,v) = \left\{\dfrac{1}{2L}\int^{2L}_{0} ( u(t,x), v(t,x) )\mathrm{e}^{-\mathrm{i}mkx}dx\right\}_{m\in\mathbb{Z}}
\end{equation}
the system \eqref{rde2} is equivalent to
\begin{equation}\label{eq:infinite}
  \begin{pmatrix}
    \dot u_m\\
    \dot v_m
  \end{pmatrix}
  =
  M_{m, \alpha}
  \begin{pmatrix}
    u_m\\
    v_m
  \end{pmatrix}
  +
  \begin{pmatrix}
    F_m\\
    G_m
  \end{pmatrix}, \quad m\in\Z,
\end{equation}
where $(u_m, v_m) \in \R \times \R$ $(m \in \mathbb{Z})$ and $(F_m, G_m)$ is written by
\begin{multline*}
  F_m = \sum_{m_{1}, m_{2} \in \mathbb{Z} \atop m_1 + m_2 = m}\bigg(\dfrac{f_{uu}}{2} u_{m_1} u_{m_2} + f_{uv} u_{m_1} v_{m_2} + \dfrac{f_{vv}}{2} v_{m_1} v_{m_2}\bigg)\\
  + \sum_{m_{1}, m_{2}, m_{3}\in \mathbb{Z} \atop m_1 + m_2 + m_3= m}\bigg(\dfrac{f_{uuu}}{6} u_{m_1} u_{m_2} u_{m_3} + \dfrac{f_{uuv}}{2} u_{m_1} u_{m_2} v_{m_3} \\
  + \dfrac{f_{uvv}}{2} u_{m_1} v_{m_2} v_{m_3} + \dfrac{f_{vvv}}{6} v_{m_1} v_{m_2} v_{m_3}\bigg),
\end{multline*}
\begin{multline*}
  G_m = \alpha \Bigg\{ \sum_{m_{1}, m_{2} \in \mathbb{Z} \atop m_1 + m_2 = m}\bigg(\dfrac{g_{uu}}{2} u_{m_1} u_{m_2} + g_{uv} u_{m_1} v_{m_2} + \dfrac{g_{vv}}{2} v_{m_1} v_{m_2}\bigg)\\
  + \sum_{m_{1}, m_{2}, m_{3} \in \mathbb{Z} \atop m_1 + m_2 + m_3= m}\bigg(\dfrac{g_{uuu}}{6} u_{m_1} u_{m_2} u_{m_3} + \dfrac{g_{uuv}}{2} u_{m_1} u_{m_2} v_{m_3}\\
  + \dfrac{g_{uvv}}{2} u_{m_1} v_{m_2} v_{m_3} + \dfrac{g_{vvv}}{6} v_{m_1} v_{m_2} v_{m_3}\bigg) \Bigg\}.
\end{multline*}
The phase space for the dynamical system \eqref{eq:infinite} is defined by
\begin{equation}
\begin{aligned}\label{eq:phase_space}
\mathscr{X}_{F} := \Big\{ \{(u_{m},v_{m})\}_{m\in\mathbb{Z}} &;\, \ (u_{m},v_{m}) = (u_{-m}, v_{-m}), \\
&\qquad \qquad \| \{(u_{m},v_{m})\}_{m\in\mathbb{Z}} \|^{2}_{\mathscr{X}_F} < \infty \Big\}
\end{aligned}
\end{equation}
with the norm $\| \{(u_{m},v_{m})\}_{m\in\mathbb{Z}} \|^2_{\mathscr{X}_F} := \sum_{m\in\mathbb{Z}}(1+m^{2})^{2}|(u_{m},v_{m})|^{2}$, 
which is equivalent to $\mathscr{X}_{\mathrm{per}}$ by the projection $\mathcal{P}:\mathscr{X}_{\mathrm{per}} \to \mathscr{X}_{F}$.

Note that $(u_{m}, v_{m}) = (u_{-m}, v_{-m})$ holds from the symmetry \eqref{eq:pbc}.
Therefore, it is sufficient to consider the system \eqref{eq:infinite} for $m \ge 0$.
In this section, we study bifurcation structures of the system \eqref{eq:infinite} on \eqref{eq:phase_space} instead of the system \eqref{rde2} on \eqref{eq:pbc}.
We also remark that the systems \eqref{2rd} and \eqref{rde2} are invariant under the spatial translation $x \mapsto x + \eta$ $(\eta \in \mathbb{R})$ and the reflection $x \mapsto -x$.
This implies that the systems are invariant with respect to the actions defined by
\[
\mathcal{T}_\eta \begin{pmatrix} u(x) \\ v(x) \end{pmatrix} =  \begin{pmatrix} u(x + \eta) \\ v(x + \eta) \end{pmatrix}, \quad \mathcal{S}  \begin{pmatrix} u(x) \\ v(x) \end{pmatrix} =  \begin{pmatrix} u(-x) \\ v(-x) \end{pmatrix},
\]
and hence, normal forms on center manifolds are equivalent with respect to the symmetry operations:
\begin{equation}\label{o2}
(u_m, v_m) \mapsto (\e^{\mathrm{i}mk\eta}u_m,\e^{\mathrm{i}mk\eta}v_m), \quad (u_m, v_m) \mapsto (\bar u_m, \bar v_m).
\end{equation}

We now consider the linearized eigenvalue problem for each Fourier mode.
In the following, we regard a triplet $(\alpha, D_u, D_v)$ as the bifurcation parameters.
We call a pair of parameters $(D_u,D_v)$ satisfying $\det M_{\pm1, \alpha}= \det M_{\pm2, \alpha} = 0$ as 1:2 degenerate point.
The 1:2 degenerate point is given by
\begin{equation}\label{degenerate-point}
(D_u, D_v) = (D_u^*, D_v^*) := \left(\dfrac{5 \delta - \sqrt{25 \delta^2 - 16 f_u g_v \delta}}{8 g_v k^2} , \dfrac{g_v k^2 D_u^* - \delta}{k^2 (k^2 D_u^* - f_u)} \right),
\end{equation}
where we put $\delta := f_u g_v - f_v g_u$.
As we have already seen in Section~\ref{HTT}, the neutral stability curves are independent of $\alpha$.
This implies that for fixed $\alpha = \alpha^*$, the linearized operator could have 1:2 degenerate point and Hopf instability point simultaneously.
That is, the following holds:
\begin{proposition}
Assume that the system \eqref{2rd} satisfies Assumption~\ref{ass:rd}.
Then, for given $L > 0$, there exists $(\alpha, D_{u}, D_v) = (\alpha^{*}, D_{u}^{*}, D_v^{*})$ such that the linearized operator of \eqref{2rd} at the constant stationary solution $(u,v) = (\tilde u, \tilde v)$ has a pair of purely imaginary eigenvalues and zero eigenvalues whose multiplicity is two.
More precisely, the eigenvalues of $M_{0, \alpha}$, $M_{1, \alpha}$ and $M_{2, \alpha}$ at $(\alpha, D_{u}, D_v) = (\alpha^{*}, D_{u}^{*}, D_v^{*})$ are $\pm \mathrm{i}\omega$, $0$ and $0$, respectively.
The others are strictly negative.
\end{proposition}
We call the parameter set $(\alpha, D_{u}, D_v) = (\alpha^{*}, D_{u}^{*}, D_v^{*})$ as Hopf--Turing--Turing instability point.


\subsection{Normal form for the Hopf--Turing--Turing bifurcation}\label{sec:NF}

Put
\begin{equation}\label{eigen}
\begin{aligned}
\mu_0 &= \mu_0(\alpha) = \dfrac{\tr M_{0, \alpha}}{2},\\
\omega_0 &= \omega_0(\alpha) = \dfrac{1}{2}\sqrt{4\det M_{0, \alpha} - (\tr M_{0, \alpha})^2},\\
\mu_j^{\pm} &= \mu_j^{\pm}(\alpha, D_u, D_v) = \dfrac{\tr M_{j, \alpha} \pm \sqrt{(\tr M_{j, \alpha})^2 - \det M_{j, \alpha}}}{2} \quad (j = 1, 2).
\end{aligned}
\end{equation}
Then, the eigenvalues of $M_{0, \alpha}$ can be written by $\lambda_0^\pm = \mu_0 \pm \I \omega_0$ for $|\alpha - \alpha^*| \ll 1$.
Note that $\mu_0 = 0$, $\omega_0 > 0$, $\mu_j^+ = 0$ and $\mu_j^- = \tr M_{j, \alpha} < 0$ $(j = 1, 2)$ hold on the Hopf--Hopf--Turing instability point.

To derive the normal form on the center manifolds, we introduce the change of variables $(\alpha_j, \beta_j) = T_j^{-1}(u_j, v_j)$ for $j = 0, 1,2$, that transforms \eqref{eq:infinite} into
\begin{equation}\label{eq:tth-diagonal}
  \begin{cases}
    \begin{pmatrix}
      \dot\alpha_0\\
      \dot\beta_0
    \end{pmatrix}
    =
    \begin{pmatrix}
      \mu_0 & -\omega_0\\
      \omega_0 & \mu_0
    \end{pmatrix}
    \begin{pmatrix}
      \alpha_0\\
      \beta_0
    \end{pmatrix}
    + T_0^{-1}
    \begin{pmatrix}
      F_0\\
      G_0
    \end{pmatrix}, \bigskip \\
    \begin{pmatrix}
      \dot\alpha_j\\
      \dot\beta_j
    \end{pmatrix}
    =
    \begin{pmatrix}
      \mu_j^+ & 0\\
      0 & \mu_j^-
    \end{pmatrix}
    \begin{pmatrix}
      \alpha_j\\
      \beta_j
    \end{pmatrix}
    + T_j^{-1}
    \begin{pmatrix}
      F_j\\
      G_j
    \end{pmatrix}, \quad j = 1,2,\bigskip\\
    \begin{pmatrix}
      \dot u_m\\
      \dot v_m
    \end{pmatrix}
    = M_{m, \alpha}
    \begin{pmatrix}
      u_m\\
      v_m
    \end{pmatrix} + 
    \begin{pmatrix}
      F_m\\
      G_m
    \end{pmatrix}, \quad m \in\mathbb{N}\setminus\{1,2\}
  \end{cases}
\end{equation}
where
\begin{equation*}
T_0 = 
\begin{pmatrix} - f_v & 0 \\
f_u - \mu_0 & \omega_0
\end{pmatrix}, \quad
  T_j =
  \begin{pmatrix}
    -f_v & -f_v\\
    f_u - \mu_j^+ & f_u - \mu_j^-
  \end{pmatrix}.
\end{equation*}
By applying the center manifold reduction (\cite{HI}), we obtain the following.
\begin{proposition}\label{prop:normalform-tth}
  Put $\boldsymbol{\alpha} := (\alpha_{0}, \beta_{0}, \alpha_{1}, \alpha_{-1}, \alpha_{2}, \alpha_{-2})$ and $\bs\mu := (\alpha, D_u, D_v) - (\alpha^*, D_u^*, D_v^*)$.
For a sufficiently small $\varepsilon > 0$, let $\mathcal{N}_{\varepsilon}$ be a neighborhood of $\mathscr{X}_F \times \mathbb{R}^{3}$:
  \begin{equation}
  \begin{aligned}
  \mathcal{N}_{\varepsilon} &:= \big\{ (\{(u_m, v_m)\}_{m\in\mathbb{Z}}, \bs{\mu}) \in \mathscr{X}_F \times \mathbb{R}^{3};\, \|\{(u_m, v_m)\}_{m \in \mathbb{Z}}\|_{\mathscr{X}_F} + |\bs{\mu}| < \varepsilon \big\}.
\end{aligned}
  \end{equation}
  Then, for given functions $f$ and $g$ satisfying Assumption~\ref{ass:rd} and a positive constant $L$,  there exists a positive constant $\varepsilon$ such that the local center manifold $\mathcal{M}^{c}_{loc}$ of \eqref{eq:infinite} is contained in $\mathcal{N}_{\varepsilon}$.
  Furthermore, the dynamics of the dynamical system \eqref{eq:infinite} on $\mathcal{M}^{c}_{loc}$ is locally topologically equivalent to the dynamics given by the following dynamical system:
  \begin{equation}\label{eq:reduced}
    \begin{cases}
      \dot \alpha_0 = \mu_0 \alpha_0 - \omega_0 \beta_0 + A_1\alpha_0^2 + A_2 \alpha_0 \beta_0 + A_3 \beta_0^2 + A_4 \alpha_1^2 + A_5 \alpha_2^2\\
      \quad+ (a_1 \alpha_0^2 + a_2 \alpha_1^2 + a_3 \alpha_2^2) \alpha_0 + (a_4 \alpha_0 + a_5 \beta_0) \alpha_0 \beta_0\\
      \quad+ (a_6 \beta_0^2 + a_7 \alpha_1^2 + a_8 \alpha_2^2) \beta_0 + a_9 \alpha_1^2 \alpha_2
      + \mathcal{O}_4,\\
      \dot \beta_0 = \omega_0 \alpha_0 + \mu_0 \beta_0 + B_1\alpha_0^2 + B_2 \alpha_0 \beta_0 + B_3 \beta_0^2 + B_4 \alpha_1^2 + B_5 \alpha_2^2\\
      \quad + (b_1 \alpha_0^2 + b_2 \alpha_1^2 + b_3 \alpha_2^2) \alpha_0 + (b_4 \alpha_0 + b_5 \beta_0) \alpha_0\beta_0\\
      \quad + (b_6 \beta_0^2 + b_7 \alpha_1^2 + b_8 \alpha_2^2)\beta_0 + b_9 \alpha_1^2 \alpha_2
      + \mathcal{O}_4,\\
      \dot \alpha_1 = \mu_1^+ \alpha_1 + (C_1\alpha_0 + C_2\beta_0)\alpha_1 + C_3 \alpha_1 \alpha_2\\
      \quad + (c_1\alpha_0^2 + c_2 \alpha_0 \beta_0 + c_3 \beta_0^2 + c_4 \alpha_1^2 + c_5 \alpha_2^2) \alpha_1\\
      \quad + c_6 \alpha_0 \alpha_1 \alpha_2 + c_7 \beta_0 \alpha_1\alpha_2 + \mathcal{O}_4,\\
      \dot \alpha_2 = \mu_2^+ \alpha_2 + D_1 \alpha_1^2 + (D_2 \alpha_0 + D_3 \beta_0) \alpha_2\\
      \quad + (d_1 \alpha_0^2 + d_2 \alpha_0 \beta_0 + d_3 \beta_0^2 + d_4 \alpha_1^2 + d_5 \alpha_2^2) \alpha_2\\
      \quad + (d_6 \alpha_0 + d_7 \beta_0) \alpha_1^2 + \mathcal{O}_4,
    \end{cases}
  \end{equation}
where the coefficients $A_{j}$, $a_{j}$, $B_{j}$, $b_{j}$, $D_{j}$, $d_{j}$, $E_{j}$ and $e_{j}\in\mathbb{R}$ depend on the coefficients and parameters appearing in \eqref{eq:infinite}, and $\mathcal{O}_4 = \mathcal{O}(|(\bs{\alpha}, \bs{\mu})|^4)$.
\end{proposition}
\begin{proof}
  Let the index set be $\mathscr{I} := \{0, \pm 1, \pm 2\}$.
  Center manifold theorem with suspension trick states that there exist functions $\beta_j = h_j^{\beta}(\boldsymbol{\alpha}, \boldsymbol{\mu})$, $u_m = h_m^{u}(\boldsymbol{\alpha}, \boldsymbol{\mu})$ and $v_m = h_m^{v}(\boldsymbol{\alpha}, \boldsymbol{\mu})$ such that the graphs have a quadratic tangency with the center eigenspace at the origin, that is, $h_{j}^{\beta}$, $h_m^{u}$ and $h_m^{v}$ is $\mathcal{O}(|(\boldsymbol{\alpha}, \boldsymbol{\mu})|^2)$ as $(\boldsymbol{\alpha}, \boldsymbol{\mu}) \to \boldsymbol{0}$. 
Then the local invariant manifold $\mathcal{M}^c_{loc}$ is given by
\begin{align*}
\mathcal{M}^c_{loc} &= \big\{ \{(\alpha_i, \beta_i), \ (u_m, v_m) \}_{i \in \mathscr{I}, \ |m| \not\in \mathscr{I}} \in \mathscr{X}_F ; \, \beta_i = h_i^{\beta}(\bs{\alpha}, \bs{\mu}), \\
& (u_m, v_m) = (h_m^u(\bs{\alpha}, \bs{\mu}), h_m^v(\bs{\alpha}, \bs{\mu})), \ i \in \mathscr{I}\setminus\{0\}, \ m \not\in \mathscr{I} \big\}.
\end{align*} 
Therefore, we have
  \begin{align}
    \dot u_m &= \dot h_m^u = \sum_{j \in \mathscr{I}}\dfrac{\pa h_m^u}{\pa \alpha_j}\dot\alpha_j + \sum_{j\in\mathscr{I}\setminus\{0\}}\dfrac{\pa h_m^u}{\pa \mu_j^{+}} \dot \mu^{+}_j + \dfrac{\pa h_m^u}{\pa \beta_0}\dot\beta_0 + \dfrac{\pa h_m^u}{\pa \mu_0}\dot\mu_0\\
    &= \dfrac{\pa h_m^u}{\pa \alpha_0}(\mu_0 \alpha_0 - \omega_0 \beta_0)
     + \sum_{j \in \mathscr{I}\setminus\{0\}}\dfrac{\pa h_m^u}{\pa \alpha_j}\mu_j^+\alpha_j + \dfrac{\pa h_m^u}{\pa \beta_0}(\omega_0 \alpha_0 + \mu_0 \beta_0) + \mathcal{O}_3,\\
    \dot v_m &= \dot h_m^u = \sum_{j \in \mathscr{I}}\dfrac{\pa h_m^v}{\pa \alpha_j}\dot\alpha_j + \sum_{j\in\mathscr{I}\setminus\{0\}}\dfrac{\pa h_m^v}{\pa \mu_j^{+}} \dot \mu^{+}_j + \dfrac{\pa h_m^v}{\pa \beta_0}\dot\beta_0 + \dfrac{\pa h_m^u}{\pa \mu_0}\dot\mu_0\\
    &= \dfrac{\pa h_m^v}{\pa \alpha_0}(\mu_0 \alpha_0 - \omega_0 \beta_0)
     + \sum_{j \in \mathscr{I}\setminus\{0\}}\dfrac{\pa h_m^v}{\pa \alpha_j}\mu_j^+\alpha_j + \dfrac{\pa h_m^v}{\pa \beta_0}(\omega_0 \alpha_0 + \mu_0 \beta_0) + \mathcal{O}_3,
  \end{align}
and then, that yields 
  \begin{equation}\label{h}
    \begin{pmatrix}
      \dot h_m^u\\
      \dot h_m^v
    \end{pmatrix}
    = \omega_0(\alpha)
    \begin{pmatrix}
      \dfrac{\pa h_m^u}{\pa \beta_0} & -\dfrac{\pa h_m^u}{\pa \alpha_0}\\
      \dfrac{\pa h_m^v}{\pa \beta_0} & -\dfrac{\pa h_m^v}{\pa \alpha_0}
    \end{pmatrix}
    \begin{pmatrix}
      \alpha_0\\
      \beta_0
    \end{pmatrix}
    = M_{m, \alpha}
    \begin{pmatrix}
      h_m^u\\
      h_m^v
    \end{pmatrix}
    +
    \begin{pmatrix}
      F_m\\
      G_m
    \end{pmatrix}
    + \mathcal{O}_3.
  \end{equation}

  In the following, we obtain the polynomial approximations for $h_3^u$, $h_3^v$, $h_4^u$ and $h_4^v$ up to the quadratic terms.
Put
  \begin{equation*}
    M_{m, \alpha} =
    \begin{pmatrix}
      M_m^{11} & M_m^{12}\\
      M_m^{21} & M_m^{22}
    \end{pmatrix} \ (m \in \mathbb{Z}), \quad 
    T_{i} = 
    \begin{pmatrix}
      T_{i}^{11} & T_{i}^{12}\\
      T_{i}^{21} & T_{i}^{22}
    \end{pmatrix} \ (i = \pm1, \pm2).
  \end{equation*}
  Then, we have
  \begin{align*}
  h_m^u &= \dfrac{1}{\det M_m} \bigg\{ \omega_0 \left[ \left(M_m^{22}\dfrac{\pa h_m^u}{\pa \beta_0} - M_m^{12}\dfrac{\pa h_m^v}{\pa \beta_0}\right)\alpha_0 +\left(-M_m^{22}\dfrac{\pa h_m^u}{\pa \alpha_0} + M_m^{12}\dfrac{\pa h_m^v}{\pa \alpha_0}\right)\beta_0 \right] \\
  & - M_m^{22}F_m + M_m^{12}G_m \bigg\} + \mathcal{O}_3,\\
  h_m^v &= \dfrac{1}{\det M_m} \bigg\{ \omega_0 \left[ \left(-M_m^{21}\dfrac{\pa h_m^u}{\pa \beta_0} + M_m^{11}\dfrac{\pa h_m^v}{\pa \beta_0}\right)\alpha_0 + \left(M_m^{21}\dfrac{\pa h_m^u}{\pa \alpha_0} - M_m^{11}\dfrac{\pa h_m^v}{\pa \alpha_0}\right)\beta_0 \right] \\
  & + M_m^{21}F_m - M_m^{11}G_m \bigg\} + \mathcal{O}_3
  \end{align*}
  from \eqref{h}.
 Here, the invariance of the equations under the mapping \eqref{o2} implies that the expansions for $m = 3, 4$ are of the form
 \begin{equation}
     \begin{pmatrix}
         h_3^u \\ h_3^v
     \end{pmatrix}
     =
     \begin{pmatrix}
         U_{0011}^3 \\ V_{0011}^3
     \end{pmatrix}
      \alpha_1 \alpha_2 + \mathcal{O}_3, \quad 
      \begin{pmatrix}
         h_4^u \\ h_4^v
     \end{pmatrix}
     =
     \begin{pmatrix}
         U_{0002}^4 \\ V_{0002}^4
     \end{pmatrix}
      \alpha_2^2 + \mathcal{O}_3,
 \end{equation}
 and the coefficients are given by
\begin{align*}
  U_{0011}^3 &= -\dfrac{M_3^{22} f_{1,2}^{1,1} - M_3^{12} g_{1,2}^{1,1}}{\det M_{3, \alpha}}, 
  & V_{0011}^3 &= \dfrac{M_3^{21} f_{1,2}^{1,1} - M_3^{11} g_{1,2}^{1,1}}{\det M_{3, \alpha}},\\
  U_{0002}^4 &= -\dfrac{M_4^{22} f_{2,2}^{1,1} - M_4^{12} g_{2,2}^{1,1}}{2 \det M_{4, \alpha}}, 
  & V_{0002}^4 &= \dfrac{M_4^{21} f_{2,2}^{1,1} - M_4^{11} g_{2,2}^{1,1}}{2 \det M_{4, \alpha}},
\end{align*}
where
 \begin{align*}
  f_{m,n}^{i,j} &:= f_{uu} T_m^{1i} T_n^{1j} + f_{uv} (T_m^{1i} T_n^{2j} + T_m^{2i} T_n^{1j}) + f_{vv} T_m^{2i} T_n^{2j},\\
  g_{m,n}^{i,j} &:= g_{uu} T_m^{1i} T_n^{1j} + g_{uv} (T_m^{1i} T_n^{2j} + T_m^{2i} T_n^{1j}) + g_{vv} T_m^{2i} T_n^{2j}.
\end{align*}
The quadratic approximations of $h_i^\beta$ $(i = \pm1, \pm2)$ can be obtained similarly:
  \begin{align*}
    h_1^\beta &= (B_{1010}^1 \alpha_0 + B_{0110}^1 \beta_0) \alpha_1 + B_{0011}^1 \alpha_1 \alpha_2 + \mathcal{O}_3,\\
    h_2^\beta &= B_{0020}^2 \alpha_1^2 + (B_{1001}^2 \alpha_0 + B_{0101}^2 \beta_0) \alpha_2 + \mathcal{O}_3,
  \end{align*}
  where
  \begin{align*}
   B_{1010}^1 &= \dfrac{\mu_1^{-} (T_1^{21} f_{0,1}^{1,1} - T_1^{11} g_{0,1}^{1,1}) + \omega_0 (T_1^{21} f_{0,1}^{1,2} - T_1^{11} g_{0,1}^{1,2})}{\det T_1\{ (\mu_1^{-})^2 + \omega_0^2 \}},\\
   B_{0110}^1 &= \dfrac{\mu_1^{-} (T_1^{21} f_{0,1}^{1,2} - T_1^{11} g_{0,1}^{1,2}) - \omega_0 (T_1^{21} f_{0,1}^{1,1} - T_1^{11} g_{0,1}^{1,1})}{\det T_1\{ (\mu_1^{-})^2 + \omega_0^2 \}},\\
   B_{0011}^1 &= \dfrac{T_1^{21} f_{1,2}^{1,1} - T_1^{11} g_{1,2}^{1,1}}{\mu_1^{-} \det T_1},\\
   B_{0020}^2 &= \dfrac{T_2^{21} f_{1,1}^{1,1} - T_2^{11} g_{1,1}^{1,1}}{2 \mu_2^{-} \det T_2},\\  
   B_{1001}^2 &= \dfrac{\mu_2^{-} (T_2^{21} f_{0,2}^{1,1} - T_2^{11} g_{0,2}^{1,1}) + \omega_0 (T_2^{21} f_{0,2}^{2,1} - T_2^{11} g_{0,2}^{2,1})}{\det T_2\{ (\mu_2^{-})^2 + \omega_0^2 \}},\\
  B_{0101}^2 &= \dfrac{\mu_2^{-} (T_2^{21} f_{0,2}^{2,1} - T_2^{11} g_{0,2}^{2,1}) - \omega_0 (T_2^{21} f_{0,2}^{1,1} - T_2^{11} g_{0,2}^{1,1})}{\det T_2\{ (\mu_1^{-})^2 + \omega_0^2 \}}.
  \end{align*}
  Finally, substituting $h_{3}^{u}$, $h_{3}^{v}$, $h_{4}^{u}$, $h_{4}^{v}$, $h_{\pm 1}^{\beta}$ and $h_{\pm 2}^{\beta}$ into \eqref{eq:tth-diagonal}, we obtain \eqref{eq:reduced}.
  The explicit form of coefficients of \eqref{eq:reduced} are shown in Appendix~\ref{sec:apA}.
\end{proof}


Next, we rewrite the reduced system \eqref{eq:reduced} using standard method (see the section~3.5 of \cite{Kuz}) and some near-identity transformations.
Set
\begin{equation}
  A_{0, \alpha} =
  \begin{pmatrix}
    \mu_0 & -\omega_0\\
    \omega_0 & \mu_0
  \end{pmatrix}.
\end{equation}
Let $\bs{q}(\alpha) \in \mathbb{C}^2$ be an eigenvector of $A_{0, \alpha}$ corresponding to the eigenvalue $\lambda_0$ such as $A_{0, \alpha}\bs{q}(\alpha) = \lambda_0 \bs{q}(\alpha)$.
Also let $\bs{p}(\alpha) \in \mathbb{C}^2$ be an eigenvector of the transposed matrix $A_0^t(\alpha)$ corresponding to its eigenvalue $\bar\lambda_0$ such as $A_{0, \alpha}^t \bs{p}(\alpha) = \bar\lambda_0 \bs{p}(\alpha)$.
By setting
\begin{equation}
  \bs{p} (\alpha) = \begin{pmatrix} p_1 \\ p_2 \end{pmatrix} = \dfrac{1}{2\omega_0}
  \begin{pmatrix}
    1 \\ - \mathrm{i}
  \end{pmatrix},
  \quad \bs{q} (\alpha) = \begin{pmatrix} q_1 \\ q_2 \end{pmatrix} = \omega_0
  \begin{pmatrix}
    1 \\ - \mathrm{i}
  \end{pmatrix},
\end{equation}
we can normalize $\bs{p}(\alpha)$ with respect to $\bs{q}(\alpha)$, namely $\langle \bs{p}(\alpha), \bs{q}(\alpha)\rangle = 1$, where $\langle\cdot,\cdot\rangle$ means the standard scalar product in $\mathbb{C}^2$: $\langle \bs{p}(\alpha), \bs{q}(\alpha) \rangle = \bar p_1 q_1 + \bar p_2 q_2$.
We introduce a new variable $z(t) \in \mathbb{C}$ by $z = \langle \bs{p}(\alpha), (\alpha_0, \beta_0)^t\rangle$.
By the above representation, any vector $(\alpha_0, \beta_0)^t$ can be uniquely represented as $(\alpha_0, \beta_0)^t = z \bs{q}(\alpha) + \bar z \bar{\bs{q}}(\alpha)$ (\cite{Kuz}).
The direct computation yields
  \begin{equation}\label{eq:ourcase}
    \begin{cases}
      \dot z = \lambda_0 z + E_1 z^2 + E_2 |z|^2 + E_3 \bar z^2 + E_4 \alpha_1^2 + E_5 \alpha_2^2 + e_1 z^3\\
      \quad  + (e_2 z + e_3 \bar z)|z|^2 + e_4 \bar z^3 + (e_5 z + e_6 \bar z) \alpha_1^2 + (e_7 z + e_8 \bar z) \alpha_2^2\\
      \quad + e_9 \alpha_1^2 \alpha_2 
      + \mathcal{O}_{4},\\
      \dot \alpha_1 = \mu_1^+ \alpha_1 + (H_1 z + \bar H_1 \bar z) \alpha_1 + C_3 \alpha_1 \alpha_2 + (H_2 z^2 + \bar H_2 \bar z^2) \alpha_1 \\
      \quad + H_3 |z|^2\alpha_1 + (H_4 z + \bar H_4 \bar z) \alpha_1 \alpha_2 + (c_4 \alpha_1^2 + c_5 \alpha_2^2) \alpha_1 + \mathcal{O}_{4},\\
      \dot \alpha_2 = \mu_2^+ \alpha_2 + D_1 \alpha_1^2 + (I_1 z + \bar I_1 \bar z) \alpha_2 + (I_2 z^2  + \bar I_2 \bar z^2) \alpha_2 + I_3 |z|^2\alpha_2 \\
      \quad + (I_4 z + \bar I_4 \bar z) \alpha_1^2 + (d_4 \alpha_1^2 + d_5 \alpha_2^2) \alpha_2 + \mathcal{O}_{4}.
    \end{cases}
  \end{equation}
where $E_j, e_j, H_j, I_j \in \mathbb{C}$.
The explicit form of the coefficients are listed in Appendix~\ref{sec:apB}.

We now eliminate the quadratic and cubic terms as possible.
By applying some near-identity transformations given in Appendix~\ref{sec:apC}, we end up with the normal form for the Hopf--Turing--Turing bifurcation:
\begin{theorem}\label{thm:1}
For given functions $f$ and $g$ satisfying Assumption~\ref{ass:rd} and a positive constant $L$, there exists a positive constant $\varepsilon$ such that the local center manifold $\mathcal{M}^{c}_{loc}$ of the dynamical system \eqref{eq:infinite} is contained in $\mathcal{N}_{\varepsilon}$.
Furthermore, the dynamics of \eqref{eq:infinite} on $\mathcal{M}^{c}_{loc}$ is locally topologically equivalent to the dynamics given by the normal form for the Hopf--Turing--Turing bifurcation:
\begin{equation}\label{eq:normalform}
 \begin{cases}
   \dot z_0 = (\lambda_0 + \tilde a_0 |z_0|^2 + \tilde a_1 z_1^2 + \tilde a_2 z_2^2)z_0 + \mathcal{O}_4,\\
   \dot z_1 = (\mu_1 + \tilde b_0 |z_0|^2 + \tilde b_1 z_1^2 + \tilde b_2 z_2^2) z_1 + \tilde B z_1 z_2 + \mathcal{O}_4,\\
   \dot z_2 = (\mu_2 + \tilde c_0 |z_0|^2 + \tilde c_1 z_1^2 +  \tilde c_2 z_2^2)z_2 + \tilde C z_1^2 + \mathcal{O}_4,
 \end{cases}
\end{equation}
where $\mu_1 = \mu_1^+$, $\mu_2 = \mu_2^+$, $(z_0, z_1, z_2) \in \mathbb{C} \times \mathbb{R}^2$, $\tilde a_j \in \mathbb{C}$, $\tilde b_j \in \mathbb{R}$, $\tilde c_j \in \mathbb{R}$ $(j = 0,1,2)$, $\tilde B \in \mathbb{R}$, $\tilde C \in \mathbb{R}$ and $\mathcal{O}_4 = \mathcal{O}(|(z_{0}, \bar z_{0}, z_1, z_2)|^{4})$.
\end{theorem}
This theorem is one of main results in the present paper.
The proof and the explicit form of the coefficients are given in Appendix~\ref{sec:apC}.

\begin{remark}\label{rem:normal}
We have discussed the normal form for the Hopf--Turing--Turing bifurcation with the $0:1:2$ mode interaction. 
The Hopf--Turing--Turing bifurcation with the $0:m:m+1$ mode interaction $(m\geq 2)$ can be also treated similarly. 
Then, the normal form for this type becomes the form of
\begin{equation*}
 \begin{cases}
   \dot z_0 = (\lambda_0 + \tilde a_0 |z_0|^2 + \tilde a_1 z_m^2 + \tilde a_2 z_{m+1}^2)z_0 + \mathcal{O}_4,\\
   \dot z_m = (\mu_m + \tilde b_0 |z_0|^2 + \tilde b_1 z_m^2 + \tilde b_2 z_{m+1}^2) z_m  + \mathcal{O}_4,\\
   \dot z_{m+1} = (\mu_{m+1} + \tilde c_0 |z_0|^2 + \tilde c_1 z_m^2 + \tilde c_2 z_{m+1}^2)z_{m+1} + \mathcal{O}_4,
 \end{cases}
\end{equation*}
that is, the resonance terms are only dropped from \eqref{eq:normalform}. 
\end{remark}


\subsection{Dynamics on the center manifold}\label{sec:NF_anal}

By truncating up to the cubic terms and using the polar coordinate $z_0 = r_0 \e^{\I \theta_0}$, we transform the normal form of the Hopf--Turing--Turing bifurcation \eqref{eq:normalform} into
\begin{equation}\label{eq:NF_neumann}
  \begin{cases}
    \dot r_0 = (\mu_0 + (\mathrm{Re}\,\tilde a_0) r_0^2 + (\mathrm{Re}\,\tilde a_1) z_1^2 + (\mathrm{Re}\,\tilde a_2) z_2^2) r_0,\\ 
    \dot z_1 = (\mu_1 + \tilde b_0 r_0^2 + \tilde b_1 z_1^2 + \tilde b_2 z_2^2) z_1 + \tilde B z_1 z_2,\\
    \dot z_2 = (\mu_2 + \tilde c_0 r_0^2 + \tilde c_1 z_1^2 + \tilde c_2 z_2^2) z_2 + \tilde C z_1^2,\\
    \dot \theta_0 = \omega_0 + \mathcal{O}_2.
  \end{cases}
  \end{equation}
The above system can be simplified by the rescaling
\begin{align*}
l &= \dfrac{|\tilde c_2|}{\tilde B^2}, \ t \mapsto l t, \ r_0 \mapsto \sqrt{\dfrac{1}{l |\mathrm{Re}\, \tilde a_0|}} r_0, \ z_1 \mapsto \sqrt{\dfrac{|\tilde B|}{l |\tilde c_2 \tilde C|}} z_1, \ z_2 \mapsto \dfrac{1}{l \tilde B} z_2, \\
\mu_j &\mapsto l \mu_j \ (j = 0, 1, 2),
\end{align*}
Because $\omega_0 \neq 0$ near the Hopf--Turing--Turing bifurcation point, the angular component can be decoupled to obtain the following reduced system 
\begin{equation}\label{eq:012NF}
\begin{cases}
    \dot r_0 = (\mu_0 + \sigma_1 r_0^2 + d_{01} z_1^2 + d_{02} z_2^2) r_0,\\
    \dot z_1 = (\mu_1 + d_{10} r_0^2 + d_{11} z_1^2 + d_{12} z_2^2) z_1 + z_1 z_2,\\
    \dot z_2 = (\mu_2 + d_{20} r_0^2 + d_{21} z_1^2 + \sigma_2 z_2^2) z_2 + \sigma_3 z_1^2,
\end{cases}
\end{equation}
where 
\begin{align*}
\sigma_1 &= \mathrm{sign}(\mathrm{Re}\, \tilde a_0), &
d_{01} &= \sqrt{\dfrac{|\tilde B|}{|\tilde c_2 \tilde C|}} \mathrm{Re}\, \tilde a_1, &
d_{02} &= \dfrac{\mathrm{Re}\, \tilde a_2 }{|\tilde c_2|}, \\ 
d_{10} &= \dfrac{\tilde b_0}{|\mathrm{Re}\, \tilde a_0|}, &
d_{11} &= \dfrac{\tilde b_1 |\tilde B|}{|\tilde c_2 \tilde C|}, &
d_{12} &= \dfrac{\tilde b_2}{|\tilde c_2|}, \\
d_{20} &= \dfrac{\tilde c_0}{|\mathrm{Re}\, \tilde a_0|}, &
d_{21} &= \dfrac{\tilde c_1 |\tilde B|}{|\tilde c_2 \tilde C|}, &
\sigma_2 &= \mathrm{sign}(\tilde c_2), \\
\sigma_3 &= \mathrm{sign}(\tilde B \tilde C).
\end{align*}

For the system \eqref{eq:012NF}, there are possibly four equilibria as follows:
\begin{align*}
\mathrm{O} &= (0, 0, 0), & \text{for all} \ \mu_0, \mu_1, \mu_2 \\
\mathrm{PM}_0 &= \left(\sqrt{- \dfrac{\mu_0}{\sigma_1}}, 0, 0 \right), & \text{for} \quad \mu_0 \sigma_1 < 0,\\
\mathrm{PM}_{2}^{\pm} &=  \left(0, 0, \pm\sqrt{-\dfrac{\mu_{2}}{\sigma_2}} \right), & \text{for} \quad \mu_{2} \sigma_2 < 0,\\
\mathrm{MM}_{0:2}^{\pm} &= \left( \sqrt{\dfrac{d_{02} \mu_2 - \sigma_2 \mu_0}{\Delta}}, 0, \pm\sqrt{\dfrac{d_{20} \mu_0 - \sigma_1 \mu_2}{\Delta}} \right), &\\
 &\text{for} \quad \dfrac{d_{02} \mu_2 - \sigma_2 \mu_0}{\Delta} > 0 \ \text{and} \ \dfrac{d_{20} \mu_0 - \sigma_1 \mu_2}{\Delta} > 0,
\end{align*}
where $\Delta := \sigma_1 \sigma_2 - d_{02} d_{20}$.
The symbols O, PM$_{n}$ and MM$_{m:n}$ mean the trivial solution, pure mode and $m:n$-mixed mode, respectively.

PM$_0$ corresponds to the spatially homogeneous time periodic solution in the original reaction-diffusion system \eqref{2rd}.
If $\sigma_1 < 0$, then the locally asymptotically stable periodic solution bifurcates from the trivial state O through the super-critical Hopf bifurcation at $\mu_0 = 0$.

PM$_{2}^{\pm}$ corresponds to the 2-mode stationary solution in \eqref{2rd}.
The 2-mode stationary solutions bifurcate from the trivial state O through the pitchfork bifurcation at $\mu_2 = 0$, and they exist when $\mu_2 \sigma_2 < 0$.
Pure modes of other kind, i.e., with $r_0 = z_2 = 0$ and $z_1 \neq 0$, are not possible.

The $0:2$ mixed mode equilibria denoted by MM$_{0:2}^{\pm}$ correspond to time periodic solutions with a constant 2-mode amplitude.
The leading term of the solution is formally written by
\begin{align*}
u = \tilde u + \varepsilon_1 \Theta(t) + \varepsilon_2 \cos \dfrac{2\pi x}{L}, \quad 
v = \tilde v + \varepsilon_3 \Theta(t) + \varepsilon_4 \cos \dfrac{2\pi x}{L},
\end{align*}
where $\Theta(t) = \mathcal{O}(1)$ is a periodic function with period approximately $2\pi/\omega_0$ and $\varepsilon_j = \mathcal{O}(\sqrt{|\mu_0|}, \sqrt{|\mu_2|})$.
The instability and bifurcation from MM$_{0:2}^{\pm}$ will be discussed in Section~\ref{sec:02}.

The systems \eqref{eq:NF_neumann} and \eqref{eq:012NF} are invariant under the mappings $(r_0, z_1, z_2) \mapsto (- r_0, z_1, z_2)$ and $(r_0, z_1, z_2) \mapsto ( r_0, - z_1, z_2)$, therefore, $\mathcal{S}_1 := \{(r_0, z_1, z_2) \in \mathbb{R}^3;\, r_0 = 0\}$ and $\mathcal{S}_2 := \{(r_0, z_1, z_2) \in \mathbb{R}^3;\, z_1 = 0\}$ are invariant subspaces.
In the following, we mention the instability of mixed mode equilibria in $\mathcal{S}_1$ and $\mathcal{S}_2$.
The study to full dynamics and bifurcation structures of \eqref{eq:NF_neumann} on $\mathbb{R}^4$ is our future work.


\subsubsection{Bifurcation from a 1:2 mixed mode equilibrium}\label{sec:12}

On the invariant subspace $\mathcal{S}_1$, the system \eqref{eq:012NF} is reduced to
\begin{equation}\label{eq:12resonance}
  \begin{cases}
    \dot z_1 = (\mu_1 + d_{11} z_1^2 + d_{12} z_2^2) z_1 + z_1 z_2,\\
    \dot z_2 = (\mu_2 + d_{21} z_1^2 + \sigma_2 z_2^2) z_2 + \sigma_3 z_1^2.
  \end{cases}
\end{equation}
The above system has the spatial 1:2 resonance with O(2) symmetry, and has already been studied extensively (\cite{D,AGH,PK,SMH,IK}).
These authors identified the presence of equilibria, periodic orbits and structurally stable heteroclinic cycles.
In this subsection, we mention the interaction between the Hopf bifurcation from a $1:2$ mixed mode equilibrium (MM$_{1:2}$) and the Hopf bifurcation of spatially uniform oscillation.
We will see the occurrence of the Hopf--Hopf bifurcation, i.e. 3-torus may be generated.

Following the reference \cite{IK}, we briefly describe the computation of MM$_{1:2}$ which has the Hopf instability.
Suppose that $z_1 z_2 \neq 0$ and MM$_{1:2}$ has a form $z_1 = \rho z_2$ $(\rho \in \mathbb{R}\setminus\{0\})$.
Then we have
\begin{align*}
0 &= \mu_1 + (d_{11} \rho^2 z_2 + d_{12} z_2 + \rho ) z_2,\\
0 &= \mu_2 + (d_{21} \rho^2 z_2 + \sigma_2 z_2 + \sigma_3 \rho^2) z_2.
\end{align*}
 The linearized matrix at MM$_{1:2}$ is given by
 \[
 M = 
 \begin{pmatrix}
     \mu_0 + (d_{01} \rho^2 + d_{02}) z_2^2 & 0 & 0 \\
    0 & 2 d_{11} \rho^2 z_2^2 & (2 d_{12} z_2 + 1) \rho z_2 \\
    0 & 2 \rho z_2 (d_{21} z_2 + \sigma_3) & (2 \sigma_2 z_2 - \sigma_3 \rho^2) z_2
 \end{pmatrix}.
 \]
 Then, the linearized matrix $M$ has a pair of purely imaginary eigenvalues and a zero eigenvalue if and only if $\mu_0 = -(\rho^2 d_{01} + d_{02}) z_2^2$, 
$2 (d_{11} \rho^2 + \sigma_2) z_2 - \sigma_3 \rho^2 = 0$ and 
$d_{11}  (2 \sigma_2 z_2 - \sigma_3 \rho^2)z_2 - (2 d_{12} z_2 + 1)  (d_{21} z_2 + \sigma_3) > 0$ hold, simultaneously.
We thus obtain the Hopf--Hopf bifurcation point for $\sigma_3 = -1$ as follows:
 \begin{align*}
   z_2 &= z_2^{*} = -\dfrac{\rho^2}{2(d_{11} \rho^2 + \sigma_2)},\\
   z_1 &= z_1^* = \rho z_2^*,\\
   \mu_0 &= \mu_0^* = - \dfrac{(\rho^2 d_{01} + d_{02}) \rho^4}{4(\sigma_2 + d_{11} \rho^2)^2 },\\
   \mu_1 &= \mu_1^* = \dfrac{\rho^2 \left\{ 2\rho (\sigma_2 + d_{11} \rho^2) - (d_{12} + d_{11}\rho^2) \rho^2 \right\}}{4(\sigma_2 + d_{11} \rho^2)^2}, \\
   \mu_2 &= \mu_2^* = - \dfrac{\rho^4 \left\{ \rho^2 (2d_{11} + d_{21}) + 3 \sigma_2 \right\}}{4(\sigma_2 + d_{11} \rho^2)^2}.
 \end{align*}

Hereafter, we put $\sigma_3 = -1$.
The translation to the origin by means of $\nu_0 = \mu_0 - \mu_0^*$, $Z_1 = z_1 - z_1^*$ and $Z_2 = z_2 - z_2^*$, transforms the system \eqref{eq:012NF} into
 \begin{equation}\label{eq:r0Z1Z2}
 \begin{cases}
 \dot r_0 = \nu_0 r_0 + \tilde F_0(r_0, Z_1, Z_2),\\
 \dot Z_1 = \nu_{11} Z_1 + \nu_{12} Z_2 + \tilde F_1(r_0, Z_1, Z_2),\\
 \dot Z_2 = \nu_{21} Z_1 + \nu_{22} Z_2 + \tilde F_2(r_0, Z_1, Z_2), 
 \end{cases}
 \end{equation}
 where $\nu_{11} = 2 d_{11} (z_1^*)^2$, $\nu_{12} = z_1^* (2 d_{12} z_2^* + 1)$, $\nu_{21} = 2 z_1^* (d_{21} z_2^* - 1)$, $\nu_{22} = z_2^* (2 \sigma_2 z_2^* + \rho^2)$ and
 \begin{align*}
\tilde F_0(r_0, Z_1, Z_2) &= 2 (z_1^* d_{01} Z_1 + z_2^* d_{02} Z_2) r_0 + (\sigma_1 r_0^2 + d_{01} Z_1^2 + d_{02} Z_2^2) r_0,\\
\tilde F_1(r_0, Z_1, Z_2) &= d_{10} z_1^* r_0^2 + 3 d_{11} z_1^* Z_1^2 + (2 d_{12} z_2^* + 1) Z_1 Z_2 + d_{12} z_1^* Z_2^2\\
  &+ (d_{10} r_0^2 + d_{11} Z_1^2 + d_{12} Z_2^2) Z_1,\\
\tilde F_2(r_0, Z_1, Z_2) &= d_{20} z_2^* r_0^2 + (d_{21} z_2^* - 1) Z_1^2 + 2 d_{21} z_1^* Z_1 Z_2 + 3 \sigma_2 z_2^* Z_2^2\\
   &+ (d_{20} r_0^2 + d_{21} Z_1^2 + \sigma_2 Z_2^2 ) Z_2.
 \end{align*}
By applying the results given in the section~4.3.2 of \cite{KS} to \eqref{eq:r0Z1Z2}, we get the following system:
 \begin{align*}
 \begin{cases}
 \dot \rho_0 = (\nu_0 + J_1 \rho_0^2 + J_2 |w_{1,2}|^2) \rho_0 + \mathcal{O}_4(\rho_0, w_{1,2}, \bar w_{1,2}),\\
 \dot w_{1,2} = (\lambda_{1,2} + K_1 \rho_0^2 + K_2 |w_{1,2}|^2) w_{1,2} + \mathcal{O}_4(\rho_0, w_{1,2}, \bar w_{1,2}),
 \end{cases}
\end{align*}
where $J_1, J_2 \in \mathbb{R}$, $K_1, K_2 \in \mathbb{C}$, $\rho_0 = \rho_0(t) \in \mathbb{R}$, $w_{1,2} = w_{1,2}(t) \in \mathbb{C}$, $\lambda_{1,2} = \nu_{1,2} + \I \omega_{1,2}$, $\nu_{1,2} = (\nu_{11} + \nu_{22})/2$ and $\omega_{1,2} = \sqrt{4(\nu_{11} \nu_{22} - \nu_{12} \nu_{21}) - (\nu_{11} + \nu_{22})^2}/2$.
Note that $\rho_0(t)$ and $w_{1,2}(t)$ correspond to $x(t)$ and $w(t)$ of the system (4.17) in \cite{KS}, respectively, and that as we seen in Section~\ref{sec:NF}, the variables $w_{1,2}(t) \in \mathbb{C}$ is introduced as the complex amplitude induced by the Hopf bifurcation between 1 and 2 mode interaction.
The explicit forms of $J_1$, $J_2$, $K_1$ and $K_2$ are
\begin{align*}
J_1 &= \sigma_1 - 2 \mathrm{Re}\, \left( \dfrac{U_1 V_1}{\lambda_{1,2}} \right), & J_2 &= V_2 - 2 \mathrm{Re}\, \left( \dfrac{U_2 V_1}{\lambda_{1,2}} \right),\\
K_1 &= U_3 + \dfrac{\bar U_1 U_2 - U_1 U_4}{\lambda_{1,2}} - \dfrac{U_1 V_1}{2 \nu_0 - \lambda_{1,2}}, & K_2 &= U_5 + \dfrac{|U_2|^2 - U_2 U_4}{\lambda_{1,2}} - \dfrac{2 |\tilde U_4|^2}{2 \bar \lambda_{1,2} - \lambda_{1,2}},
\end{align*}
where
\begin{align*}
U_j &= \dfrac{1}{2 \mathrm{i} \omega_{1,2}} \left( \dfrac{\nu_{11} + \mathrm{i} \omega_{1,2}}{\nu_{12}} P_j + Q_j \right), \quad \tilde U_4 = \dfrac{1}{2 \mathrm{i} \omega_{1,2}} \left( \dfrac{\nu_{11} + \mathrm{i} \omega_{1,2}}{\nu_{12}} \bar P_4 + \bar Q_4 \right),\\
V_1 &= 2 \left\{ z_1^* d_{01} \nu_{12} - z_2^* d_{02} (\nu_{11} - \mathrm{i} \omega_{1,2}) \right\} \\
V_2 &= 2 \nu_{12} (d_{01} \nu_{12} - d_{02} \nu_{21})\\
P_1 &= d_{10} z_1^* \\
P_2 &= 2 \left\{3 d_{11} z_1^* \nu_{12}^2 + d_{12} z_1^* (\nu_{11}^2 + \omega_{1,2}^2) - (2 d_{12} z_2^* + 1) \nu_{11} \nu_{12} \right\}\\
P_3 &= \nu_{12} d_{10}\\
P_4 &= 3 \nu_{12}^2 d_{11} z_1^* + (2 \nu_{11}^2 - 2 \nu_{11} \mathrm{i} \omega_{1,2} + \nu_{12} \nu_{21}) d_{12} z_1^* + \nu_{12} (2 d_{12} z_2^* + 1) (\mathrm{i} \omega_{1,2} - \nu_{11}),\\
P_5 &= \nu_{12} \left\{ 3 d_{11} \nu_{12}^2 + d_{12} (3 \nu_{11}^2 + \omega_{1,2}^2 - 2 \mathrm{i} \omega_{1,2} \nu_{11} ) \right\}\\
Q_1 &= d_{20} z_2^*,\\
Q_2 &= 2 \left\{ (d_{21} z_2^* - 1) \nu_{12}^2 + 3 \sigma_2 z_2^* (\nu_{11}^2 + \omega_{1,2}^2) - 2 d_{21} z_1^* \nu_{11} \nu_{12} \right\}\\
Q_3 &= d_{20} (\mathrm{i} \omega_{1,2} - \nu_{11})\\
Q_4 &= \nu_{12}^2 (d_{21} z_2^* - 1) + (2 \nu_{11}^2 - 2 \nu_{11} \mathrm{i} \omega_{1,2} + \nu_{12} \nu_{21}) 3 \sigma_2 z_2^* + 2 d_{21} z_1^* \nu_{12} (\mathrm{i} \omega_{1,2} - \nu_{11}),\\
Q_5 &= d_{21} \nu_{12}^2 (\mathrm{i} \omega_{1,2} - 3 \nu_{11}) + 3 \sigma_2 (\mathrm{i} \omega_{1,2} - \nu_{11})(\omega_{1,2}^2 + \nu_{11}^2).
\end{align*}
Moreover, by using the polar coordinate $w_{1,2} = \rho_{1,2}(t) \e^{\I \theta_{1,2}(t)}$, recovering the azimuthal component of $z_0$, and truncating up to third-order, we obtain 
\begin{equation}\label{eq:hopf-hopf2}
\begin{cases}
\dot \rho_0 = (\nu_0 + J_1 \rho_0^2 +  J_2 \rho_{1,2}^2) \rho_0,\\ 
\dot \rho_{1,2} = (\nu_{1,2} + \mathrm{Re}\, K_1 \rho_0^2 + \mathrm{Re}\, K_2 \rho_{1,2}^2 ) \rho_{1,2},\\ 
\dot \theta_0 = \omega_0 + \mathcal{O}_2(\rho_0, \rho_{1,2}, \theta_0, \theta_{1,2}),\\
\dot \theta_{1,2} = \omega_{1,2} + \mathcal{O}_2(\rho_0, \rho_{1,2}, \theta_0, \theta_{1,2}),
\end{cases}
\end{equation}
which is equivalent to the normal form for the Hopf--Hopf bifurcation.

According to the computation given in \cite{GH}, we rescale \eqref{eq:hopf-hopf2} by $(\rho_0, \rho_{1,2}) \mapsto (\rho_0/\sqrt{|J_1|}, \rho_{1,2}/\sqrt{|\mathrm{Re}\, K_2|})$.
If $J_1 > 0$ holds, then by setting $b = J_2/|\mathrm{Re}\, K_2|, \ c = \mathrm{Re}\, K_1/|J_1|, \ d = \sigma(\mathrm{Re}\, K_2)$, we have
\begin{equation}\label{eq:hopf-hopf3}
\begin{cases}
\dot \rho_0 = (\nu_0 + \rho_0^2 + b \rho_{1,2}^2) \rho_0,\\
\dot \rho_{1,2} = (\nu_{1,2} + c \rho_0^2 + d \rho_{1,2}^2) \rho_{1,2},\\
\dot \theta_0 = \omega_0 + \mathcal{O}_2,\\
\dot \theta_{1,2} = \omega_{1,2} + \mathcal{O}_2.
\end{cases}
\end{equation}
While if $J_1 < 0$ holds, then by setting $b = - J_2/|\mathrm{Re}\, K_2|, \ c = \mathrm{Re}\, K_1/|J_1|, \ d = - \sigma(\mathrm{Re}\, K_2), \ \nu_0 \mapsto - \nu_0, \ \nu_{1,2} \mapsto - \nu_{1,2}, \ t \mapsto -t$, we obtain the same system.
Once more, by ignoring the azimuthal components, we get the reduced planer system
\begin{equation}\label{eq:hopf-hopf_plane}
\begin{cases}
\dot \rho_0 = (\nu_0 + \rho_0^2 + b \rho_{1,2}^2) \rho_0, \\
\dot \rho_{1,2} = (\nu_{1,2} + c \rho_0^2 + d \rho_{1,2}^2) \rho_{1,2}.
\end{cases}
\end{equation}
For the system \eqref{eq:hopf-hopf_plane}, there are possibly four equilibria as follows:
\begin{align*}
\tilde E_1 &= (0, 0), & \text{for all} \ \nu_0, \ \nu_{1,2},\\
\tilde E_2 &= (\sqrt{- \nu_0}, 0), & \text{for} \ \nu_0 < 0,\\
\tilde E_3 &= \left(0, \sqrt{-\dfrac{\nu_{1,2}}{d}} \right), & \text{for} \ \nu_{1,2} d < 0,\\
\tilde E_4 &= \left( \sqrt{\dfrac{b \nu_{1,2} - d \nu_0}{\tilde\Delta}}, \sqrt{\dfrac{c \nu_0 - \nu_{1,2}}{\tilde\Delta}} \right) 
&\text{for} \ \dfrac{b \nu_{1,2} - d \nu_0}{\tilde\Delta} > 0 \ \text{and} \ \dfrac{c \nu_0 - \nu_{1,2}}{\tilde\Delta} > 0,
\end{align*}
where $\tilde\Delta = d- bc$.
The equilibrium $\tilde E_1$ is $1:2$ mixed mode stationary solution, i.e., MM$_{1:2}$.
The equilibrium $\tilde E_2$ corresponds to the 0-mode which induces time periodic solutions, whose leading term is formally written by
\begin{align*}
u &= \tilde u + \varepsilon_1 \Theta(t) + \tilde \varepsilon_1 \cos \dfrac{\pi x}{L} + \tilde \varepsilon_2 \cos \dfrac{2 \pi L}{x},\\
v &= \tilde v + \varepsilon_2 \Theta(t) + \tilde \varepsilon_3 \cos \dfrac{\pi x}{L} + \tilde \varepsilon_4 \cos \dfrac{2 \pi L}{x},
\end{align*}
where $\varepsilon_{j} = \mathcal{O}(\sqrt{|\nu_0|})$ $(j = 1,2)$ and $\tilde\varepsilon_j = \mathcal{O}(\rho^2)$ $(j = 1,2,3,4)$.
The equilibrium $\tilde E_3$ is the limit cycle which consists of the interaction between 1 and 2 mode solution.
$\tilde E_4$ is invariant two-torus, whose periods are approximately $2 \pi/\omega_0$ and $2 \pi/\omega_{1,2}$. 

\begin{table}[hbtp]
 \caption{Classification of the unfoldings of \eqref{eq:hopf-hopf_plane}.
 The notations ($-$) and ($+$) mean that the sign of $\tilde \Delta$ is fully determined for the given signs of $d$, $b$ and $c$ (see \cite{GH}).}\label{table:HH}
  \centering
  \begin{tabular}{lllllllllllll}
    \hline
    Case   & Ia   & Ib   & II   & III  & IVa    & IVb   & V    & VIa  & VIb   & VIIa   & VIIb & VIII \\
    \hline
    $d$    & $+1$ & $+1$ & $+1$ & $+1$ &  $+1$  & $+1$  & $-1$ & $-1$ &  $-1$ &  $-1$  & $-1$ & $-1$  \\
    $b$    &  $+$ &  $+$ & $+$  & $-$  &   $-$  &  $-$  & $+$  & $+$  &  $+$  &  $-$   & $-$  & $-$   \\
    $c$    &  $+$ &  $+$ & $-$  & $+$  &   $-$  &  $-$  & $+$  & $-$  &  $-$  &  $+$   & $+$  & $-$   \\
  $\tilde \Delta$ &  $+$ &  $-$ & $+$  & $+$  &   $+$  &  $-$  & $-$  & $+$  &  $-$  &  $+$   & $-$  & $-$   \\
    \hline
  \end{tabular}
\end{table}
As known in the literature (\cite{GH}), the system \eqref{eq:hopf-hopf_plane} has twelve distinct types of unfoldings, and the bifurcation diagrams and phase portraits were sketched in the section~7.5 of \cite{GH}.
For the convenience of application to the numerical example given in Section~\ref{sec:numerical}, we briefly mention Case VIa, that is, with $d = -1$, $b > 0$, $c < 0$ and $\tilde \Delta > 0$.

By simple calculation, we find that the equilibrium $\tilde E_4$ has Hopf instability point for $\nu_0 < 0$.
Indeed, the linearized matrix of \eqref{eq:hopf-hopf_plane} at $\tilde E_4$ is given by
\[
\begin{pmatrix}
\nu_0 + 3 \rho_0^2 + b \rho_{1, 2}^2 & 2 b \rho_0 \rho_{1,2} \\
2 c \rho_0 \rho_{1,2} & \nu_{1,2} + c \rho_0^2 + 3 d \rho_{1,2}^2
\end{pmatrix},
\]
and thus, the matrix has a pair of purely imaginary eigenvalues on the line $\nu_{1,2} = (c - 1)\nu_0/(b + 1)$.
However, the third-order approximation to the normal form is not enough to determine the stability of periodic orbits ($\tilde{\mathrm{PO}}$) bifurcating from $\tilde E_4$, since the system is integrable for the values on the Hopf instability.
Hence, \eqref{eq:hopf-hopf_plane} has a family of periodic orbits that ends in a degenerate heteroclinic cycle (see Figure~7.5.6 in \cite{GH}).
To unfold and calculate the curves of Hopf bifurcation and the heteroclinic cycle connecting $\tilde E_1$, $\tilde E_2$ and $\tilde E_3$ as well as their stability, we have to calculate the explicit form of the fifth order coefficients, that is, we need to compute the fifth order reduced system on the center manifold.
This derivation is quite hard because we need the approximations of the functions $h_m^u$, $h_m^v$ and $h_i^\beta$ up to the fourth order terms.

From the symmetric properties of the system \eqref{eq:hopf-hopf_plane}, we can formally restore the truncated quintic terms such as
\begin{equation}\label{eq:hopf-hopf_plane_quintic}
\begin{cases}
\dot \rho_0 = (\nu_0 + \rho_0^2 + b \rho_{1,2}^2) \rho_0 + (e_{11} \rho_0^4 + e_{12} \rho_0^2 \rho_{1,2}^2 + e_{13} \rho_{1,2}^4) \rho_0, \quad\\
\dot \rho_{1,2} = (\nu_{1,2} + c \rho_0^2 + d \rho_{1,2}^2) \rho_{1,2} + (e_{21} \rho_0^4 + e_{22} \rho_0^2 \rho_{1,2}^2 + e_{23} \rho_{1,2}^4) \rho_{1,2}.
\end{cases}
\end{equation}
By applying a degree three coordinate change appropriately, we can choose a coordinate system in which all but one of these six coefficients $e_{11}, \dots, e_{23}$ are zero.
Thus, we can take $e_{11} = e_{12} = e_{13} = e_{21} = e_{22} = 0$ and study the Hopf bifurcation from $\tilde E_4$ by considering the effects of the quintic term $e_{23} \rho_{1,2}^5$.
In Table~\ref{table:1}, we summarize the correspondence of solutions between the planar system \eqref{eq:hopf-hopf_plane}, the four-dimensional system \eqref{eq:NF_neumann} and the reaction-diffusion system \eqref{2rd}.
Note that there are two rotations $\dot \theta_0 \approx \omega_0$ and $\dot \theta_{1,2} \approx \omega_{1,2}$ which should be restored to elucidate the dynamics of the solutions.
\begin{table}[hbtp]
 \caption{Classification of equilibria and solution correspondence.}\label{table:1}
\centering
  \begin{tabular}{lll}
    \hline
    2D-system \eqref{eq:hopf-hopf_plane} & 4D-system \eqref{eq:NF_neumann} & R.D. system \eqref{2rd} \\
    \hline
    $\tilde E_1$ & equilibrium & 1:2 mixed mode stationary sol. \\
    $\tilde E_2$ & periodic orbit & uniform osci. + 1:2 stationary sol. \\
    $\tilde E_3$ & periodic orbit & 1:2 mixed mode osci. \\
    $\tilde E_4$ & invariant 2-torus & uniform osci. + 1:2 mixed mode osci. \\
    $\tilde{\mathrm{PO}}$ & invariant 3-torus & \\
    \hline
  \end{tabular}
\end{table}


\subsubsection{Bifurcation from a 0:2 mixed mode equilibrium}\label{sec:02}

On the invariant subspace $\mathcal{S}_2$, the system \eqref{eq:012NF} is reduced into
\begin{equation}\label{eq:02planar}
\begin{cases}
    \dot r_0 = (\mu_0 + \sigma_1 r_0^2 + d_{02} z_2^2) r_0,\\
    \dot z_2 = (\mu_2 + d_{20} r_0^2 + \sigma_2 z_2^2) z_2,
\end{cases}
\end{equation}
which is equivalent to the system \eqref{eq:hopf-hopf_plane}.
As we have already seen in the previous section, the above system possibly has four equilibria, namely 
\begin{align*}
\hat E_1 &: (r_0, z_2) = (0, 0),\\
\hat E_2 &: (r_0, z_2) = \left( \sqrt{-\dfrac{\mu_0}{\sigma_1}}, 0 \right)  \ \text{for} \ \sigma_1 \mu_0 < 0,\\
\hat E_3 &: (r_0, z_2) = \left( 0, \pm\sqrt{-\dfrac{\mu_2}{\sigma_2}} \right) \ \text{for} \ \sigma_2 \mu_2 < 0,\\
\hat E_{4}^{\pm} &: (r_0, z_2) = \left(\sqrt{\dfrac{d_{02} \mu_2 - \sigma_2 \mu_0}{\hat\Delta}},  \pm\sqrt{\dfrac{d_{20} \mu_0 - \sigma_1 \mu_2}{\hat\Delta}} \right)\\
& \text{for} \ \dfrac{d_{02} \mu_2 - \sigma_2 \mu_0}{\hat\Delta} > 0 \ \text{and} \ \dfrac{d_{20} \mu_0 - \sigma_1 \mu_2}{\hat\Delta} > 0,
\end{align*}
where $\hat\Delta := \sigma_1 \sigma_2 - d_{02} d_{20}$, and $\hat E_{4}^{\pm}$ correspond to periodic orbit that possess the Hopf--zero instability under a certain condition.
Indeed, the linearized matrix of \eqref{eq:012NF} at $\hat E_{4}^{\pm}$ becomes a form of
\[
\begin{pmatrix}
2\sigma_1 r_0^2 & 0 & 2 d_{02}r_0 z_2 \\
0 & \mu_1 + d_{10} r_0^2 + d_{12} z_2^2 + z_2 & 0 \\
2 d_{20} r_0 z_2 & 0 & 2 \sigma_2 z_2^2
\end{pmatrix}
\]
and the matrix has a pair of purely imaginary eigenvalues at $\mu_0 = (d_{02} + \sigma_1)\mu_2/(d_{20} + \sigma_2)$ in the case where
\[
\sigma_1 = 1, \quad d_{02} > 0, \quad d_{20} < 0, \quad \sigma_2 = -1 \quad \text{and} \quad \hat\Delta > 0 \quad \text{(case VIa in \cite{GH})}
\]
or
\[
\sigma_1 = -1, \quad d_{02} < 0, \quad d_{20} > 0, \quad \sigma_2 = 1 \quad \text{and} \quad \hat\Delta > 0 \quad \text{(case VIa in \cite{DIR})}.
\]
In addition, if $\mu_1 = -(d_{10} r_0^2 + d_{12} z_2^2 + z_2)$ holds simultaneously, the matrix has totally the Hopf--zero singularity.
Therefore, we expect the interaction between Hopf bifurcation induced by $0:2$ mode interaction and pitchfork bifurcation induced by $1$ mode.
This implies that the Hopf--picthfork bifurcation may occur at $\hat E_{4}^{\pm}$, however, to classify the bifurcation structures from $\hat E_{4}^{\pm}$, we again encounter the same difficulty mentioned above, that is, we need to calculate up to the fifth-order terms of \eqref{eq:normalform}.
Generically, the normal form for the Hopf--pitchfork bifurcation in cylindrical coordinate is given by
\begin{equation}\label{eq:hopf-pitch}
\begin{cases}
\dot \rho_{0,2} = (\nu_{0, 2} + \check J_1 \rho_{0, 2}^2 + \check J_2 X_1^2) \rho_{0,2},\\ 
\dot X_1 = (\nu_{1} + \check K_1 \rho_{0, 2}^2 + \check K_2 X_1^2 ) X_1,\\ 
\dot \theta_{0, 2} = \omega_{0, 2} + \mathcal{O}_2(\rho_{0,2}, X_1, \theta_{0, 2}),
\end{cases}
\end{equation}
and hence, the planar system of \eqref{eq:hopf-pitch} is similar to \eqref{eq:hopf-hopf_plane}.
Therefore, the planar system of \eqref{eq:hopf-pitch} posesses four equilibria $\check E_j$ $(j = 1, 2, 3, 4)$, and the bifurcation diagrams and phase portraits can be classified into twelve cases (see Table~\ref{table:HH}).
Furthermore, for Case VIa, by restoring the quintic terms to \eqref{eq:hopf-pitch}, we can obtain a periodic orbit ($\check{\mathrm{PO}}$) which bifurcates from $\check{E}_4^{\pm}$.
In Table~\ref{table:2}, we list the classification of equilibria and solution correspondence.
\begin{table}[hbtp]
 \caption{Classification of equilibria and solution correspondence.}\label{table:2}
 \centering
  \begin{tabular}{lll}
    \hline
    \eqref{eq:hopf-pitch} & \eqref{eq:NF_neumann} & \eqref{2rd} \\
    \hline
    $\check E_1$ & periodic orbit & stationary 2-mode sol. + uniform osci. \\
    $\check E_2$ & invariant 2-torus & uniform osci. + 0:2 mixed mode osci. \\
    $\check E_3$ & periodic orbit & stationary 1-mode sol. + 0:2 mixed mode osci. \\
    $\check E_{4}^{\pm}$ & invariant 2-torus & uniform osci. + 0:2 mixed mode osci. + stationary 1-mode sol.\\
    $\check{\mathrm{PO}}$ & invariant 3-torus & \\
    \hline
  \end{tabular}
\end{table}

\subsection{Numerical experiments to the normal form}\label{sec:numerical}

In this subsection, we numerically solve the reduced system \eqref{eq:012NF} for the particular case, and visualize the periodic orbits, invariant tori, heteroclinic cycle and chaotic attractors.
We set the coefficients of \eqref{eq:012NF} as
\begin{equation}\label{coef:real_normal}
\begin{gathered}
\sigma_j = -1 \ (j = 1,2,3), \quad d_{01} = 3.0, \quad d_{02} = 3.0, \quad d_{10} = -3.0, \\
d_{11} = -1.0, \quad d_{12} = -3.0, \quad d_{20} = -3.0, \quad d_{21} = 3.0.
\end{gathered}
\end{equation}
On the invariant subspace $\mathcal{S}_2$, the behavior of the solution follows Case~III in the section~7.5 of \cite{GH}, namely, Hopf bifurcation cannot occur at $\tilde E_{4}$.
By putting $\rho = 0.5$, we have $J_1 = 1.78182\cdots$, $J_2 = 0.0165818\cdots$, $\mathrm{Re}\, K_1 = -5.00455\dots$, $\mathrm{Re}\, K_2 = -0.00512727\cdots$, and hence $b = 3.23404\cdots$, $c = -2.80867\cdots$, $d = -1$ and $-1 - b c = 8.08337 \cdots > 0$ hold.
Therefore, the classification of the bifurcation structure to the system \eqref{eq:hopf-hopf_plane} is included in Case~VIa.
The Hopf--Hopf bifurcation point at MM$_{1:2}$ is 
\[
(r_0, z_1, z_2, \mu_0, \mu_1, \mu_2) = (0, 0.05, 0.1, -0.0375, -0.0175, 0.0275),
\]
and around this point, we can find the limit cycle, heteroclinic cycle and chaotic attractors in $(r_0, z_1, z_2)$-space as shown in Figs.~\ref{fig:12hopf}--\ref{fig:049}.
Note that the green dots in Figs.~\ref{fig:12hopf} and \ref{fig:chaos} represent the initial conditions.

In Figs.~\ref{fig:12hopf} and \ref{fig:torus}, we have fixed $(\mu_1, \mu_2) = (-0.07, 0.035)$ and $\mu_0$ varies in the interval $[-0.07, -0.04906]$.
Fig.~\ref{fig:12hopf} shows the solution orbit in $(r_0, z_1, z_2)$-space, and (a), (b) and (c) in Fig.~\ref{fig:12hopf} correspond to $\tilde E_3$, $\tilde E_{4}$ and $\tilde{\mathrm{PO}}$, respectively.
Note that we should restore the rotation component $\theta_0$ to elucidate the solution orbit of the full system \eqref{eq:normalform}.
That is, one must notice that in the sense of \eqref{eq:normalform}, the Figs.~\ref{fig:12hopf}(b) and (c) correspond to the invariant 2-torus and invariant 3-torus, respectively.

\begin{figure}[htbp]
\centering
\begin{tabular}{ccc}
\includegraphics[width=0.32\columnwidth]{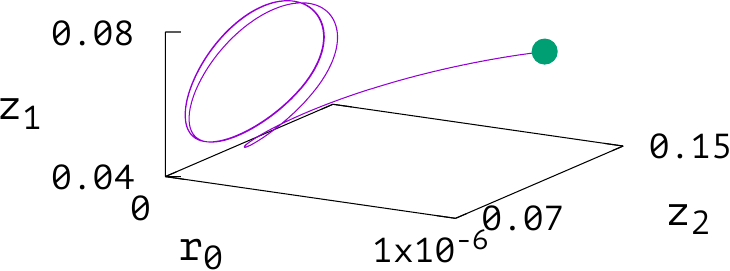}&
\includegraphics[width=0.32\columnwidth]{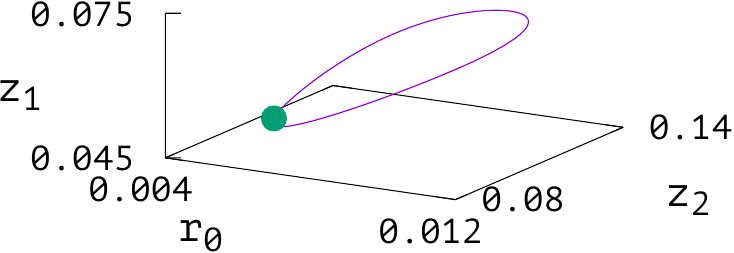}&
\includegraphics[width=0.32\columnwidth]{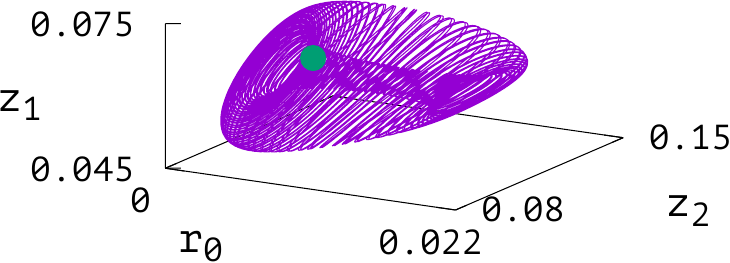}\\
(a) $\mu_0 = -0.07$. & (b) $\mu_0 = -0.05$ & (c) $\mu_0 = -0.04906$
\end{tabular}
\caption{The numerical results for the reduced system \eqref{eq:012NF} under the parameter settings \eqref{coef:real_normal}.}\label{fig:12hopf}
\end{figure}

Furthermore, various projections, time evolutions of $r_0(t)$, $z_1(t)$, $z_2(t)$ and the $\mathbb{R}^3$-norm of Fig.~\ref{fig:12hopf}(c) are shown in Fig.~\ref{fig:torus}.
Figs.~\ref{fig:torus}(a), (b) and (c) show projections of the orbit onto $(r_0, z_1)$, $(r_0, z_2)$ and $(z_1, z_2)$-planes, respectively.
\begin{figure}[htbp]
  \centering
  \begin{tabular}{ccc}
    \includegraphics[width=0.31\columnwidth]{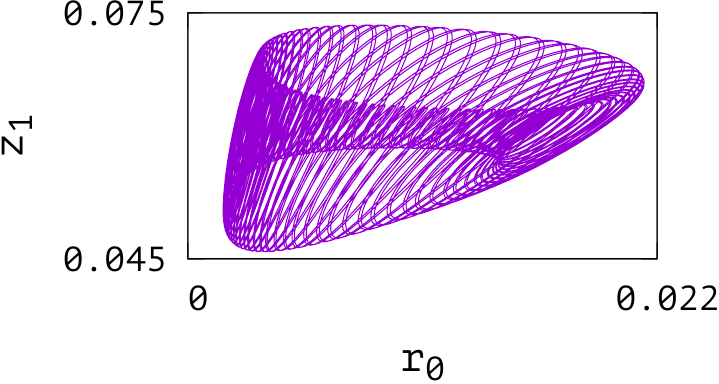}&
    \includegraphics[width=0.31\columnwidth]{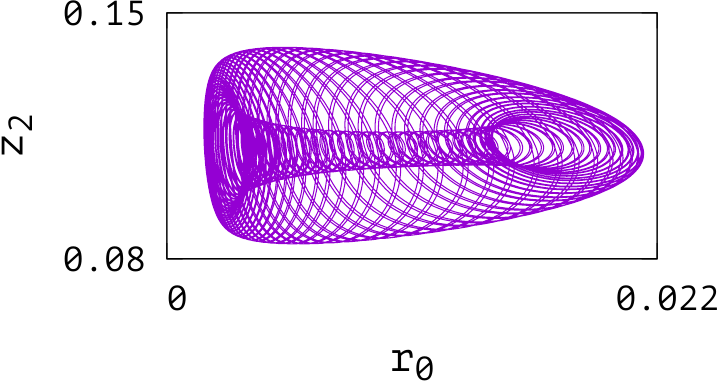}&
    \includegraphics[width=0.31\columnwidth]{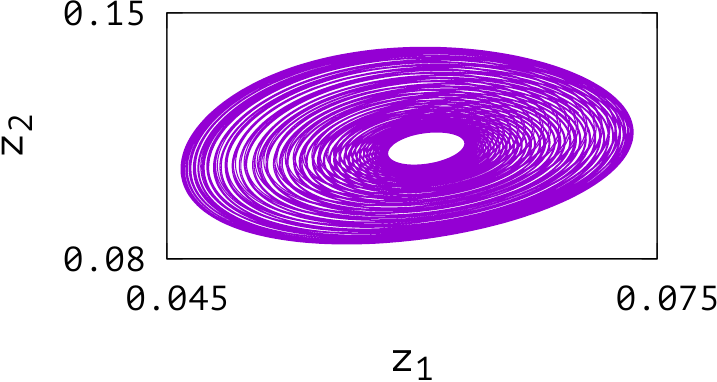}\\
    (a) & (b) & (c)\\
    \includegraphics[width=0.32\columnwidth]{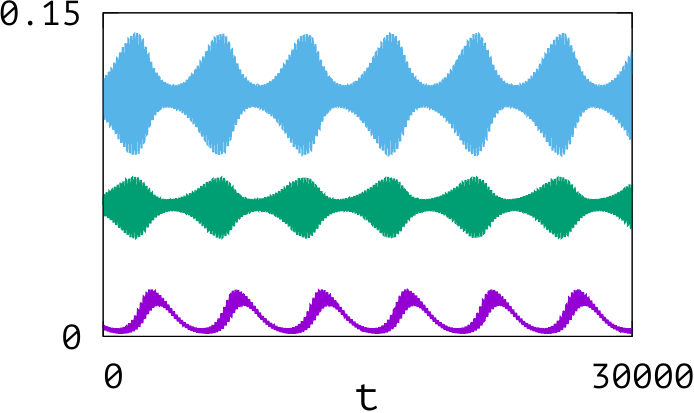}&
    \includegraphics[width=0.32\columnwidth]{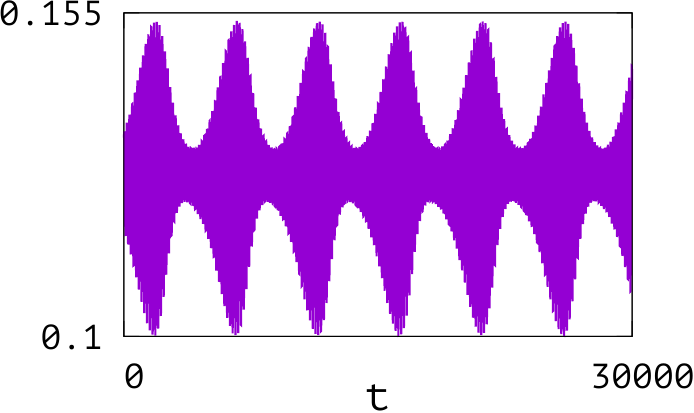}& \\
    (d) & (e) & 
  \end{tabular}
\caption{The detailed description of the solution in Fig.~\ref{fig:12hopf}(c).
(a) Projection onto $(r_0, z_1)$-plane.
(b) Projection onto $(r_0, z_2)$-plane.
(c) Projection onto $(z_1, z_2)$-plane.
(d) Numerical data ($r_0(t)$: purple, $z_1(t)$: green, $z_2(t)$: light blue).
(e) $\mathbb{R}^3$-norm of the numerical solution.}\label{fig:torus}
\end{figure}

We can numerically find that the system \eqref{eq:012NF} has a heteroclinic orbit connecting O and PM$_2^+$ on $\mathcal{S}_1$, as shown in Fig.~\ref{fig:chaos}(a).
This connecting orbit can be observed by setting the parameter values so that the amplitude of the periodic orbit in Fig.~\ref{fig:12hopf}(a) increases.
In the following, when the parameters are fixed as $(\mu_1, \mu_2) = (-0.0625, 0.035)$ and $\mu_0$ varies in the interval $[-0.08, -0.045]$, 
we investigate how the orbit changes according to the parameter value $\mu_0$ and what kind of solution behavior is observed.
\begin{figure}[htbp]
  \centering
  \begin{tabular}{ccc}
    \includegraphics[width=0.31\columnwidth]{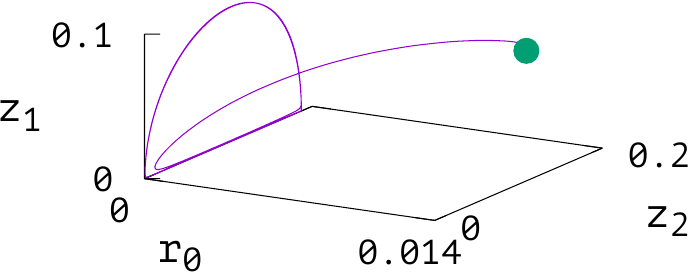}&
    \includegraphics[width=0.31\columnwidth]{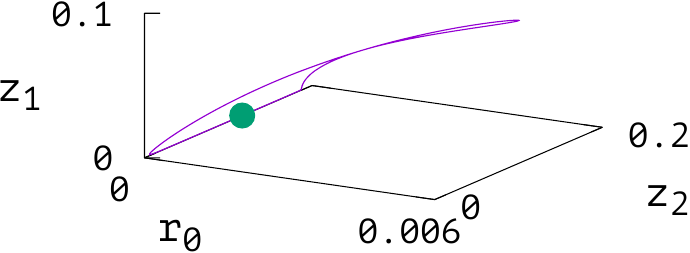}&
    \includegraphics[width=0.31\columnwidth]{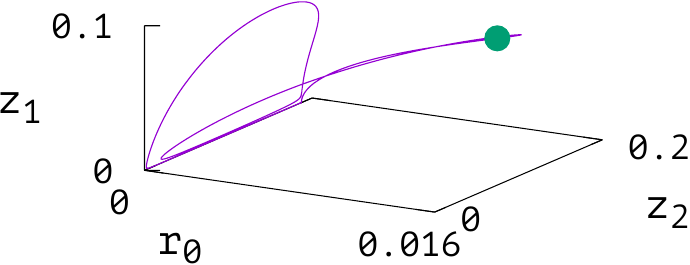}\\
    (a) $\mu_0 = -0.08$. & (b) $\mu_0 = -0.07$. & (c) $\mu_0 = -0.069$. \\
    \includegraphics[width=0.31\columnwidth]{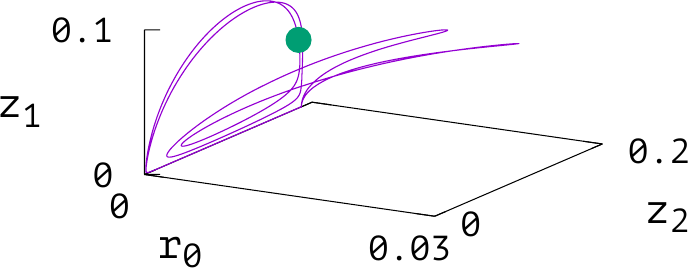}&
    \includegraphics[width=0.31\columnwidth]{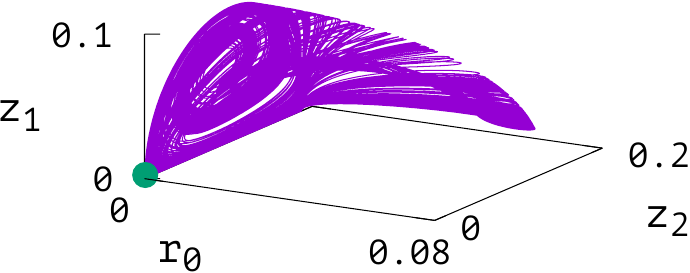}&
    \includegraphics[width=0.31\columnwidth]{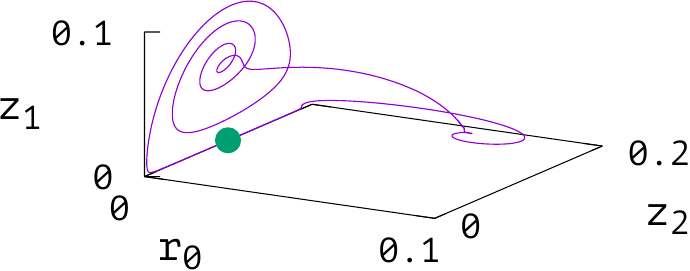}\\
    (d) $\mu_0 = -0.068$ & (e) $\mu_0 = -0.063$ & (f) $\mu_0 = -0.05$. \\
    \includegraphics[width=0.31\columnwidth]{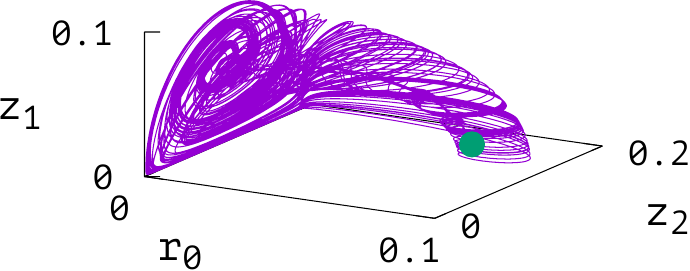}&
    \includegraphics[width=0.31\columnwidth]{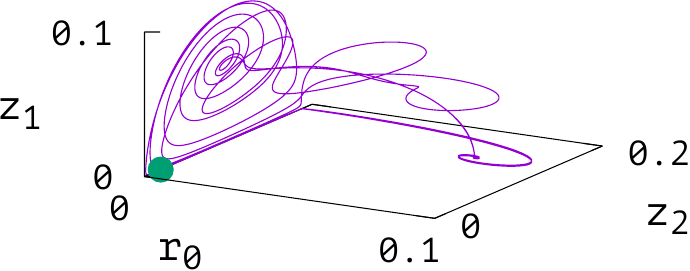}&
    \includegraphics[width=0.31\columnwidth]{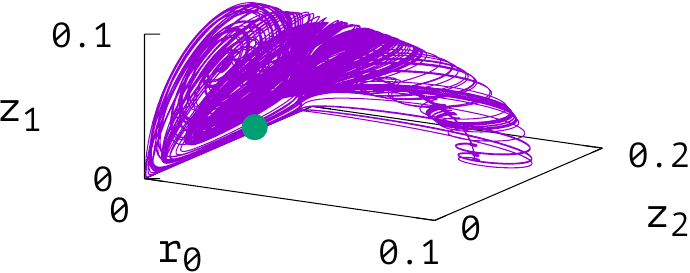}\\
    (g) $\mu_0 = -0.0495$. & (h) $\mu_0 = -0.0491$ & (i) $\mu_0 = -0.049$. \\
    \includegraphics[width=0.31\columnwidth]{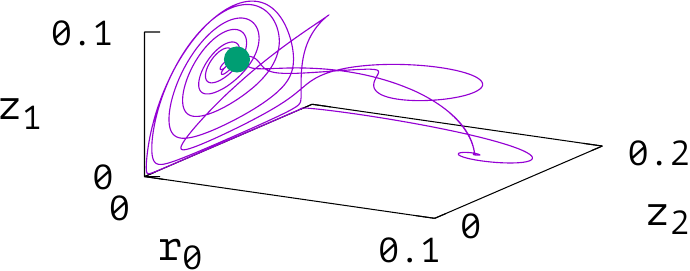}&
    \includegraphics[width=0.31\columnwidth]{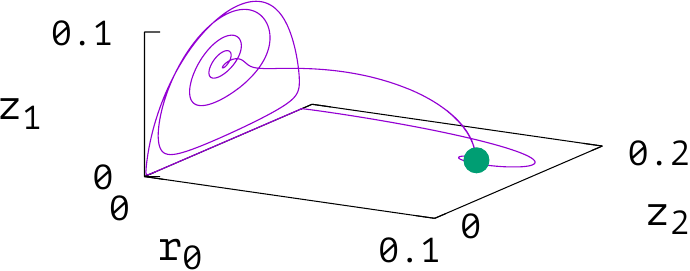}&
    \includegraphics[width=0.31\columnwidth]{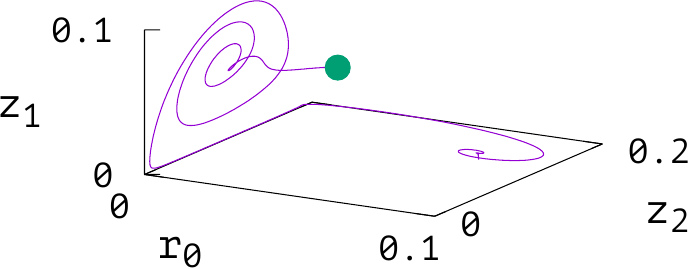}\\
    (j) $\mu_0 = -0.0488$. & (k) $\mu_0 = -0.048$ & (l) $\mu_0 = -0.045$.
    \end{tabular}
  \caption{The numerical results for the reduced system \eqref{eq:012NF} under the settings \eqref{coef:real_normal}.
    (a): Heteroclinic cycle connecting O and $\mathrm{PM}_2^+$ in $(z_1, z_2)$-plane.
    (b): Heteroclinic cycle connecting O and $\mathrm{PM}_2^+$ in $(r_0, z_1, z_2)$-space.
    (c): Period doubling orbit.
    (d): Period four orbit.
    (e): Chaotic attractor.
    (f): Heteroclinic cycle connecting O, MM$_{0:2}^+$, PM$_2^+$ and MM$_{1:2}$.
    (g): Chaotic attractor.
    (h): Chaotic heteroclinic cycle.
    (i): Chaotic attractor.
    (j): Chaotic heteroclinic cycle.
    (k): Heteroclinic cycle connecting O, MM$_{0:2}^+$, PM$_2^+$ and MM$_{1:2}$.
    (l): Orbit converging to MM$_{0:2}^+$.
    }\label{fig:chaos}
\end{figure}
Various projections and norms of (e), (g) and (l) of Fig.~\ref{fig:chaos} are shown in Figs.~\ref{fig:0063}, \ref{fig:005} and \ref{fig:049}.
The figures (a), (b) and (c) in Figs.~\ref{fig:0063}, \ref{fig:005} and \ref{fig:049} show projections of the solution onto $(r_0, z_1)$, $(r_0, z_2)$ and $(z_1, z_1)$-planes, respectively.
\begin{figure}[htbp]
\centering
\begin{tabular}{ccc}
    \includegraphics[width=0.31\columnwidth]{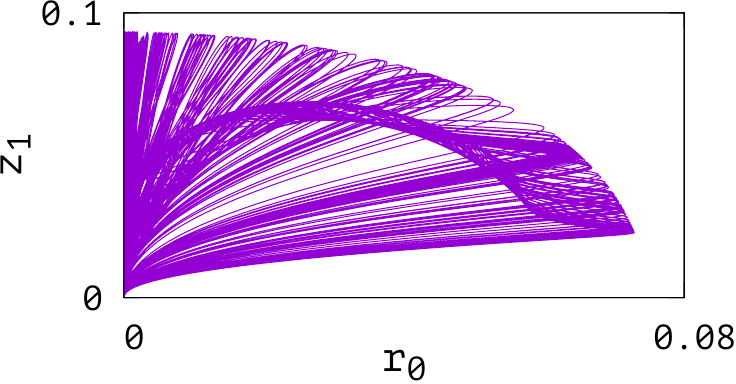}&
    \includegraphics[width=0.31\columnwidth]{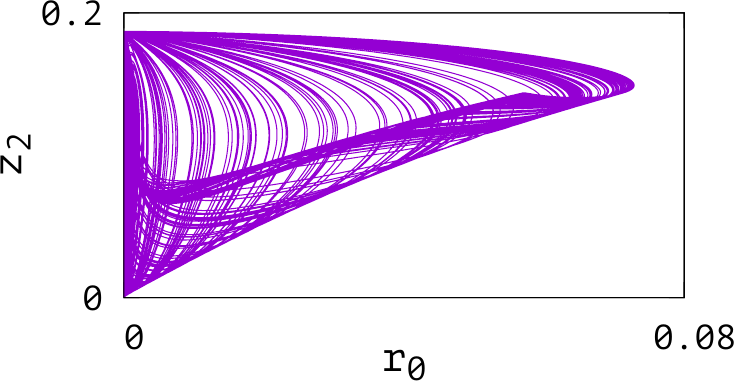}&
    \includegraphics[width=0.31\columnwidth]{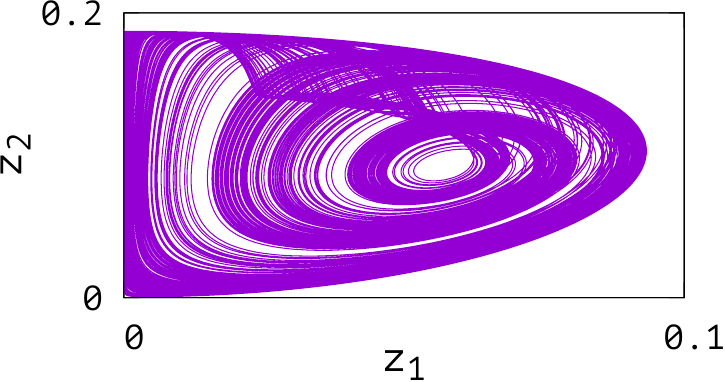}\\
    (a) & (b) & (c) \\
    \includegraphics[width=0.31\columnwidth]{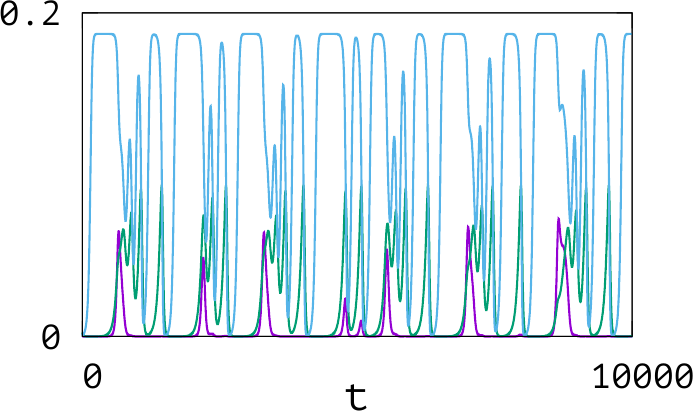}&
    \includegraphics[width=0.31\columnwidth]{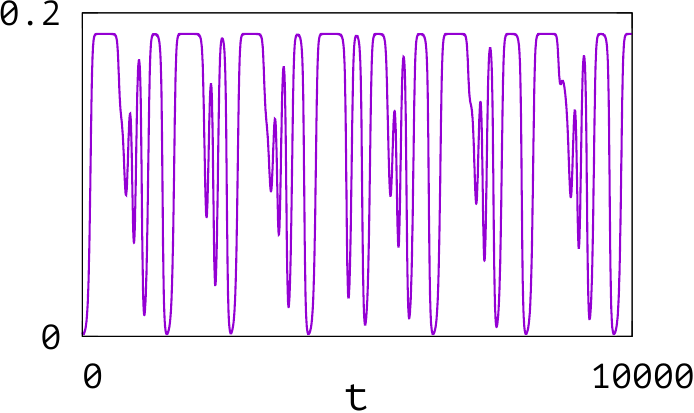}& \\
    (d) & (e)  
   \end{tabular}
\caption{Various projections and norms at $\mu_0 = - 0.063$. 
(a) Projection onto $(r_0, z_1)$-plane.
(b) Projection onto $(r_0, z_2)$-plane.
(c) Projection onto $(z_1, z_2)$-plane.
(d) Numerical data ($r_0(t)$: purple, $z_1(t)$: green, $z_2(t)$: light blue).
(e) $\mathbb{R}^3$-norm of the numerical solution.}\label{fig:0063}
\end{figure}
\begin{figure}[htbp]
\centering
\begin{tabular}{ccc}
    \includegraphics[width=0.31\columnwidth]{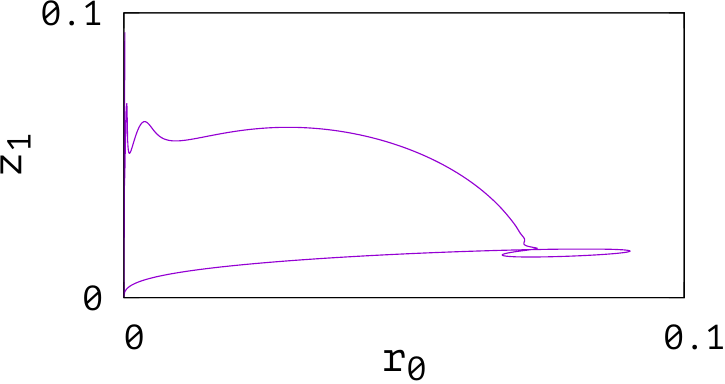}&
    \includegraphics[width=0.31\columnwidth]{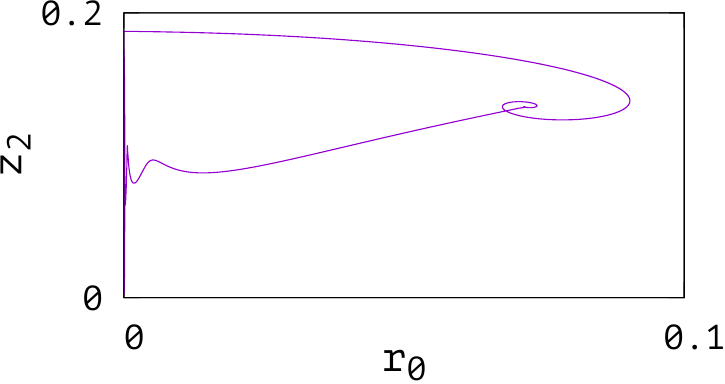}&
    \includegraphics[width=0.31\columnwidth]{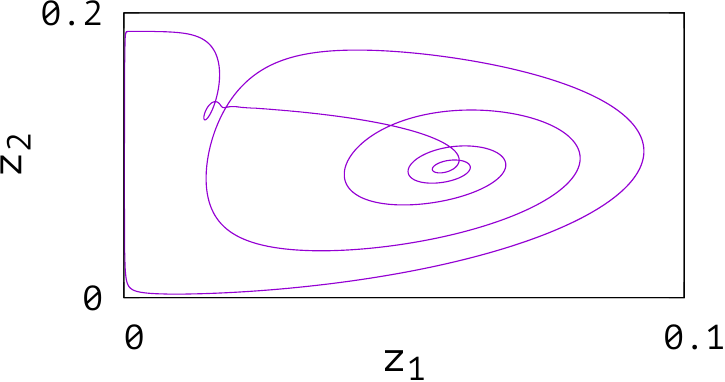}\\
    (a) & (b) & (c) \\
    \includegraphics[width=0.31\columnwidth]{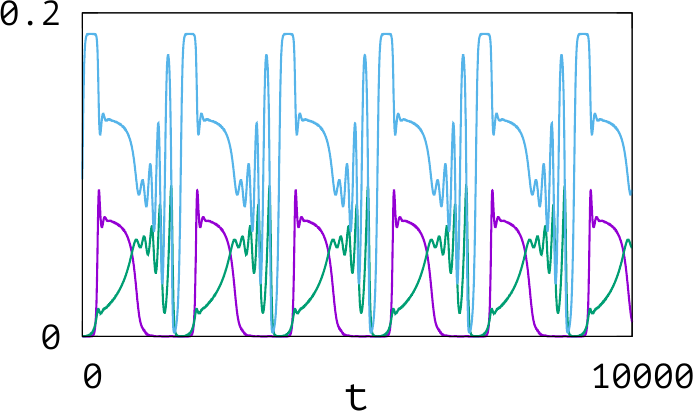}&
    \includegraphics[width=0.31\columnwidth]{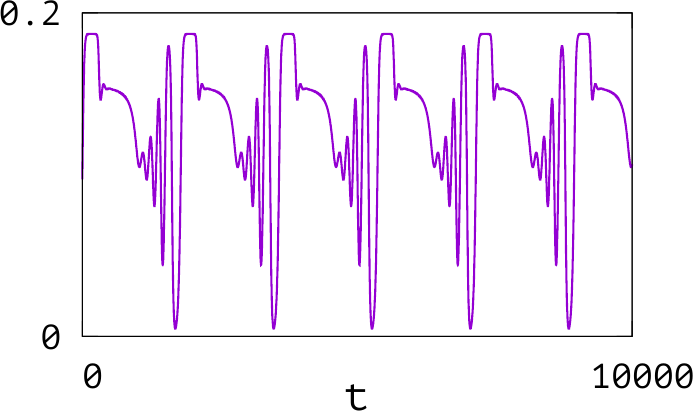}& \\
    (d) & (e) 
   \end{tabular}
\caption{Various projections and norms at $\mu_0 = - 0.05$.
(a) Projection onto $(r_0, z_1)$-plane.
(b) Projection onto $(r_0, z_2)$-plane.
(c) Projection onto $(z_1, z_2)$-plane.
(d) Numerical data ($r_0(t)$: purple, $z_1(t)$: green, $z_2(t)$: light blue).
(e) $\mathbb{R}^3$-norm of the numerical solution.}\label{fig:005}
\end{figure}
\begin{figure}[htbp]
\centering
\begin{tabular}{ccc}
    \includegraphics[width=0.32\columnwidth]{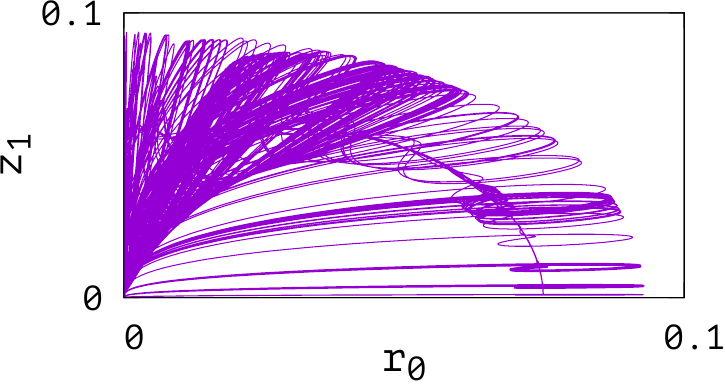}&
    \includegraphics[width=0.32\columnwidth]{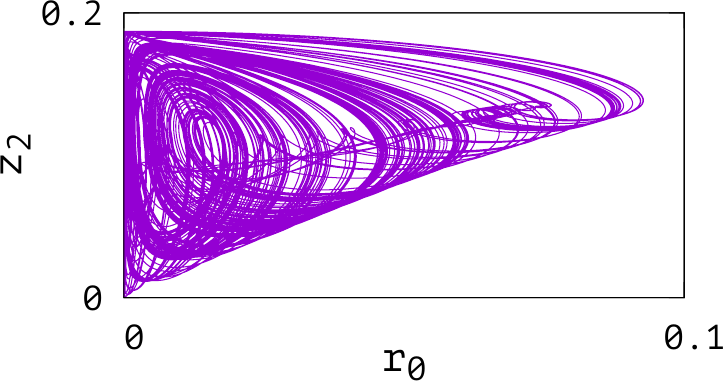}&
    \includegraphics[width=0.32\columnwidth]{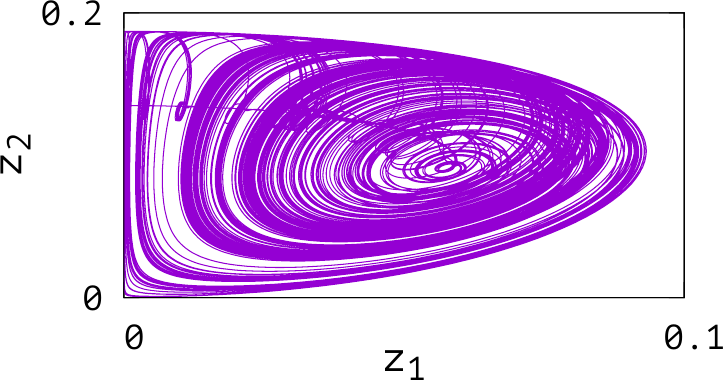}\\
    (a) & (b) & (c) \\
    \includegraphics[width=0.32\columnwidth]{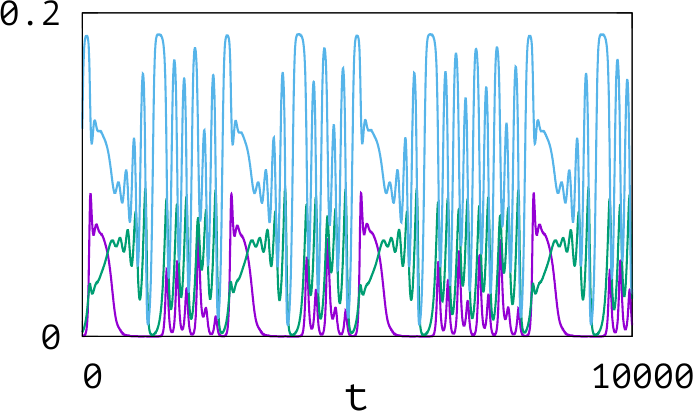}&
    \includegraphics[width=0.32\columnwidth]{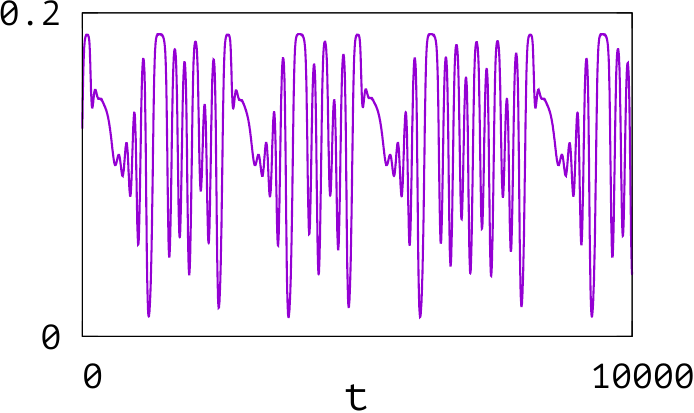}& \\
    (d) & (e) & 
    \end{tabular}
\caption{Various projections and norms at $\mu_0 = - 0.049$.
(a) Projection onto $(r_0, z_1)$-plane.
(b) Projection onto $(r_0, z_2)$-plane.
(c) Projection onto $(z_1, z_2)$-plane.
(d) Numerical data ($r_0(t)$: purple, $z_1(t)$: green, $z_2(t)$: light blue).
(e) $\mathbb{R}^3$-norm of the numerical solution.}\label{fig:049}
\end{figure}

Fig.~\ref{fig:chaos} strongly suggests the existence of heteroclinic cycles and strange attractors in \eqref{eq:normalform}, and therefore, we emphasize that the Hopf--Turing--Turing instability may potentially induce the existence of strange attractors in the reaction-diffusion system \eqref{2rd} with \eqref{BC}.

Theoretical analysis to the existence of the strange attractor in the infinite dimensional dynamical system, that is, to construct a horseshoe map for the heteroclinic orbits as in Figs.\ref{fig:chaos}(f) or (k) and clarity a condition to the existence of horseshoe like Shilnikov's condition remain as interesting tasks for the future.
Furthermore, detailed numerical explorations to \eqref{eq:normalform} or \eqref{2rd} near the Hopf--Turing--Turing bifurcation point are also our interesting future works.

\section{Concluding remarks}\label{sec:Conclude}

We considered 2-component reaction-diffusion systems, and proposed a new type of instability framework; the Hopf instability of $0$-mode and the diffusion-induced instability of $m$ and $m+1$-modes. 
The pattern dynamics around the doubly degenerate point of $m$ and $m+1$-modes and around the Hopf--Turing point have been considered so far. 
In the present paper, we investigated the dynamics around triply degenerate point from the viewpoint of $0 : m : m+1$ mode interaction, 
and revealed that bifurcation structures and dynamics are much richer than those of the codimension 2 bifurcation cases. 
Based on the dynamics of the reaction-diffusion systems, we derived the normal form for the Hopf--Turing--Turing bifurcation.
The derivation of the normal form of this type is quite new, which is one of our main results in the present paper. 
Our numerical computations of the normal form strongly suggests that the system possesses chaotic behavior of solutions when the parameter values are fixed suitably. 
Unfortunately, we generally do not know whether or not the original reaction-diffusion systems possess the corresponding chaotic behavior of solutions even though the normal form shows a chaotic behavior. 
However, it seems that complicated dynamics is potentially included in the 2-component reaction-diffusion systems. 
The further discussion of the relation between the reaction-diffusion systems and the normal form from pattern dynamics point of view is left. 

It is well known that some reaction-diffusion systems, for instance the Gray--Scott model, show a spatio-temporal chaos in one-space dimension (\cite{NU}).
The Keller--Segel model with the logistic growth, which is not classified into reaction-diffusion systems, also exhibits a chaotic behavior (\cite{PH}). 
Since the parameter regime for the occurrence of these spatio-temporal chaos is far from equilibrium, it can not be captured directly via the reduced finite dimensional dynamical system. 
However, the organizing center of a spatio-temporal chaos observed in the systems may be a triply degenerate point discussed in this paper. 
To explore the onset, detailed investigations are necessary with an aid of computer. 
Our instability framework has much potential for unveiling complicated dynamics observed in the reaction-diffusion systems. 

Finally, we adress some future works in this research direction. 
In Section~\ref{Numerics}, we focused on bifurcations from a stable time periodic solution branch with spatially homogeneity. 
On the other hand, the bifurcations from an unstable time periodic solution branch with spatially homogeneity are also possible when $\alpha>\alpha^*$, where $\alpha^*$ is the Hopf criticality of $0$-mode. 
An interaction between stationary solutions and unstable uniform oscillations for $\alpha>\alpha^*$ is unclarified. 
Investigation of bifurcation structures when $\alpha$ is used as a bifurcation parameter may be also interesting since the value $D_u$ or $D_v$ is a bifurcation parameter in Section~\ref{Numerics}. 
We derived the normal form for the Hopf--Turing--Turing bifurcation, but the analysis is quite hard because the detailed investigations of equilibria need the information on the higher order terms such as fifth order terms. 
That is hardly realized. 
We think that complementary analysis combined with rigorous analysis and numerical computation is required to reveal the dynamics inside the system. 
Moreover, if we impose the periodic boundary conditions instead of the zero flux boundary conditions, 
we can guess that the pattern dynamics is much richer. 
In fact, if we set
$D_1 = 0.014$, $D_2 = 0.1502$, $A = 0.1$, $B = 1.0$, $\alpha = 0.6$ and $L = 2.0$ in \eqref{Schnak}, 
then we numerically find the chaotic dynamics such as Fig.~\ref{fig:tth-periodic}.
\begin{figure}[htbp]
  \centering
  \begin{tabular}{ccc}
    \includegraphics[width=0.46\columnwidth]{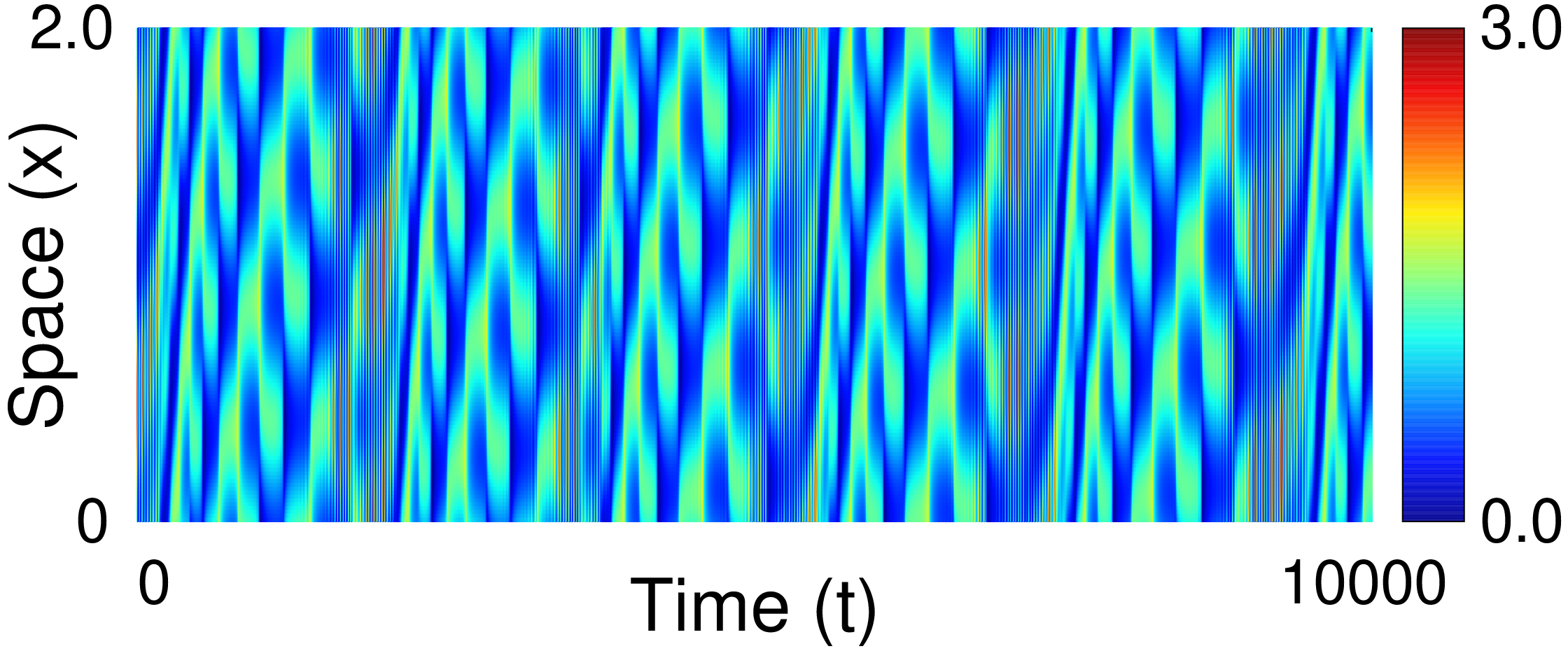}&
    \includegraphics[width=0.46\columnwidth]{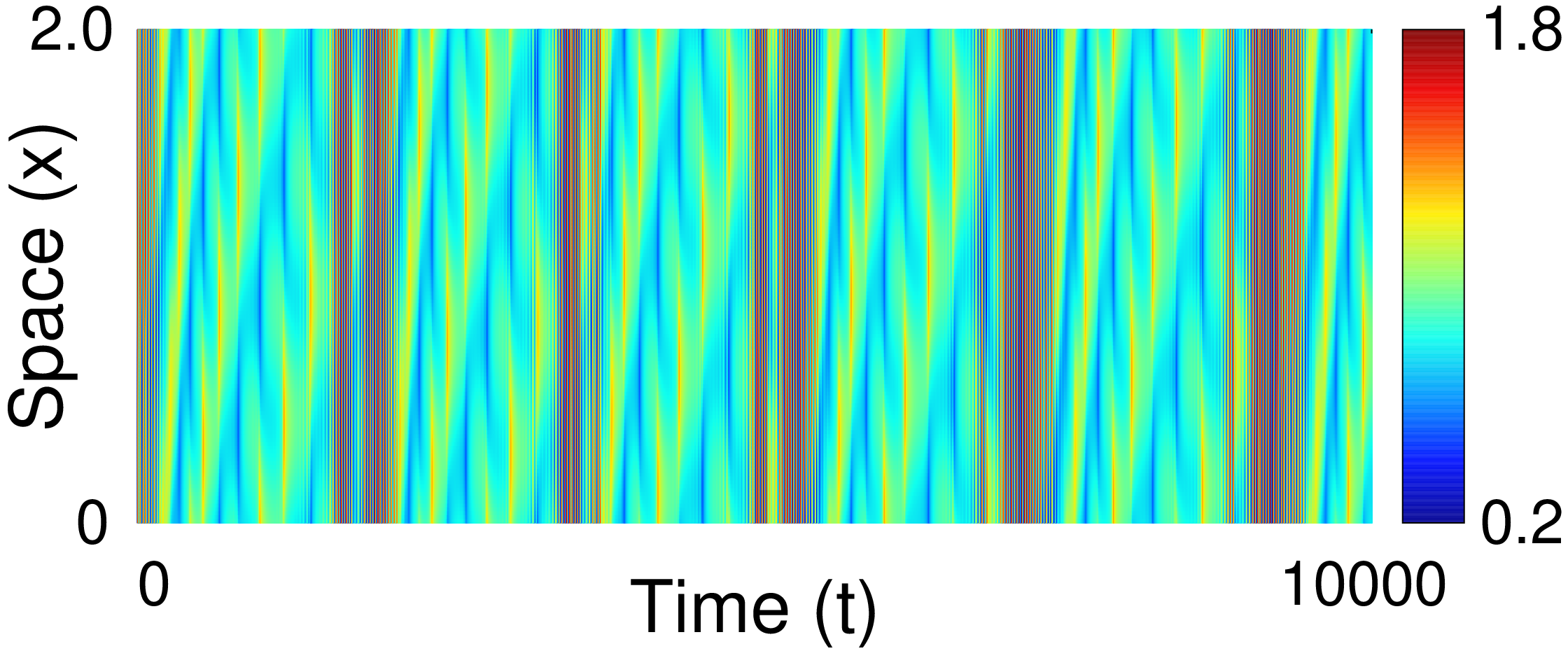}\\
    (a) & (b) \\
    \includegraphics[width=0.46\columnwidth]{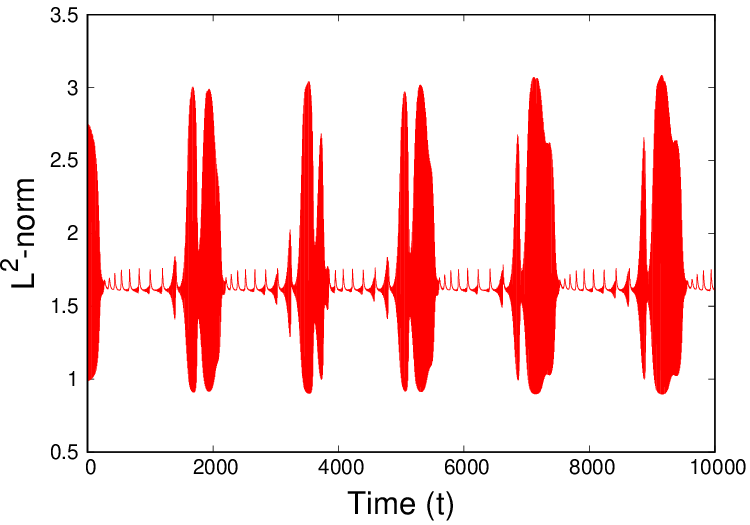} & \\
    (c) & 
    \end{tabular}
  \caption{The numerical result for \eqref{Schnak} under the periodic boundary conditions.
  (a) Numerical solution of $u(x, t)$.
  (b) Numerical solution of $v(x, t)$.
  (c) $L^2$-norm of $u(x, t)$.}\label{fig:tth-periodic}
\end{figure}
Extension of this study to the periodic boundary conditions is also interesting.


\backmatter

\bmhead{Acknowledgments}
H. I. is partially supported by JSPS KAKENHI Grant Number 21K03353. 
S. K. is partially supported by JSPS KAKENHI, Grant Numbers 20K22307.
  
\section*{Declarations}

\bmhead{Conflict of interest}
The authors have no conflict of interest. 

%


\begin{appendices}





\section{Coefficients list of (13)}\label{sec:apA}
\begin{align*}
  A_1 &= \dfrac{1}{2 \det T_0}(T_0^{22} f_{0,0}^{1,1} - T_0^{12} g_{0,0}^{1,1}),\\
  A_2 &= \dfrac{1}{\det T_0}(T_0^{22} f_{0,0}^{1,2} - T_0^{12} g_{0,0}^{1,2}),\\
  A_3 &= \dfrac{1}{2 \det T_0}(T_0^{22} f_{0,0}^{2,2} - T_0^{12} g_{0,0}^{2,2}),\\
  A_4 &= \dfrac{1}{\det T_0}(T_0^{22} f_{1,1}^{1,1} - T_0^{12} g_{1,1}^{1,1}),\\
  A_5 &= \dfrac{1}{\det T_0}(T_0^{22} f_{2,2}^{1,1} - T_0^{12} g_{2,2}^{1,1}),\\
  a_1 &= \dfrac{1}{6 \det T_0}(T_0^{22} f_{0,0,0}^{1,1,1} - T_0^{12} g_{0,0,0}^{1,1,1}),\\  
  a_2 &= \dfrac{1}{\det T_0}\left(T_0^{22} (2 f_{1,1}^{1,2} B_{1010}^1 + f_{0,1,1}^{1,1,1}) - T_0^{12} (2 g_{1,1}^{1,2} B_{1010}^1 + g_{0,1,1}^{1,1,1})\right),\\
  a_3 &= \dfrac{1}{\det T_0}\left(T_0^{22} (2 f_{2,2}^{1,2} B_{1001}^2 + f_{0,2,2}^{1,1,1}) - T_0^{12} (2 g_{2,2}^{1,2} B_{1001}^2 + g_{0,2,2}^{1,1,1})\right),\\
  a_4 &= \dfrac{1}{2 \det T_0}(T_0^{22} f_{0,0,0}^{1,1,2} - T_0^{12} g_{0,0,0}^{1,1,2}),\\
  a_5 &= \dfrac{1}{2 \det T_0}(T_0^{22} f_{0,0,0}^{1,2,2} - T_0^{12} g_{0,0,0}^{1,2,2}),\\
  a_6 &= \dfrac{1}{6 \det T_0}(T_0^{22} f_{0,0,0}^{2,2,2} - T_0^{12} g_{0,0,0}^{2,2,2}),\\  
  a_7 &= \dfrac{1}{\det T_0}\left(T_0^{22} (2 f_{1,1}^{1,2} B_{0110}^1 + f_{0,1,1}^{2,1,1}) - T_0^{12} (2 g_{1,1}^{1,2} B_{0110}^1 + g_{0,1,1}^{2,1,1})\right),\\
  a_8 &= \dfrac{1}{\det T_0}\left(T_0^{22} (2 f_{2,2}^{1,2} B_{0101}^2 + f_{0,2,2}^{2,1,1}) - T_0^{12} (2 g_{2,2}^{1,2} B_{0101}^2 + g_{0,2,2}^{2,1,1})\right),\\
  a_9 &= \dfrac{2}{\det T_0} \Bigg(T_0^{22} \left(f_{1,1}^{1,2} B_{0011}^1 + f_{2,2}^{1,2} B_{0020}^2 + \dfrac{f_{1,1,2}^{1,1,1}}{2}\right) \\
  &- T_0^{12} \left(g_{1,1}^{1,2} B_{0011}^1 + g_{2,2}^{1,2} B_{0020}^2 + \dfrac{g_{1,1,2}^{1,1,1}}{2}\right) \Bigg),\\
  B_1 &= \dfrac{1}{2 \det T_0}(-T_0^{21} f_{0,0}^{1,1} + T_0^{11} g_{0,0}^{1,1}),\\
  B_2 &= \dfrac{1}{\det T_0}(-T_0^{21} f_{0,0}^{1,2} + T_0^{11} g_{0,0}^{1,2}),\\
  B_3 &= \dfrac{1}{2 \det T_0}(-T_0^{21} f_{0,0}^{2,2} + T_0^{11} g_{0,0}^{2,2}),\\
  B_4 &= \dfrac{1}{\det T_0}(-T_0^{21} f_{1,1}^{1,1} + T_0^{11} g_{1,1}^{1,1}),\\
  B_5 &= \dfrac{1}{\det T_0}(-T_0^{21} f_{2,2}^{1,1} + T_0^{11} g_{2,2}^{1,1}),\\
  b_1 &= \dfrac{1}{6 \det T_0}(-T_0^{21} f_{0,0,0}^{1,1,1} + T_0^{11} g_{0,0,0}^{1,1,1}),\\  
  b_2 &= \dfrac{1}{\det T_0}\left(-T_0^{21} (2 f_{1,1}^{1,2} B_{1010}^1 + f_{0,1,1}^{1,1,1}) + T_0^{11} (2 g_{1,1}^{1,2} B_{1010}^1 + g_{0,1,1}^{1,1,1})\right),\\
  b_3 &= \dfrac{1}{\det T_0}\left(-T_0^{21} (2 f_{2,2}^{1,2} B_{1001}^2 + f_{0,2,2}^{1,1,1}) + T_0^{11} (2 g_{2,2}^{1,2} B_{1001}^2 + g_{0,2,2}^{1,1,1})\right),\\
  b_4 &= \dfrac{1}{2 \det T_0}(-T_0^{21} f_{0,0,0}^{1,1,2} + T_0^{11} g_{0,0,0}^{1,1,2}),\\
  b_5 &= \dfrac{1}{2 \det T_0}(-T_0^{21} f_{0,0,0}^{1,2,2} + T_0^{11} g_{0,0,0}^{1,2,2}),\\
  b_6 &= \dfrac{1}{6 \det T_0}(-T_0^{21} f_{0,0,0}^{2,2,2} + T_0^{11} g_{0,0,0}^{2,2,2}),\\  
  b_7 &= \dfrac{1}{\det T_0}\left(-T_0^{21} (2 f_{1,1}^{1,2} B_{0110}^1 + f_{0,1,1}^{2,1,1}) + T_0^{11} (2 g_{1,1}^{1,2} B_{0110}^1 + g_{0,1,1}^{2,1,1})\right),\\
  b_8 &= \dfrac{1}{\det T_0}\left(-T_0^{21} (2 f_{2,2}^{1,2} B_{0101}^2 + f_{0,2,2}^{2,1,1}) + T_0^{11} (2 g_{2,2}^{1,2} B_{0101}^2 + g_{0,2,2}^{2,1,1})\right),\\
  b_9 &= \dfrac{2}{\det T_0} \Bigg(-T_0^{21} \left(f_{1,1}^{1,2} B_{0011}^1 + f_{2,2}^{1,2} B_{0020}^2 + \dfrac{f_{1,1,2}^{1,1,1}}{2}\right) \\
  &+ T_0^{11} \left(g_{1,1}^{1,2} B_{0011}^1 + g_{2,2}^{1,2} B_{0020}^2 + \dfrac{g_{1,1,2}^{1,1,1}}{2}\right) \Bigg),\\
  C_1 &= \dfrac{1}{\det T_1} (T_1^{22} f_{0,1}^{1,1} - T_1^{12} g_{0,1}^{1,1} ),\\
  C_2 &= \dfrac{1}{\det T_1} (T_1^{22} f_{1,0}^{1,2} - T_1^{12} g_{1,0}^{1,2} ),\\
  C_3 &= \dfrac{1}{\det T_1} (T_1^{22} f_{1,2}^{1,1} - T_1^{12} g_{1,2}^{1,1} ),\\
  c_1 &= \dfrac{1}{\det T_1} \left(T_1^{22} \left(f_{0,1}^{1,2} B_{1010}^1 + \dfrac{f_{0,0,1}^{1,1,1}}{2}\right) - T_1^{12} \left(g_{0,1}^{1,2} B_{1010}^1 + \dfrac{g_{0,0,1}^{1,1,1}}{2}\right) \right),\\
  c_2 &= \dfrac{1}{\det T_1} \bigg(T_1^{22} (f_{0,1}^{1,2} B_{0110}^1 + f_{0,1}^{2,2} B_{1010}^1 + f_{0,0,1}^{1,2,1}) \\
  &- T_1^{12} (g_{0,1}^{1,2} B_{0110}^1 + g_{0,1}^{2,2} B_{1010}^1 + g_{0,0,1}^{1,2,1}) \bigg),\\
  c_3 &= \dfrac{1}{\det T_1} \left(T_1^{22} \left(f_{0,1}^{2,2} B_{0110}^1 + \dfrac{f_{0,0,1}^{2,2,1}}{2}\right) 
  - T_1^{12} \left(g_{0,1}^{2,2} B_{0110}^1 + \dfrac{g_{0,0,1}^{2,2,1}}{2}\right) \right),\\
  c_4 &= \dfrac{1}{\det T_1} \left(T_1^{22} \left(f_{2,1}^{2,1} B_{0020}^2 + \dfrac{f_{1,1,1}^{1,1,1}}{2}\right) - T_1^{12} \left(g_{2,1}^{2,1} B_{0020}^2 + \dfrac{g_{1,1,1}^{1,1,1}}{2}\right) \right),\\
  c_5 &= \dfrac{1}{\det T_1} \left(T_1^{22} (f_{2,1}^{1,2} B_{0011}^1 + \tilde f_{2,3} + f_{1,2,2}^{1,1,1}) - T_1^{12} (g_{2,1}^{1,2} B_{0011}^1 + \tilde g_{2,3} + g_{1,2,2}^{1,1,1}) \right),\\
  c_6 &= \dfrac{1}{\det T_1} \bigg(T_1^{22} (f_{0,1}^{1,2} B_{0011}^1 + f_{2,1}^{2,1} B_{1001}^2 + f_{2,1}^{1,2} B_{1010}^1 + f_{0,1,2}^{1,1,1})\\
  & - T_1^{12} (g_{0,1}^{1,2} B_{0011}^1 + g_{2,1}^{2,1} B_{1001}^2 + g_{2,1}^{1,2} B_{1010}^1 + g_{0,1,2}^{1,1,1}) \bigg),\\
  c_7 &= \dfrac{1}{\det T_1} \bigg(T_1^{22} (f_{0,1}^{2,2} B_{0011}^1 + f_{2,1}^{1,2} B_{0110}^1 + f_{2,1}^{2,1} B_{0101}^2 + f_{0,1,2}^{2,1,1})\\
  & - T_1^{12} (g_{0,1}^{2,2} B_{0011}^1 + g_{2,1}^{1,2} B_{0110}^1 + g_{2,1}^{2,1} B_{0101}^2 + g_{0,1,2}^{2,1,1}) \bigg),\\
  D_1 &= \dfrac{1}{2 \det T_2} (T_2^{22} f_{1,1}^{1,1} - T_2^{12} g_{1,1}^{1,1} ),\\
  D_2 &= \dfrac{1}{\det T_2} (T_2^{22} f_{0,2}^{1,1} - T_2^{12} g_{0,2}^{1,1} ),\\
  D_3 &= \dfrac{1}{\det T_2} (T_2^{22} f_{0,2}^{2,1} - T_2^{12} g_{0,2}^{2,1} ),\\
  d_1 &= \dfrac{1}{\det T_2} \left(T_2^{22} \left(f_{0,2}^{1,2} B_{1001}^2 + \dfrac{f_{0,0,2}^{1,1,1}}{2}\right) - T_2^{12} \left(g_{0,2}^{1,2} B_{1001}^2 + \dfrac{g_{0,0,2}^{1,1,1}}{2}\right) \right),\\
  d_2 &= \dfrac{1}{\det T_2} \left(T_2^{22} (f_{0,2}^{1,2} B_{0101}^2 + f_{0,2}^{2,2} B_{1001}^2 + f_{0,0,2}^{1,2,1} ) - T_2^{12} (g_{0,2}^{1,2} B_{0101}^2 + g_{0,2}^{2,2} B_{1001}^2 + g_{0,0,2}^{1,2,1} ) \right),\\
  d_3 &= \dfrac{1}{\det T_2} \left(T_2^{22} \left(f_{0,2}^{2,2} B_{0101}^2 + \dfrac{f_{0,0,2}^{2,2,1}}{2}\right) - T_2^{12} \left(g_{0,2}^{2,2} B_{0101}^2 + \dfrac{g_{0,0,2}^{2,2,1}}{2}\right) \right),\\
  d_4 &= \dfrac{1}{\det T_2} \left(T_2^{22} (f_{1,1}^{1,2} B_{0011}^1 + \tilde f_{1,3} + f_{2,1,1}^{1,1,1}) - T_2^{12} (g_{1,1}^{1,2} B_{0011}^1 + \tilde g_{1,3} + g_{2,1,1}^{1,1,1}) \right),\\
  d_5 &= \dfrac{1}{\det T_2} \left(T_2^{22} \left(\tilde f_{2,4} + \dfrac{f_{2,2,2}^{1,1,1}}{2}\right) - T_2^{12} \left(\tilde g_{2,4} + \dfrac{g_{2,2,2}^{1,1,1}}{2} \right) \right),\\
  d_6 &= \dfrac{1}{\det T_2} \Bigg(T_2^{22} \left(f_{1,1}^{1,2} B_{1010}^1 + f_{0,2}^{1,2} B_{0020}^2 + \dfrac{f_{1,1,0}^{1,1,1}}{2} \right) \\
  &- T_2^{12} \left(g_{1,1}^{1,2} B_{1010}^1 + g_{0,2}^{1,2} B_{0020}^2 + \dfrac{g_{1,1,0}^{1,1,1}}{2} \right) \Bigg),\\
  d_7 &= \dfrac{1}{\det T_2} \left(T_2^{22} \Bigg(f_{1,1}^{1,2} B_{0110}^1 + f_{0,2}^{2,2} B_{0020}^2 + \dfrac{f_{1,1,0}^{1,1,2}}{2} \right) \\
  &- T_2^{12} \left(g_{1,1}^{1,2} B_{0110}^1 + g_{0,2}^{2,2} B_{0020}^2 + \dfrac{g_{1,1,0}^{1,1,2}}{2} \right) \Bigg).
\end{align*}


\section{Coefficients list of (15)}\label{sec:apB}
\begin{align}
  E_1 &= \bar p_1 L_1 + \bar p_2 N_1, & E_2 &= \bar p_1 L_2 + \bar p_2 N_2, & E_3 &= \bar p_1 \bar L_1 + \bar p_2 \bar N_1,\\
  E_4 &= \bar p_1 A_4 + \bar p_2 B_4, & E_5 &= \bar p_1 A_5 + \bar p_2 B_5, & e_1 &= \bar p_1 L_3 + \bar p_2 N_3,\\
  e_2 &= \bar p_1 L_4 + \bar p_2 N_4, & e_3 &= \bar p_1 \bar L_4 + \bar p_2 \bar N_4, & e_4 &= \bar p_1 \bar L_3 + \bar p_2 \bar N_3,\\
  e_5 &= \bar p_1 L_5 + \bar p_2 N_5, & e_6 &= \bar p_1 \bar L_5 + \bar p_2 \bar N_5, & e_7 &= \bar p_1 L_6 + \bar p_2 N_6,\\
  e_8 &= \bar p_1 \bar L_6 + \bar p_2 \bar N_6, & e_9 &= \bar p_1 a_9 + \bar p_2 b_9, &   H_1 &= q_1 C_1 + q_2 C_2, \\
  H_2 &= q_1^2 c_1 + q_1 q_2 c_2 + q_2^2 c_3, & H_3 &= 2(|q_1|^2 c_1 + |q_2|^2 c_3), &  H_4 &= q_1 c_6 + q_2 c_7, \\
  I_1 &= q_1 D_2 + q_2 D_3, & I_2 &= q_1^2 d_1 + q_1 q_2 d_2 + q_2^2 d_3, & I_3 &= 2(|q_1|^2 d_1 + |q_2|^2 d_3),\\
  I_4 &= q_1 d_6 + q_2 d_7, & &
\end{align}
where
\begin{align*}
  L_1 &= A_1 q_1^2 + A_2 q_1 q_2 + A_3 q_2^2, & L_2 &= 2(A_1 |q_1|^2 + A_3 |q_2|^2), \\
  L_3 &= a_1 q_1^3 + (a_4 q_1 + a_5 q_2) q_1 q_2 + a_6 q_2^3, & L_4 &= 3 a_1 |q_1|^2 q_1 + (a_4 \bar q_1 + a_5 \bar q_2) q_1 q_2 + 3 a_6 |q_2|^2 q_2,\\
  L_5 &= a_2 q_1 + a_7 q_2, & L_6 &= a_3 q_1 + a_8 q_2,\\
  N_1 &= B_1 q_1^2 + B_2 q_1 q_2 + B_3 q_2^2, & N_2 &= 2(B_1 |q_1|^2 + B_3 |q_2|^2),\\
  N_3 &= b_1 q_1^3 + (b_4 q_1 + b_5 q_2) q_1 q_2 + b_6 q_2^3, & N_4 &= 3 b_1 |q_1|^2 q_1 + (b_4 \bar q_1 + b_5 \bar q_2) q_1 q_2 + 3 b_6 |q_2|^2 q_2,\\
  N_5 &= b_2 q_1 + b_7 q_2, & N_6 &= b_3 q_1 + b_8 q_2.
\end{align*}


\section{Proof of Theorem~\ref{thm:1}}\label{sec:apC}
\begin{lemma}\label{lem:nit}
  The system
  \begin{equation}
    \begin{cases}
      \dot z = \lambda_0 z + \displaystyle\sum_{0 \le p, q, r, s \le 2 \atop p + q + r + s = 2} \tilde F^0_{p q r s} z^{p} \bar z^{q} \alpha_1^{r} \alpha_2^{s} + \mathcal{O}_3,\\
      \dot \alpha_1 = \mu_1^+ \alpha_1 + \displaystyle\sum_{0 \le p, q, r, s \le 2 \atop p + q + r + s = 2} \tilde F^1_{p q r s} z^{p} \bar z^{q} \alpha_1^{r} \alpha_2^{s} + \mathcal{O}_3,\\
      \dot \alpha_2 = \mu_2^+ \alpha_2 + \displaystyle\sum_{0 \le p, q, r, s \le 2 \atop p + q + r + s = 2} \tilde F^2_{p q r s} z^{p} \bar z^{q} \alpha_1^{r} \alpha_2^{s} + \mathcal{O}_3
    \end{cases}
  \end{equation}
  can be transformed by an invertible parameter-dependent change of complex coordinate
  \begin{equation}
    \begin{cases}
      z = \zeta + \displaystyle\sum_{0\le p, q, r, s \le 2 \atop p + q + r + s = 2} \Gamma_{p q r s} \zeta^{p} \bar\zeta^{q} x^{r} y^{s},\\
      \alpha_1 = x + \displaystyle\sum_{0\le p, q, r, s \le 2 \atop p + q + r + s = 2} \Theta_{p q r s} \zeta^{p} \bar\zeta^{q} x^{r} y^{s},\\
      \alpha_2 = y + \displaystyle\sum_{0\le p, q, r, s \le 2 \atop p + q + r + s = 2} \Lambda_{p q r s} \zeta^{p} \bar\zeta^{q} x^{r} y^{s}
    \end{cases}
  \end{equation}
  into the following dynamical system by setting $\Gamma_{j_1 j_2 k_1 k_2 l_1 l_2}$, $\Theta_{j_1 j_2 k_1 k_2 l_1 l_2}$ and $\Lambda_{j_1 j_2 k_1 k_2 l_1 l_2}$, appropriately:
  \begin{equation}
    \begin{cases}
      \dot \zeta = \lambda_0 \zeta + \tilde F^{0}_{1010} \zeta x + \tilde F^{0}_{1001} \zeta y + \mathcal{O}_3,\\
      \dot x = \mu_1^+ x + \tilde F^{1}_{1100}|\zeta|^{2} + \tilde F^{1}_{0020} x^{2} + \tilde F^{1}_{0011} xy + \tilde F^{1}_{0002} y^{2} + \mathcal{O}_3,\\
      \dot y = \mu_2^+ y + \tilde F^{2}_{1100}|\zeta|^{2} + \tilde F^{2}_{0020} x^{2} + \tilde F^{2}_{0011} xy + \tilde F^{2}_{0002} y^{2} + \mathcal{O}_3.
    \end{cases}
  \end{equation}
\end{lemma}
\begin{proof}
  The inverse change of variables are given by the expressions
  \begin{equation}
    \begin{cases}
      \zeta = z - \displaystyle\sum_{0\le p, q, r, s \le 2 \atop p + q + r + s = 2} \Gamma_{p q r s}z^{p} \bar z^{q} \alpha_1^{r} \alpha_2^{s} + \mathcal{O}_3,\\
      x = \alpha_1 - \displaystyle\sum_{0\le p, q, r, s \le 2 \atop p + q + r + s = 2} \Theta_{p q r s}z^{p} \bar z^{q} \alpha_1^{r} \alpha_2^{s} + \mathcal{O}_3,\\
      y = \alpha_2 - \displaystyle\sum_{0\le p, q, r, s \le 2 \atop p + q + r + s = 2} \Lambda_{p q r s}z^{p} \bar z^{q} \alpha_1^{r} \alpha_2^{s} + \mathcal{O}_3.
    \end{cases}
  \end{equation}
  Differentiating the above on $t$, we have
  \begin{align}
    \dot \zeta &= \lambda_0 \zeta + \displaystyle\sum_{0 \le p, q, r, s \le 2 \atop p + q + r + s = 2}\Big\{\tilde F^0_{p q r s} - \Gamma_{p q r s}\big(\lambda(p - 1) + \bar\lambda q + \mu_1 r  + \mu_2 s \big)\Big\}\zeta^{p} \bar \zeta^{q} x^{r} y^{s} + \mathcal{O}_{3},\\
    \dot x &= \mu_{1}^{+} x + \displaystyle\sum_{0 \le p, q, r, s \le 2 \atop p + q + r + s = 2}\Big\{\tilde F^1_{p q r s} - \Theta_{p q r s} \big(\lambda p + \bar\lambda q + \mu_1 r + \mu_2 s \big)\Big\}\zeta^{p} \bar \zeta^{q} x^{r}  y^{s}  + \mathcal{O}_{3},\\
    \dot y &= \mu_{2}^{+} y + \displaystyle\sum_{0\le p, q, r, s \le 2 \atop p + q + r + s = 2}\Big\{\tilde F^2_{p q r s} - \Lambda_{p q r s}\big(\lambda p + \bar\lambda q + \mu_1 r + \mu_2 s \big)\Big\} \zeta^{p} \bar \zeta^{q} x^{r} y^{s} + \mathcal{O}_{3}.
  \end{align}
  This implies that if $(p, q) = (1, 0)$, then for all $(r, s)$ $\tilde F^{0}_{10rs}$ cannot be erased.
  Also if $(p, q) = (1, 1)$ or $(0, 0)$, then for all $(r, s)$ $\tilde F^{n}_{11rs}$ and $\tilde F^{n}_{00rs}$ $(n = 1, 2)$ cannot be erased.
The other terms can be erased by setting $\Gamma_{p q r s}$, $\Theta_{p q r s}$ and $\Lambda_{p q r s}$, appropriately.
\end{proof}

\begin{corollary}\label{cor:transform-tth}
  Consider our case \eqref{eq:ourcase}.
  By using the invertible parameter-dependent change of complex coordinate:
  \begin{equation}\label{eq:trans_quad}
    \begin{cases}
      z = \zeta + \Gamma_{2000} \zeta^{2} + \Gamma_{1100} |\zeta|^{2} + \Gamma_{0200} \bar\zeta^{2} + \Gamma_{0020} x^{2} + \Gamma_{0002} y^{2},\\
      \alpha_{1} = x + \Theta_{1010} \zeta x + \Theta_{0110} \bar\zeta x,\\
      \alpha_{2} = y + \Lambda_{1001} \zeta y + \Lambda_{0101} \bar\zeta y,
    \end{cases}
  \end{equation}
  where
  \begin{align*}
    \Gamma_{2000} &= \dfrac{E_{1}}{\lambda_0}, & \Gamma_{1100} &= \dfrac{E_{2}}{\bar\lambda_0}, & \Gamma_{0200} &= \dfrac{E_{3}}{2\bar\lambda_0 - \lambda_0},\\
    \Gamma_{0020} &= \dfrac{E_{4}}{2\mu_{1}^+ - \lambda_0}, & \Gamma_{0002} &= \dfrac{E_{5}}{2\mu_{2}^+ - \lambda_0}, & \Theta_{1010} &= \dfrac{H_{1}}{\lambda_0},\\
    \Theta_{0110} &= \Theta_{1010}, & \Lambda_{1001} &= \dfrac{I_{1}}{\lambda_0}, & \Lambda_{0101} &= \Lambda_{1001},
  \end{align*}
  the dynamical system \eqref{eq:ourcase} is transformed into
  \begin{equation}\label{eq:ourcase2}
    \begin{cases}
      \dot \zeta = \lambda_0 \zeta + P^0(\zeta, \bar \zeta, x, \bar x, y, \bar y) + \mathcal{O}_4,\\
      \dot x = \mu^{+}_{1} x + C_{3} x y + P^1(\zeta, \bar \zeta, x, \bar x, y, \bar y) + \mathcal{O}_4,\\
      \dot y = \mu^{+}_{2} y + D_{1} x^{2} + P^2(\zeta, \bar \zeta, x, \bar x, y, \bar y) + \mathcal{O}_4,\\
    \end{cases}
  \end{equation}
  where $P^j(\zeta, \bar \zeta, x, \bar x, y, \bar y)$ $(j = 1,2,3)$ have only cubic terms satisfying
  \begin{align*}
      P^0 (\zeta, \bar \zeta, e^{\mathrm{i} \eta} x, e^{2 \mathrm{i} \eta} y ) &= P^0(\zeta, \bar \zeta, x, y),\\
      P^1 (\zeta, \bar \zeta, e^{\mathrm{i} \eta} x, e^{2 \mathrm{i} \eta} y ) &= e^{\mathrm{i} \eta} P^1 (\zeta, \bar \zeta, x, y),\\
      P^2 (\zeta, \bar \zeta, e^{\mathrm{i} \eta} x, e^{2 \mathrm{i} \eta} y) &= e^{2 \mathrm{i} \eta} P^2 (\zeta, \bar \zeta, x, y).      
  \end{align*}
  \begin{remark}
    The near identity transformation \eqref{eq:trans_quad} preserves the invariance under the mapping $\alpha_j \mapsto e^{\mathrm{i} j \eta} \alpha_j$, that is, \eqref{eq:ourcase2} is also invariant under the mapping $x \mapsto e^{\mathrm{i} \eta} x$ and $y \mapsto e^{2 \mathrm{i} \eta} y$.
  \end{remark}
\end{corollary}

Next, we remove the cubic terms in \eqref{eq:ourcase2} as possible.
Since the procedure is similar to that of Lemma~\ref{lem:nit}, we use the following result without proof.
\begin{lemma}\label{lem:nearidentity-third}
  Consider the system
  \begin{equation}\label{eq:ourcase2-2}
    \begin{cases}
      \dot z = \lambda_0 z + Z_{3000} z^3 + Z_{2100}|z|^{2}z + Z_{1200} |z|^2 \bar z + Z_{0300} \bar z^3 + Z_{1020} z \alpha_1^{2}\\
      \quad  + Z_{0120} \bar z \alpha_1^2  + Z_{1002} z \alpha_2^{2}  + Z_{0102} \bar z \alpha_2^2 + Z_{0021} \alpha_1^2 \alpha_2 + \mathcal{O}_4,\\
      \dot \alpha_1 = \mu^{+}_{1} \alpha_1 + C_{3} \alpha_1 \alpha_2 + X_{2010} z^2 \alpha_1 + X_{1110} |z|^{2} \alpha_1 + X_{0210} \bar z^2 \alpha_1 \\
      \quad + X_{1011} z \alpha_1 \alpha_2 + X_{0111} \bar z \alpha_1 \alpha_2 + X_{0030} \alpha_1^{3} + X_{0012} \alpha_1 \alpha_2^{2} + \mathcal{O}_4,\\
      \dot \alpha_2 = \mu^{+}_{2} \alpha_2 + D_{1} \alpha_1^{2} + Y_{1020} z \alpha_1^2 + Y_{0120} \bar z \alpha_1^2 + Y_{2001} z^2 \alpha_2\\
      \quad + Y_{1101} |z|^{2} \alpha_2 + Y_{0201} \bar z^2 \alpha_2 + Y_{0021} \alpha_1^2 \alpha_2 + Y_{0003} \alpha_2^{3} +  \mathcal{O}_4,
    \end{cases}
  \end{equation}
  whose cubic terms consist of all terms which are invariant under $\alpha_j \mapsto e^{\mathrm{i} j \eta} \alpha_j$.
  The system \eqref{eq:ourcase2-2} can be transformed by an invertible parameter-dependent change of complex coordinate
  \begin{equation}
    \begin{cases}
      z = \zeta + \displaystyle\sum_{0 \le p, q, r, s \le 3 \atop p + q + r + s = 3} \Xi_{p q r s} \zeta^{p} \bar\zeta^{q} x^{r} y^{s},\\
      \alpha_1 = x + \displaystyle\sum_{0 \le p, q, r, s \le 3 \atop p + q + r + s = 3} \Pi_{p q r s} \zeta^{p} \bar\zeta^{q} x^{r} y^{s},\\
      \alpha_2 = y + \displaystyle\sum_{0 \le p, q, r, s \le 3 \atop p + q + r + s = 3} \Upsilon_{p q r s} \zeta^{p} \bar\zeta^{q} x^{r} y^{s},\\
    \end{cases}
  \end{equation}  
  into a dynamical system as follows:
  \begin{equation}\label{eq:tth}
    \begin{cases}
      \dot \zeta = \lambda_0 \zeta + Z_{2100} |\zeta|^{2} \zeta + Z_{1020} \zeta x^{2} + Z_{1002} \zeta y^{2} + \mathcal{O}_{4},\\
      \dot x = \mu_1^+ x + C_{3} x y + X_{1110} |\zeta|^{2} x + X_{0030} x^3 + X_{0012} x y^2 + \mathcal{O}_4,\\
      \dot y = \mu_2^+ y + D_{1} x^2 + Y_{1101} |\zeta|^{2} y + Y_{0021} x^{2} y + Y_{0003} y^{3} + \mathcal{O}_{4}.
    \end{cases}
  \end{equation}
  Here,
  \begin{align*}
    \Xi_{3000} &= \dfrac{Z_{3000}}{2\lambda_0}, & \Xi_{1200} &= \dfrac{Z_{1200}}{2\bar\lambda_0}, & \Xi_{0300} &= \dfrac{Z_{0300}}{3\bar\lambda_0 - \lambda_0},\\
    \Xi_{0120} &= \dfrac{Z_{0120}}{\bar\lambda_0 - \lambda_0 + 2\mu_{1}^{+}}, & \Xi_{0102} &= \dfrac{Z_{0102}}{\bar\lambda_0 - \lambda_0 + 2 \mu_{2}^{+}}, & \Xi_{0021} &= \dfrac{Z_{0021}}{2\mu_{1}^{+} + \mu_{2}^{+} - \lambda_0},\\
    \Pi_{2010} &= \dfrac{X_{2010}}{2\lambda_0}, & \Pi_{0210} &= \dfrac{X_{0210}}{2\bar\lambda_0}, & \Pi_{1011} &= \dfrac{X_{1011}}{\lambda_0 + \mu_{2}^{+}}, \\
    \Pi_{0111} &= \dfrac{X_{0111}}{\bar\lambda_0 + \mu_{2}^{+}}, & \Upsilon_{1020} &= \dfrac{Y_{1020}}{\mu_{2}^{+} - 2\mu_{1}^{+} - \lambda_0}, & \Upsilon_{0120} &= \dfrac{Y_{0120}}{\mu_{2}^{+} - 2\mu_{1}^{+} - \bar\lambda_0},\\
    \Upsilon_{2001} &= \dfrac{Y_{2001}}{2\lambda_0}, & \Upsilon_{0201} &= \dfrac{Y_{0201}}{2\bar\lambda_0}
  \end{align*}
  and the other terms $C_{p q r s}$, $A_{p q r s}$ and $B_{p q r s}$ are $0$.
\end{lemma}

We now combine Corollary~\ref{cor:transform-tth} and Lemma~\ref{lem:nearidentity-third}:
\begin{corollary}
  The system \eqref{eq:ourcase} can be transformed by an invertible parameter-dependent change of complex coordinate, smoothly depending on the parameters,
  \begin{equation}\label{eq:trans_cubic}
    \begin{cases}
      z = \zeta + \Gamma_{2000} \zeta^{2} + \Gamma_{1100}|\zeta|^{2} + \Gamma_{0200} \bar\zeta^{2} + \Gamma_{0020} x^{2} + \Gamma_{0002} y^{2}\\
      \quad + \Xi_{3000} \zeta^3 + \Xi_{1200} |\zeta| \bar \zeta + \Xi_{0300} \bar \zeta^3 + \Xi_{0120} \bar \zeta x^2 + \Xi_{0102} \bar\zeta y^2 + \Xi_{0021} x^2 y,\\
      \alpha_{1} = x + \Theta_{1010} \zeta x + \Theta_{0110} \bar\zeta x + \Pi_{2010} \zeta^2 x + \Pi_{0210} \bar\zeta^2 x\\
      \quad + \Pi_{1011} \zeta x y + \Pi_{0111} \bar \zeta x y,\\
      \alpha_{2} = y + \Lambda_{1001} \zeta y + \Lambda_{0101} \bar\zeta y + \Upsilon_{1020} \zeta x^2 + \Upsilon_{0120} \bar \zeta x^2 + \Upsilon_{2001} \zeta^2 y\\
      \quad + \Upsilon_{0201} \bar \zeta^2 y, 
    \end{cases}
  \end{equation}
  for all sufficiently small $\mu_0,\mu_1^{+}, \mu_2^{+}$, into \eqref{eq:tth}.  
\end{corollary}

Remark that $Z_{p q r s}, X_{p q r s}$ and $Y_{p q r s}$ include the coefficients in $\mathcal{O}_3$ of the inverse transformation of \eqref{eq:trans_quad}.
Therefore it is complicated to calculate the coefficients directly.
In the following, we compute the coefficients of \eqref{eq:tth}.

Substituting \eqref{eq:trans_cubic} into \eqref{eq:ourcase}, we have 
\begin{align*}
  \dot z &= \lambda_0 (\zeta + \Gamma_{2000} \zeta^{2} + \Gamma_{1100}|\zeta|^{2} + \Gamma_{0200} \bar\zeta^{2} + \Gamma_{0020} x^{2} + \Gamma_{0002} y^{2})\\
  &+ E_1 (\zeta + \Gamma_{2000} \zeta^{2} + \Gamma_{1100} |\zeta|^{2} + \Gamma_{0200} \bar\zeta^{2} + \Gamma_{0020} x^{2} + \Gamma_{0002} y^{2})^2\\
  &+ E_2 (\zeta + \Gamma_{2000} \zeta^{2} + \Gamma_{1100} |\zeta|^{2} + \Gamma_{0200} \bar\zeta^{2} + \Gamma_{0020} x^{2} + \Gamma_{0002} y^{2})\\
  &\times (\bar \zeta + \bar \Gamma_{2000} \bar \zeta^{2} + \bar \Gamma_{1100} |\zeta|^{2} + \bar \Gamma_{0200} \zeta^{2} + \bar \Gamma_{0020} x^{2} + \bar \Gamma_{0002} y^{2})\\
  &+ E_3 (\bar \zeta + \bar \Gamma_{2000} \bar \zeta^{2} + \bar \Gamma_{1100}|\zeta|^{2} + \bar \Gamma_{0200} \zeta^{2} + \bar \Gamma_{0020} x^{2} + \bar \Gamma_{0002} y^{2})^2\\
  &+ E_4 (x + \Theta_{1010} \zeta x + \Theta_{0110} \bar\zeta x)^2
  + E_5 (y + \Lambda_{1001} \zeta y + \Lambda_{0101} \bar\zeta y)^2 \\
  &+ e_1 \zeta^3 + (e_2 \zeta + e_3 \bar \zeta) |\zeta|^2 + e_4 \bar \zeta^3 + (e_5 \zeta + e_6 \bar \zeta) x^2 + (e_7 \zeta + e_8 \bar \zeta) y^2\\
  & + e_9 x^2 y 
  + \mathcal{O}_4.
\end{align*}
Then, differentiating \eqref{eq:trans_cubic} and substituting \eqref{eq:tth}, we also have
\begin{align*}
  \dot z 
  &= \lambda_0 \zeta + Z_{2100} |\zeta|^{2} \zeta + Z_{1020} \zeta x^{2}  + Z_{1002} \zeta y^{2} + 2 \Gamma_{2000} \lambda_0 \zeta^2\\
  & + \Gamma_{1100} (\lambda_0 |\zeta|^2 + \bar\lambda_0 |\zeta|^2) + 2 \Gamma_{0200} \bar\lambda_0 \bar\zeta^2 + 2 \Gamma_{0020} (\mu^+_1 x^2 + C_3 x^2 y)\\
  &+ 2 \Gamma_{0002} (\mu^+_2 y^2 + D_1 x^2 y) + \mathcal{O}_4.
\end{align*}
By the coefficient comparison method, we obtain 
\begin{align*}
Z_{2100} &= 2 E_1 \Gamma_{1100} + E_2 (\bar \Gamma_{1100} + \Gamma_{2000}) + 2 E_3 \bar \Gamma_{0200} + e_2,\\
Z_{1020} &= 2 E_1 \Gamma_{0020} + E_2 \bar \Gamma_{0020} + 2 E_4 \Theta_{1010} + e_5,\\
Z_{1002} &= 2 E_1 \Gamma_{0002} + E_2 \bar \Gamma_{0002} + 2 E_5 \Lambda_{1001} + e_7.
\end{align*}
By the same manner, we have
\begin{align*}
X_{1110} &= H_1 (\Theta_{0110} + \Gamma_{1100}) + \bar H_1 (\Theta_{1010} + \bar \Gamma_{1100}) + H_3,\\
X_{0030} &= H_1 \Gamma_{0020} + \bar H_1 \bar \Gamma_{0020} + c_4,\\
X_{0012} &= H_1 \Gamma_{0002} + \bar H_1 \bar \Gamma_{0002} + c_5,\\
Y_{1101} &= I_1 (\Lambda_{0101} + \Gamma_{1100}) + \bar I_1 (\Lambda_{1001} + \bar \Gamma_{1100}) + I_3,\\
Y_{0021} &= I_1 \Gamma_{0020} + \bar I_1 \bar \Gamma_{0020} + d_4,\\
Y_{0003} &= I_1 \Gamma_{0002} + \bar I_1 \bar \Gamma_{0002} + d_5.
\end{align*}

It should be noted that $X_{1110}, X_{0030}, X_{0012}, Y_{1101}, Y_{0021}, Y_{0003} \in \mathbb{R}.$
Replacing $(\zeta, x, y)$ with $(z_{0}, z_{1}, z_{2}) \in \mathbb{C} \times \mathbb{R}^{2}$ and rewriting $Z_{2100}$, $Z_{1020}$, $Z_{1002}$, $X_{1110}$, $X_{0030}$, $X_{0012}$, $Y_{1101}$, $Y_{0021}$, $Y_{0003}$, $C_3$ and $D_1$ to $\tilde a_{0}$, $\tilde a_{1}$, $\tilde a_{2}$, $\tilde b_{0}$, $\tilde b_{1}$, $\tilde b_{2}$, $\tilde c_{0}$, $\tilde c_{1}$, $\tilde c_{2}$, $\tilde B$ and $\tilde C$, respectively, we finally obtain the normal form for the Turing--Turing--Hopf bifurcation with $\mathrm{O}(2)$ symmetry.

\end{appendices}

\bibliography{sn-bibliography}

\end{document}